\documentclass[times]{iapress}

\usepackage{moreverb}

\usepackage[colorlinks,bookmarksopen,bookmarksnumbered,citecolor=red,urlcolor=red]{hyperref} 

\def\volumeyear{2015}
\def\volumenumber{3}
\def\volumemonth{December}
\usepackage{amsmath,amssymb}
\usepackage{graphicx}

\newtheorem{example}{Example}

\newtheorem{theorem}{Theorem}
\newtheorem{teo}{Theorem}

\newtheorem{cor}{Corollary}
\newtheorem{prikl}{Example}
\newtheorem{definition}{Definition}
\newtheorem{thm}{Theorem}[section]
\newtheorem{lem}{Lemma}[section]

\newtheorem{nasl}{Corrolary}[section]
\theoremstyle{Remark}
\newtheorem{zauv}{Remark}[section]

\theoremstyle{definition}
\newtheorem{ozn}{Definition}[section]
\newtheorem{exm}{Example}

\newcommand{\lp}[1]{\left( \begin{array}{#1} }
\newcommand{\rp}{\end{array} \right)}

\newcommand{\be}{\begin{equation}}
\newcommand{\ee}{\end{equation}}

\begin{document}

\runningheads{Minimax-robust estimation problems for stationary stochastic sequences}{Mikhail Moklyachuk}

\title{Minimax-robust estimation problems for stationary stochastic sequences}

\author{ Mikhail Moklyachuk \affil{1}$^,$ }

\address{
\affilnum{1}Department of Probability Theory, Statistics and Actuarial
Mathematics, Taras Shevchenko National University of Kyiv, Kyiv, Ukraine
}

\corraddr{Mikhail Moklyachuk (Email: Moklyachuk@gmail.com). Department of Probability Theory, Statistics and Actuarial
Mathematics, Taras Shevchenko National University of Kyiv, Volodymyrs'ka 64 Str., Kyiv 01601, Ukraine.}

\begin{abstract}
This survey provides an overview of optimal estimation of linear functionals which depend on the unknown values of a stationary stochastic sequence. Based on observations of the sequence without noise as well as observations of the sequence with a stationary noise, estimates could be obtained. Formulas for calculating the spectral characteristics and the mean-square errors of the optimal estimates of functionals are derived in the case of spectral certainty, where spectral densities of the sequences are exactly known. In the case of spectral uncertainty, where spectral densities of the sequences are not known exactly while sets of admissible spectral densities are given, the minimax-robust method of estimation is applied. Formulas that determine the least favourable spectral densities and the minimax spectral characteristics of estimates are presented for some special classes of admissible spectral densities.
\end{abstract}

\keywords{Stochastic sequence, minimax-robust estimate, mean square error, least favourable spectral
density, minimax-robust spectral characteristic}

\maketitle
\noindent{\bf AMS 2010 subject classifications.} Primary: 60G10, 60G25, 60G35, Secondary: 62M20, 62P20,
93E10, 93E11

\vskip 8pt%
\noindent{\bf DOI:} 10.19139/soic.v3i4.173

\tableofcontents

\section{Introduction}

Theory of estimation of the unknown values of stationary stochastic processes based on a set of observations of the processes plays an important role in many practical applications. The development of the theory started from the classical works of Kolmogorov \cite{Kolmogorov} and Wiener \cite{Wiener}, in which they presented methods of solution of the extrapolation and interpolation problems for stationary processes.
The interpolation problem considered by Kolmogorov means estimation of the missed values of stochastic sequences. The prediction problem consists in estimation the future values of the process based on observations of the process in the past. The third classical problem is filtering of stochastic processes which consists in estimation the original values of the signal process from observations of the process with noise. All these problems for stationary sequences and processes are clearly described in books by Gikhman and Skorokhod \cite{Gihman}, Hannan \cite{Hannan},
Rozanov \cite{Rozanov}, Yaglom \cite{Yaglom1,Yaglom2}.

One of the fields of practical applications of the stationary and related stochastic sequences is economical modelling and financial time series. Most simple examples of stationary linear models are moving average (MA) sequences, autoregressive (AR) and autoregressive-moving average (ARMA) sequences, state space models, all of which refer to stationary sequences with rational spectral function without unit AR-roots.  The main results concerning the model description, parameter estimation, forecasting and further investigations are described in the classical book by Box, Jenkins and Reinsel \cite{Box:Jenkins}.
Most of results which have appeared since that time were based on the assumption that the spectral structure of the stationary sequence is known.

Traditional methods of solution of linear extrapolation,
interpolation and filtering problems for stationary stochastic sequences may be employed under the condition that
spectral densities of the sequences are known exactly
(see, for example, selected works of Kolmogorov \cite{Kolmogorov}, survey article by Kailath \cite{Kailath},
 books by Rozanov \cite{Rozanov}, Wiener \cite{Wiener}, Yaglom \cite{Yaglom1,Yaglom2}).
In practice, however,
complete information on the spectral densities is impossible in most cases.
To solve the problem the parametric or nonparametric estimates
of the unknown spectral densities are found or these densities are selected
by other reasoning. Then the classical estimation method is applied
provided that the estimated or selected densities are the true one. This
procedure can result in a significant increasing of the value of the error as
Vastola and Poor~\cite{Vastola} have demonstrated with the help of
some examples. This is a reason to search estimates which are optimal
for all densities from a certain class of the admissible spectral densities.
These estimates are called minimax since they minimize the maximal value of
the error. Many investigators have been interested
in minimax extrapolation, interpolation and filtering problems for
stationary stochastic sequences. See, for example, the papers by Hosoya \cite{Hosoya} and Taniguchi \cite{Taniguchi1981}.
A survey of results in minimax (robust) methods of data
processing can be found in the paper by Kassam and Poor \cite{Kassam}.
The paper by Grenander~\cite{Grenander} should be marked as the first one where the
minimax approach to extrapolation problem for stationary processes was
proposed.
Franke~\cite{Franke1984,Franke1985},  Franke and Poor \cite{FrankePoor}
investigated the minimax extrapolation and
filtering problems for stationary sequences with the help of convex optimization methods. This approach makes it possible to find equations that
determine the
least favourable spectral densities for various classes of admissible densities.

In the works by   Moklyachuk \cite{Moklyachuk:1981a} -- \cite{Moklyachuk:2008} problems of extrapolation, interpolation and filtering
for stationary processes and sequences were studied.
The corresponding problems for vector-valued stationary sequences and processes were investigated by Moklyachuk and Masyutka~\cite{Moklyachuk:2006} -- \cite{Moklyachuk:2012}.  In the articles by Dubovets'ka and Moklyachuk \cite{Dubovetska1} - \cite{Dubovetska8}  and in the book  by Golichenko and Moklyachuk \cite{Golichenko} the minimax estimation problems were investigated for another generalization of stationary processes --  periodically correlated stochastic sequences and processes.
Luz and Moklyachuk \cite{Luz1} -- \cite{Luz8}, \cite{Luz6}  investigated the classical and minimax extrapolation, interpolation and filtering problems for sequences and processes with $n$th stationary increments.

In the proposed paper we deal with the problem of estimation of the functional $A\xi=\sum_{k=0}^{\infty}a(k)\xi(k)$ which depends on the unknown values
of a stationary stochastic sequence $\xi(k)$ based on observations of the sequence $\xi(k)+\eta(k)$, where $ \eta(k)$ is an uncorrelated with the sequence $\xi(k)$
stationary stochastic sequence.
Formulas for calculating the mean-square errors and the
spectral characteristics of the optimal estimate of the functional
are derived in the case of spectral certainty, where the spectral
densities of the sequences $\xi(n)$ and $\eta(n)$ are exactly known.
The estimation problem is solved in
the case of spectral uncertainty, where the spectral densities of
sequences are not exactly known, but, instead, a set of admissible spectral densities is
given. The minimax-robust estimation method is applied in this case. Formulas that determine the least favourable spectral
densities and the minimax-robust spectral characteristic of the
optimal linear estimates of the functional $A\xi$ are derived in the case
of spectral uncertainty for some concrete classes of admissible spectral
densities.

The paper is organized as follows.
The spectral representations of stationary stochastic sequences and their correlation functions are described in Section 2.
Factorizations of spectral densities with relations to one-sided moving average representations and applications to the problem of prediction of stationary sequences are discussed in this section.
In section 3 the problem of the mean square optimal linear estimation of the functionals
which depend on the unknown `future' values of a stationary stochastic sequence from a class of stationary stochastic sequences is investigated.
Estimates are based on observations of the sequence in the `past'.
Following the Grenander~\cite{Grenander} approach to investigation the problem
of optimal linear estimation of the linear functional we consider the problem as a two-person zero-sum game. It is show that this game has equilibrium
point.
In Section~4 the extrapolation problem for functionals of stationary sequences is investigated in the case of spectral certainty, where the spectral density of the stationary stochastic sequence  is exactly known, as well as in the case of spectral uncertainty, where the spectral density of the sequence is not exactly known, but a class of admissible spectral densities is given.
The classical Hilbert space projection method and the minimax-robust procedure to extrapolation of the functional are applied.
The corresponding results of investigation of the extrapolation problem for functionals of stationary sequences from observations with noise are presented in Section 5.
In Section 6 and Section 7 results of investigation of the interpolation  problem for functionals of stationary sequences from observations without noise and with noise are presented.
The problem is investigated in the case of spectral certainty as well as in the case of spectral uncertainty.
Formulas are proposed for calculation the value of the mean square errors and the spectral characteristics of the interpolation in the case of spectral certainty.
Relations which determine the least favourable spectral densities and the minimax spectral characteristics are proposed in the case of spectral uncertainty for some special classes of spectral densities.

\section{Stationary stochastic sequences. Spectral representation}\label{sss:spectral}

\def\cov{{\text{cov}}}

In this section properties of stationary stochastic sequences illustrated with the help of some examples are presented.
The spectral representation of stationary stochastic sequences and their correlation functions are described.
Basic properties of linear transformations of stationary stochastic sequences are proposed.
The Wold representation of stationary stochastic sequences is studied with application to the problem of prediction of stationary sequences.

\subsection{Definitions. Examples}

Let
$H=L_{2}(\Omega,{\mathcal F}, P)$ be the space of all complex valued random variables $\xi=\xi(\omega),\omega\in\Omega,$ such that ${E}|\xi|^{2} < \infty $.
Covariation of the random variables $\xi,\eta\in H$ is defined as the quantity
\begin{equation}
\label{eq3} \text{cov}\,(\xi ,\eta ) = E(\xi - E\xi )\overline{(\eta - E\eta )} .
\end{equation}

\begin{ozn} A sequence of complex valued random variables
$\xi (n)$, $n \in \mathbb Z$, with values in the space
 $H=L_{2}(\Omega,{\mathcal F}, P)$ is called {\it wide sense stationary} if for all $n\in\mathbb  Z$ and $k\in \mathbb Z$ the following relations hold true
\[
E\xi(n) = E\xi(0)=C=\text{const},
\]
\begin{equation}
\label{eq5}\cov(\xi(k + n), \xi (k)) =
\cov(\xi(n),\xi(0)).
\end{equation}
\end{ozn}

In the following we will assume that $E\xi(n) = 0$, $n\in\mathbb  Z$. This assumption gives us a possibility
to consider the covariation as the inner product in the Hilbert space  $H=L_{2}(\Omega,{\mathcal F}, P)$ and use the powerful methods and results of the Hilbert space theory.

\begin{ozn} The function
\begin{equation}
\label{eq6} R(n) = \cov(\xi(n),\xi(0)), \,\,\,n \in \mathbb Z,
\end{equation}
is called the {\it correlation function} of the wide sense stationary
sequence $\xi(n)$.
\end{ozn}

\begin{example}
White noise sequence. The white noise sequence $\varepsilon = \{\varepsilon(n),n \in \mathbb Z\}$ is a sequence of orthogonal random variables such that
 $E\varepsilon (n)
= 0,E\varepsilon (i) \overline {\varepsilon (j)} = \delta _{ij} $, where
$\delta_{ij}$ is the Kroneker symbol. This sequence is stationary with the correlation function
\begin{equation}
\label{eq16} R(n) = \int\limits_{ - \pi }^\pi {e^{i\lambda
n}dF(\lambda )} ,
\end{equation}
\begin{equation}
\label{eq17} F(\lambda ) = \frac{1}{2\pi }\int\limits_{ - \pi
}^\lambda f(v)dv,\quad f(v ) = 1,\quad - \pi  < v < \pi .
\end{equation}

\noindent The white noise sequence has the absolutely continuous spectral function $F(\lambda)$ with the constant spectral density function
 $f(v ) = 1 $.
\end{example}

\begin{example} Moving average sequence. Let $\varepsilon = \{\varepsilon(n),n \in \mathbb Z\}$ be a white noise sequence. Consider the sequence
\begin{equation}
\label{eq18} \xi (n) = \sum\limits_{k = - \infty }^\infty {a(k)
\varepsilon (n - k) ,}
\end{equation}

\noindent where $a(k)$ are complex valued numbers such that $\sum_{k =
- \infty }^\infty {\left| {a(k)} \right|^2 < \infty } $. By the Parseval equality
\[
\cov(\xi (n + m) ,\xi(m) ) = \cov(\xi(n) ,\xi(0) ) =
\sum\limits_{k = - \infty }^\infty {a(n + k) \overline {a(k)} }.
\]
For this reason $\xi  = \{\xi (n),n \in \mathbb Z\}$   is a stationary sequence, which is called (two-sided) moving average (MA) sequence generated by
the white noise sequence $\varepsilon =\{\varepsilon (n),n \in \mathbb Z\}$.

\noindent In the case where  $a(k)=0$ for all $k=-1,-2,\dots$, the sequence
\begin{equation}
\label{eq19} \xi (n) = \sum\limits_{k = 0}^\infty a(k) \varepsilon(n - k),
\end{equation}
is called one-sided moving average (MA) sequence generated by
the white noise sequence $\varepsilon =\{\varepsilon (n),n \in \mathbb Z\}$.

\noindent In the case where $a(k)=0$ for all $k=-1,-2,\dots$ and all $k=p+1,p+2,\dots$, the sequence
\begin{equation}
\label{eq20} \xi (n) = a(0) \varepsilon (n) + a(1) \varepsilon(n -1) + \cdots + a(p) \varepsilon(n - p),
\end{equation}
is called moving average $MA(p)$ sequence of order $p$  generated by
the white noise sequence $\varepsilon =\{\varepsilon (n),n \in \mathbb Z\}$.

\noindent The correlation function $R(n)$ of the moving average sequence $MA(p)$  is of the form
\[
R(n) = \frac{1}{2\pi }\int\limits_{ - \pi }^\pi {e^{i\lambda
n}f(\lambda )d\lambda },
\]
\noindent with the spectral density function
\begin{equation}
\label{eq21} f(\lambda ) = \left| {P(e^{ - i\lambda })} \right|^2,
\end{equation}
where
\[
P(z) = a(0) + a(1) z +\cdots+ a(p) z^p.
\]
is a polynomial of order $p$.
\end{example}

\begin{example}
 The autoregressive sequence $AR(q)$. Let $\varepsilon =\{\varepsilon (n),n \in \mathbb Z\}$ be a white noise sequence. A sequence
 $\xi=\{\xi (n),n \in \mathbb Z\}$ is called the autoregressive sequence of order $q$ if it satisfies the equation
\begin{equation}
\label{eq22} \xi(n) + b(1) \xi(n - 1) + \cdots + b(q) \xi (n - q)= \varepsilon(n).
\end{equation}
In the case where all zeros of the polynomial
\begin{equation}
\label{eq25} Q(z)=1+ b_1 z + \cdots + b_q z^q
\end{equation}

\noindent are outside the unit disk $D=\{z:|z|<1\}$, equation
(\ref{eq22}) has a stationary solution which can be represented as one-sided moving average sequence.
In this case the correlation function $R(n)$ is of the form
\begin{equation}
\label{eq26} R(n) = \frac{1}{2\pi }\int\limits_{ - \pi }^\pi
{e^{i\lambda n}f(\lambda )d\lambda },
\end{equation}
\[f(\lambda ) =
\frac{1}{\left| {Q(e^{ - i\lambda })} \right|^2}.
\]
\end{example}

\begin{example}
 The autoregressive and moving average $ARMA(p,q)$ sequence. Suppose that in the right-hand side of the equation (\ref{eq22}) instead of
$\varepsilon (n)$ we put the sequence $a_0 \varepsilon (n) + a(1)\varepsilon(n - 1) +\cdots+ a(p) \varepsilon(n - p) $. We get the so called
 autoregressive and moving average $ARMA(p,q)$ sequence of order $(p,q)$
\begin{equation}
\label{eq30} \xi (n) + b(1) \xi(n - 1) + \cdots + b(q) \xi(n - q)= a(0) \varepsilon (n) + a(1) \varepsilon (n - 1) + \cdots + a(p)\varepsilon (n - p).
\end{equation}
Under the same conditions on zeros of the polynomial
$Q(z)$ as in the previous example, equation (\ref{eq30})  has a stationary solution $\xi = \{\xi(n),n \in \mathbb Z\}$.
The correlation function $R(n)$ of the sequence $\xi = \{\xi(n),n \in \mathbb Z\}$ is of the form
\[R(n) =  \frac{1}{2\pi } \int\limits_{ - \pi }^\pi {e^{i\lambda
n}f(\lambda )d\lambda },
\]
\begin{equation}
\label{eq31}
 f(\lambda ) = \left|
{\frac{P(e^{ - i\lambda })}{Q(e^{ - i\lambda })}} \right|^2.
\end{equation}
\end{example}

\subsection{Spectral representation of stationary sequences}

\begin{teo} Herglotz theorem.
 Let  $R(n)$  be the covariance function of a (wide sense) stationary
stochastic sequence $\xi=\{\xi(n), n \in \mathbb Z\}$  with zero mean value. Then on the space
 $\left([-\pi,\pi),{\mathcal B}([-\pi,\pi))\right)$ there is
a finite measure $F=F(B)$,
$B\in {\cal B}([-\pi,\pi))$, such that for every $n\in \mathbb Z$
\begin{equation}
\label{eq32} R(n) = \int\limits_{ - \pi }^\pi {e^{i\lambda
n}F(d\lambda ),}
\end{equation}

\noindent where $\int\limits_{ - \pi }^\pi {e^{i\lambda
n}F(d\lambda )} $ is the Lebesgue-Stieltjes integral.
\end{teo}

\begin{zauv}
The measure $F = F(B)$ in the representation (\ref{eq32}) is called the spectral
measure, and the function $F(\lambda ) = F([-\pi,\lambda ])$ is called the spectral function of the stationary
sequence with the covariance function $R(n)$.
\end{zauv}

\begin{zauv}
In the case where $\sum_{n=-\infty}^{\infty}\vert R(n)\vert<\infty$ the spectral function $F(\lambda)$ has the spectral density
$f(\lambda)$, which is determined by the formula
\[
f(\lambda)=\sum_{n=-\infty}^{\infty}e^{-in\lambda}R(n).
\]
\end{zauv}

The following theorem provides the corresponding spectral representation
of the stationary stochastic sequence $\xi=\{\xi(n), n \in \mathbb Z\}$ itself.

\begin{teo} Let $\xi=\{\xi(n), n \in \mathbb Z\}$ be a stationary  stochastic  sequence with zero mean value
$E\xi(n) = 0$, $n \in\mathbb Z$. There is an orthogonal stochastic measure $Z=Z(\Delta)$, $\Delta \in {\cal B}([-\pi,\pi))$, such that for every
$n \in \mathbb Z$ the following spectral representation is valid
\begin{equation}
\label{eq48} \xi(n) = \int\limits_{ - \pi }^\pi {e^{i\lambda
n}Z(d\lambda )\left( { = \int\limits_{\left[ { - \pi ,\pi }
\right)} {e^{i\lambda n}Z(d\lambda )} } \right)} .
\end{equation}
Moreover, $E\left| {Z(\Delta )} \right|^2 = F(\Delta )$.
\end{teo}

Let $\xi=\{\xi(n), n \in \mathbb Z\}$ be a stationary  stochastic  sequence with the spectral representation (\ref{eq48}).
Let $H (\xi )$ be the closed in the mean-square sense linear manifold spanned by the random variables $\xi(n)$, $n\in \mathbb Z$.
Let $\eta \in H(\xi )$. The following theorem describes the structure of such
random variables.

\begin{teo} For any $\eta  \in H(\xi )$ there is a function $\varphi  \in L^2(F)$ such that
\begin{equation}
\label{eq60} \eta = \int\limits_{ - \pi }^\pi {\varphi (\lambda
)Z(d\lambda )}.
\end{equation}
\end{teo}

\begin{zauv} Let $H_0 (\xi )$ and $L^2_0 (F)$ be the corresponding closed linear manifolds
spanned by the variables $\{\xi (n), n=0,-1,-2,\dots\}$ and by the functions $\{e_n=e^{in\lambda}, n=0,-1,-2,\dots\}$.
Then for $\eta  \in H_0 (\xi )$
there is a function $\varphi \in L^2_0 (F)$  such that
\[\eta = \int_{ - \pi
}^\pi {\varphi (\lambda )Z(d\lambda )}.
\]
\end{zauv}

\subsection{Linear filters of stationary sequences}

\noindent Consider a special but important class of linear transformations that are defined
by means of the so called linear filters. Suppose that, at instant of time
$m$, a system (filter) receives as input a signal $x(m)$, and that the output of the
system is, at instant of time $n$, the signal $h(n - m)x(m) $, where $h = h(m)$, $m \in \mathbb Z$,  is a complex valued function called the {\it impulse response function} of the filter.
The total signal $y(n)$ obtained from the input $x(m)$, $m \in \mathbb Z$, can be represented in the form
\begin{equation}
\label{eq64} y(n) = \sum\limits_{m = - \infty }^\infty {h(n - m)x(m)
}.
\end{equation}

\noindent For physically realizable systems values of the output at instant of time $n$
are determined only by the "past" values of the signal, i.e. the values $x(m)$ at instants $m \le n$.
It is therefore natural to call a filter with the impulse response function $h = h(m)$
physically realizable (causal filter), if $h(m) = 0$ for all $m < 0$, in other words if
\begin{equation}
\label{eq65}
y(n) = \sum\limits_{m =-\infty}^n h(n - m)x(m) =
\sum\limits_{m = 0}^\infty h(m)x(n - m).
\end{equation}

\noindent An important spectral characteristic of a filter with the impulse response function $h = h(m)$
is its Fourier transform
\begin{equation}
\label{eq66}
\varphi(\lambda)=\sum\limits_{m=-\infty}^\infty{e^{-im\lambda}h(m)},
\end{equation}

\noindent which is called the {\it spectral characteristic} of the filter.

\noindent Let us now describe conditions under which the series in (\ref{eq64}) and (\ref{eq66}) convergence. Let us suppose that the input
is a stationary  stochastic  sequence $\xi = \{\xi (n), n \in \mathbb Z\}$ with the covariance function
$R(n)$ and spectral decomposition (\ref{eq48}). If the following condition holds true
\begin{equation}
\label{eq67} \sum\limits_{k,l = - \infty }^\infty {h(k)R(k -
l)\overline {h(l)} < \infty},
\end{equation}

\noindent then the series $\sum_{m = - \infty }^\infty {h(n - m)\xi(m) } $ converges in the mean-square sense and therefore there
is a stationary sequence $\eta = \{\eta (n), n \in \mathbb Z\}$ such that
\begin{equation}
\label{eq68} \eta (n) = \sum\limits_{m = - \infty }^\infty {h(n -
m)\xi(m) } = \sum\limits_{m = - \infty }^\infty {h(m)\xi ({n - m})
}.
\end{equation}

\noindent In terms of the spectral measure condition (\ref{eq67}) is equivalent to the condition that the spectral characteristic of the filter
$\varphi (\lambda ) \in L^2(F)$, that is
\begin{equation}
\label{eq69} \int\limits_{ - \pi }^\pi {\left| {\varphi (\lambda
)} \right|^2F(d\lambda ) < \infty}.
\end{equation}

\noindent Under the condition (\ref{eq67}) or condition (\ref{eq69}) from (\ref{eq68})
and (\ref{eq48}) we find the spectral representation of the sequence
$\eta = \{\eta (n), n \in \mathbb Z\}$:
\begin{equation}
\label{eq70} \eta(n) = \int\limits_{ - \pi }^\pi {e^{i\lambda
n}\varphi (\lambda )Z(d\lambda )} .
\end{equation}

\noindent The correlation function  $R_\eta (n)$ of the sequence $\eta = \{\eta (n), n \in \mathbb Z\}$
is determined by the formula
\begin{equation}
\label{eq71} R_\eta (n)=\int\limits_{ - \pi }^\pi {e^{i\lambda
n}\left| {\varphi (\lambda )} \right|^2F(d\lambda )} .
\end{equation}

\noindent In particular, if the input to a filter with the spectral characteristic $\varphi = \varphi (\lambda )$
is taken to be white noise $\varepsilon =\{\varepsilon(n)\}, n \in \mathbb Z\}$, then the output will be a moving average stationary sequence
\begin{equation}
\label{eq72} \eta(n) = \sum\limits_{m = - \infty }^\infty
{h(m)\varepsilon (n - m) }
\end{equation}

\noindent with the spectral density
\[
f_\eta (\lambda ) = \left| {\varphi (\lambda )} \right|^2.
\]

\noindent The following theorem shows that every stationary  stochastic
sequence with a spectral density can be represented as a moving average sequence.

\begin{teo}
Let $\eta = \{\eta (n), n \in \mathbb Z\}$ be a stationary  stochastic sequence with the spectral density $f_\eta (\lambda )$.
Then (possibly on the extended probability space) we
can find a white noise sequence $\varepsilon = \{\varepsilon(n), n \in \mathbb Z\}$ and a filter such that the
representation (\ref{eq72}) is valid.
\end{teo}

\begin{proof}
 For a given (nonnegative) spectral density function $f_\eta (\lambda )$ we can find a function ${\varphi (\lambda )}$
such that
\[
f_\eta (\lambda ) =\left| {\varphi (\lambda )} \right|^2.
\]
\noindent Since
\[\int_{-\pi}^{\pi}f_\eta (\lambda)d\lambda<\infty,\]
we have ${\varphi (\lambda )}\in L^2(\lambda)$, where $\lambda$ is the Lebesgue measure on $[-\pi,\pi]$.

\noindent Hence the function ${\varphi(\lambda )}$ can be represented as a
Fourier series (\ref{eq66}) with
\[
h(m)=\frac{1}{2\pi}\int_{-\pi}^{\pi}e^{im\lambda}{\varphi (\lambda
)}d\lambda,
\]

\noindent  where convergence is understood in the sense that
\[
\int_{-\pi}^{\pi}\left\vert{\varphi (\lambda )}-\sum_{|m|\leq n}
e^{-im\lambda}{h (m)}\right\vert^2d\lambda\to0,\quad n\to\infty.
\]
Let
\[
\eta(n)=\int_{-\pi}^{\pi}e^{in\lambda}{Z(d\lambda )},\quad
n\in\mathbb Z.
\]

\noindent In addition to the random measure $Z=Z(\Delta)$ we introduce another (independent) orthogonal stochastic measure $\tilde{Z}=\tilde{Z}(\Delta)$ with
$M\vert\tilde{Z}(a,b]\vert^2=\frac{b-a}{2\pi}$ and construct the measure
\[
\bar{Z}(\Delta)=\int\limits_{\Delta}{\varphi^{ \oplus} (\lambda
)}{Z(d\lambda )}+ \int\limits_{\Delta}[1-{\varphi^{\oplus}(\lambda
)}{\varphi (\lambda )}]{\tilde{Z}(d\lambda )},
\]
where $a^{ \oplus}$ is the pseudoinverse operation
\[
a^{ \oplus}= \left\{ {\begin{array}{rl} a^{-1},& a\not=0,
\\
0,&a=0. \\
\end{array}} \right.
\]

\noindent The stochastic measure $\bar{Z}=\bar{Z}(\Delta)$ is a measure with orthogonal values, and
for every  $\Delta=(a,b]$ we have
\[
E\vert\bar{Z}(\Delta)\vert^2=\frac{1}{2\pi
}\int\limits_{\Delta}\vert{\varphi^{ \oplus} (\lambda
)}\vert^2\vert{\varphi(\lambda)}\vert^2d\lambda+ \frac{1}{2\pi
}\int\limits_{\Delta}\vert1-{\varphi^{\oplus}(\lambda
)}{\varphi(\lambda)}\vert^2d\lambda=\frac{\vert\Delta\vert}{2\pi},
\]

\noindent where $\vert\Delta\vert=b-a$. Therefore the stationary sequence
\[
\varepsilon(n)=\int_{-\pi}^{\pi}e^{in\lambda}{\bar{Z}(d\lambda
)},\quad n\in\mathbb Z,
\]
is a white noise sequence.
Note that
\begin{equation}
\label{eq73}
\int_{-\pi}^{\pi}e^{in\lambda}{\varphi(\lambda)}{\bar{Z}(d\lambda)}
= \int_{-\pi}^{\pi}e^{in\lambda}{Z}(d\lambda)=\eta(n),
\end{equation}

\noindent and, from the other hand,
\[
\int_{-\pi}^{\pi}e^{in\lambda}{\varphi(\lambda)}{\bar{Z}(d\lambda)}=
\int_{-\pi}^{\pi}e^{in\lambda}\left(
\sum\limits_{m=-\infty}^{\infty}{e^{-im\lambda}h(m)}\right)
{\bar{Z}(d\lambda)}=
\]
\[
=\sum\limits_{m=-\infty}^{\infty}h(m)\int_{-\pi}^{\pi}e^{i(n-m)\lambda}
{\bar{Z}(d\lambda)}=\sum\limits_{m=-\infty}^{\infty}h(m)\varepsilon(n-m).
\]

\noindent This equality, together with (\ref{eq73}), establishes the representation (\ref{eq72}).
This completes the proof of the theorem.
\end{proof}

\begin{nasl}
Let the spectral density function $f_\eta (\lambda ) > 0$ (almost everywhere with respect to the Lebesgue measure) and let the following factorization of the function $f_\eta (\lambda )>0$ holds true
\begin{equation}
\label{eq74-1} f_\eta (\lambda ) = \left| {\varphi (\lambda )}
\right|^2,
\end{equation}
\begin{equation}
\label{eq74-2} \varphi (\lambda ) = \sum\limits_{k = 0}^\infty
{e^{ - i\lambda k}h(k)}, \quad \sum\limits_{k = 0}^\infty {\left|
{h(k)} \right|^2 < \infty }.
\end{equation}

\noindent Then the sequence $\eta=\{\eta(n), n \in \mathbb Z\}$ admits the representation as one-sided moving average sequence
\[
\eta(n) = \sum\limits_{m = 0}^\infty {h(m)\varepsilon(n - m) } .
\]
\noindent
In particular, if $P(z) = a(0) + a(1) z + \cdots + a(p) z^p$  is a polynomial that has
no zeros on $\{z:\left| z \right| =1\}$, then the sequence $\eta = \{\eta (n), n \in \mathbb Z\}$ with the spectral density
\[
f_\eta (\lambda ) =  \left| {P(e^{ - i\lambda })} \right|^2
\]
\noindent can be represented in the form
\[
\eta(n) = a(0) \varepsilon (n) + a(1) \varepsilon (n - 1) + \cdots +a(p) \varepsilon(n - p),
\]
\noindent where  $\varepsilon = \{\varepsilon(n), n \in \mathbb Z\}$ is a white noise sequence.
\end{nasl}

\begin{nasl} Let $\xi=\{\xi(n), n \in \mathbb Z\}$ be a sequence with rational spectral density
\begin{equation}
\label{eq74} f (\lambda ) = \left|\frac{ {P(e^{ - i\lambda
})} } { {Q(e^{ - i\lambda })}} \right|^2,
\end{equation}

\noindent where polynomials $P(z) = a(0) + a(1) z + \cdots + a(p) z^p$ and $Q(z)=1+b(1) z + \cdots + b(q) z^q$ have no zeros on
$\{z:\left| z \right| = 1\}$. Then there is a white noise sequence
 $\varepsilon = \{\varepsilon(n), n \in \mathbb Z\}$ such that the following equation holds true
\begin{equation}
\label{eq75} \xi(n)+b(1)\xi(n-1)+\cdots+b(q)\xi(n-q) = a(0)\varepsilon(n) + a(1) \varepsilon(n - 1) + \cdots + a(p)\varepsilon(n - p).
\end{equation}

\noindent Conversely, every stationary stochastic sequence $\xi=\{\xi(n), n \in \mathbb Z\}$ that satisfies this equation
with a white noise $\varepsilon = \{\varepsilon(n), n \in \mathbb Z\}$  and polynomial $Q(z)=1+b(1) z + \cdots + b(q) z^q$  with no zeros on
$\{z:\left| z \right| = 1\}$ has the spectral density (\ref{eq74}).
\end{nasl}

Let us describe conditions under which a function $f(\lambda)$ admits the representation (\ref{eq74-1}), (\ref{eq74-2}).

\noindent First, denote by
$H_2$ the set of all analytic in the unit disk $D=\{z:|z|<1\}$ functions $f(z)$ such that
\[|f(z)|^2=\lim_{r\uparrow1}\int_{-\pi}^{\pi}|f(re^{i\theta})|^2d\theta<\infty.
\]
For the function \[f(z)=\sum_{n=0}^{\infty}a_nz^n\]
we have
\[f(re^{i\theta})=\sum_{n=0}^{\infty}a_nr^ne^{in\theta}.\]
It means that $a_nr^n$ are the Fourier coefficients of the function $f(re^{i\theta})$.
From the Parseval equality it follows that
\[
\int_{-\pi}^{\pi}|f(re^{i\theta})|^2d\theta
=2\pi\sum_{n=0}^{\infty}|a_n|^2r^{2n}.
\]
We can conclude that $f(z)\in H_2$ if and only if
\[
\sum_{n=0}^{\infty}|a_n|^2<\infty.
\]
For any function  $f(z)\in H_2$ we can determine a series
\[f(e^{i\theta})=\sum_{n=0}^{\infty}a_ne^{in\theta},\]
which converges in $L_2[-\pi,\pi]$. The function $f(z), |z|<1$, restores from the function $f(e^{i\theta})$ with the help of the Poisson formula
\[
f(re^{i\theta})=
\frac{1}{2\pi}\int_{-\pi}^{\pi}f(e^{iu})P(r,\theta,u)du,
\]
where
\[
P(r,\theta,u)=\frac{1-r^2}{1-2r\cos(\theta-u)+r^2}=
\sum_{n=-\infty}^{\infty}r^{|n|}e^{in(\theta-u)}.
\]

\begin{teo} Let $f (\lambda)$ be a nonnegative and integrable with respect to the Lebesgue measure on
$[-\pi,\pi)$ function. In order that a function $g(z)\in H_2$ exists such that
\[
f (\lambda)=|g(e^{i\lambda})|^2,
\]
it is necessary and sufficient that the following condition holds
\begin{equation}
\label{eq74-3} \int_{-\pi}^{\pi} |\ln{f(\lambda)}|d\lambda<\infty.
\end{equation}
\end{teo}

\begin{proof}
Let condition (\ref{eq74-3}) holds true. Then the function
\[
v(r,\theta)=
\frac{1}{2\pi}\int_{-\pi}^{\pi}\ln{f(\lambda)}P(r,\theta,\lambda)d\lambda
\]
is harmonic in the unit disk $D=\{z:|z|<1\}$. From the Jensen inequality it follows that
\[
v(r,\theta)\leq\ln\left\{
\frac{1}{2\pi}\int_{-\pi}^{\pi}{f(\lambda)}P(r,\theta,\lambda)d\lambda\right\}.
\]
Denote by $\varphi(z)$ the analytic in the unit disk $D=\{z:|z|<1\}$ function with the real part $v(r,\theta)$. Consider the function $g(z)=e^{\varphi(z)/2}$. For this function
\[
|g(re^{i\theta})|^2= e^{\text{Re}\varphi(z)}=e^{v(r,\theta)}\leq
\frac{1}{2\pi}\int_{-\pi}^{\pi}{f(\lambda)}P(r,\theta,\lambda)d\lambda,
\]
\[
\int_{-\pi}^{\pi}|g(re^{i\theta})|^2d\theta\leq
\int_{-\pi}^{\pi}f(\lambda)d\lambda.
\]
So, the function $g(z)\in H_2$ and
\[
\lim_{r\uparrow1}|g(re^{i\theta})|^2=
e^{\lim_{r\uparrow1}v(r,\theta)}=f(\theta)
\]
almost everywhere.
\end{proof}

\begin{nasl} In order that a stochastic sequence $\eta= \{\eta (n ), n \in \mathbb Z\}$
 admits the representation as one-sided moving average sequence
\[
\eta(n) = \sum\limits_{m = 0}^\infty {a(m)\varepsilon(n - m) },\quad \sum_{m=0}^{\infty}|a(m)|^2<\infty.
\]
\noindent where  $\varepsilon = \{\varepsilon (n ), n \in \mathbb Z\}$ is a white noise sequence,
it is necessary and sufficient that the sequence $\eta= \{\eta (n ), n \in \mathbb Z\}$
has an absolutely continuous spectral measure and its spectral density $f(\lambda)$ satisfies the condition
\begin{equation}
\label{eq74-4} \int_{-\pi}^{\pi} \ln{f(\lambda)}d\lambda>-\infty.
\end{equation}
\end{nasl}

Determined in the theorem analytical in the unit disk $D=\{z:|z|<1\}$ function
$\varphi(z)$
with the real part $v(r,\theta)$ has the boundary value
$\ln{f(\lambda)}$. By the Schwartz formula we have
\begin{equation}
\label{eq74-5} \varphi(z)=  \frac{1}{2\pi}\int_{-\pi}^{\pi}
\ln{f(\lambda)}\frac{e^{i\lambda}+z}{e^{i\lambda}-z}d\lambda.
\end{equation}
It follows from the decomposition of the function $g(z)=\exp{\{\frac{\varphi(z)}{2}\}}$ into a power series $g(z)=\sum_{n = 0}^\infty b_nz^n$ that the coefficients $a_n=\sqrt{2\pi}b_n$.
 From the other hand, the function $g(z)$ can be represented in a different form. Since
\[
\frac{e^{i\lambda}+z}{e^{i\lambda}-z}=1+
\frac{2ze^{-i\lambda}}{1-ze^{-i\lambda}}=1+2\sum\limits_{k =
1}^\infty z^ke^{-ik\lambda},
\]
we have
\[
\overline{g(z)}= \exp\left\{ \frac{1}{4\pi}\int_{-\pi}^{\pi}
\ln{f(\lambda)}d\lambda+ \frac{1}{2\pi}\sum\limits_{k = 1}^\infty
d_k\overline{z}^k \right\},
\]
\[ d_k=\int_{-\pi}^{\pi} e^{ik\lambda}\ln{f(\lambda)}d\lambda.
\]
If we introduce the notation
\[
P=\exp\left\{ \frac{1}{4\pi}\int_{-\pi}^{\pi}
\ln{f(\lambda)}d\lambda\right\},
\]
\[
\exp\left\{ \frac{1}{2\pi}\sum\limits_{k = 1}^\infty d_k{z}^k
\right\}=\sum\limits_{n = 0}^\infty c_n{z}^n,\quad c_0=1,
\]
we will have
\[
\overline{g(z)}=P\sum\limits_{n = 0}^\infty c_n\overline{z}^n.
\]
So \[a_n=\sqrt{2\pi}Pc_n.\]

\subsection{Wold expansion of stationary sequences}

In contrast to the representation (\ref{eq48}) which gives an expansion of a
stationary sequence in the frequency domain, the Wold \cite{Wold1938,Wold1938} expansion operates
in the time domain. The main point of this expansion is that a stationary
sequence $\xi(n)$ can be represented as a sum of two stationary
sequences, one of which is completely predictable (in the sense that its
values are completely determined by its "past"), whereas the second does
not have this property.

First we give some definitions. Let
 $H(\xi)$ and $H_{n}(\xi)$
be closed linear manifolds, spanned  by all values of the stationary sequence $\{\xi(k),k\in\mathbb Z\}$ and values  $\{\xi(k), k=n,n-1,n-2,\dots\}$ respectively.
Let
\[
S(\xi ) = {\mathop { \cap} \limits_{n}} H_{n} (\xi ).
\]

\noindent For every element $\eta  \in H(\xi )$ denote by
\[
\hat {\eta}_{n} = \text{Proj}\, (\eta \vert H_{n} (\xi ))
\]

\noindent the projection of the element $\eta $ on the subspace $H_{n} (\xi )$.
Denote also
\[
\hat {\eta} _{S} = \text{Proj}\,(\eta \vert S(\xi )),
\]

\noindent Every element $\eta\in H(\xi )$ can be represented in the form
\[
\eta=\hat {\eta} _{S}+(\eta - \hat {\eta} _{S}),
\]

\noindent where
$\eta-\hat {\eta} _{S}\bot\hat {\eta} _{S}$.
Therefore the space  $H(\xi )$ can be represented as the orthogonal sum
\[
H(\xi ) = S(\xi ) \oplus R(\xi ),
\]

\noindent where $S(\xi )$ consists of the elements $\hat {\eta} _{S}$ with $\eta  \in H(\xi )$, and $R(\xi )$ consists of the elements of the form
 $\eta - \hat {\eta} _{S}$.

\begin{ozn}
A stationary sequence $\xi (n)$ is called
{\it regular}, if
\[
H(\xi ) = R(\xi ),
\]
\noindent and {\it singular}, if
\[
H(\xi ) = S(\xi ).
\]
\end{ozn}

\begin{zauv}
Singular sequence is also called {\it deterministic}, and regular
sequence is called {\it purely} or {\it completely nondeterministic}. If
$S(\xi )$ is a proper subspace of the space $H(\xi )$, then the sequence $\xi (n)$ is called {\it nondeterministic}.
\end{zauv}

\begin{teo}
Every wide sense stationary random sequence
$\xi (n)$ has a unique decomposition
\begin{equation}
\label{eq119} \xi (n) = \xi^r (n) + \xi^s (n),
\end{equation}
\noindent where $\xi^r (n)$ is regular sequence and $\xi^s (n) $ is singular sequence. The sequences $\xi^r (n)$ and nd $\xi^s (n) $ are orthogonal.
\end{teo}

\begin{proof}
Define
\[
\xi^s (n) = \text{Proj}\,(\xi (n) \vert S(\xi )),\quad
\xi^r (n)= \xi (n)-\xi^s (n).
\]

\noindent Since $\xi ^{r}(n) \bot S(\xi )$ for every $n$, we have $S(\xi ^{r}) \bot S(\xi )$.
On the other hand, $S(\xi ^{r})\subseteq S(\xi )$ and therefore $S(\xi ^{r})$ is trivial (contains only random sequences that
coincide almost surely with zero). Consequently $\xi ^{r}(n)$ is regular.

\noindent Moreover, $H_{n} (\xi ) \subseteq H_{n} (\xi ^{s}) \oplus H_{n} (\xi ^{r})$ and $H_{n} (\xi ^{s}) \subseteq H_{n} (\xi )$,
$H_{n} (\xi ^{r}) \subseteq H_{n} (\xi )$.
Therefore $H_{n} (\xi ) =H_{n} (\xi ^{s}) \oplus H_{n} (\xi ^{r})$, and hence
for every $n $
\begin{equation}
\label{eq120} S(\xi ) \subseteq H_{n} (\xi ^{s}) \oplus H_{n} (\xi
^{r}).
\end{equation}

\noindent Since $\xi _{n}^{r} \bot S(\xi )$, it follows from (\ref{eq120})  that
\[
S(\xi ) \subseteq H_{n} (\xi ^{s}),
\]

\noindent and therefore $S(\xi ) \subseteq S(\xi ^{s}) \subseteq H(\xi ^{s})$. But $\xi _{n}^{s} \in S(\xi )$, hence $H(\xi ^{s})\subseteq S(\xi )$, and
consequently
\[
S(\xi ) = S(\xi ^{s}) = H(\xi ^{s}),
\]
which means that $\xi ^{s}(n)$ is singular.

\noindent  The orthogonality of $\xi ^{s}(n)$ and $\xi ^{r}(n)$ follows in an obvious way from $\xi^{s}(n) \in S(\xi )$ and $\xi^{r}(n) \bot S(\xi )$.
\end{proof}

\begin{ozn}
Let $\xi (n)$ be a nondegenerate stationary sequence. A random sequence $\varepsilon(n)$ is called the {\it  innovation sequence} for $\xi(n) $,
if the following conditions holds true:

\noindent 1) $\varepsilon(n)$ consists of pairwise orthogonal random variables with $E\varepsilon(n) =
0$, $E{\left| {\varepsilon(n)} \right|}^{2} = 1$;

\noindent 2) $H_{n} (\xi ) = H_{n} (\varepsilon )$ for all
$n \in \mathbb Z$.
\end{ozn}

\begin{zauv}
 The reason for the term ``innovation'' is that $\varepsilon(n+1)$ provides new "information", not contained in $H_n(\xi )$,
  that is needed for forming $H_{n+1}(\xi )$.
\end{zauv}

The following fundamental theorem establishes a connection between
one-sided moving average sequences and regular sequences.

\begin{teo}\label{teo121}
The necessary and sufficient condition for a nondegenerate stationary sequence
$\xi(n) $ to be regular is that there exist an innovation sequence $\varepsilon = \{\varepsilon (n)\}$ and a sequence
 of complex numbers $\{a(n), n \ge 0\}$, with ${\sum\limits_{n = 0}^{\infty} {{\left|{a(n)} \right|}^{2}}}  < \infty,$ such that
\begin{equation}
\label{eq121} \xi (n) = {\sum\limits_{k = 0}^{\infty}  {a(k)
\varepsilon (n - k)}}.
\end{equation}
\end{teo}

\begin{proof}
Necessity. Represent $H_{n} (\xi )$ in the form
\[
H_{n} (\xi ) = H_{n - 1} (\xi ) \oplus B_{n}(\xi),
\]

\noindent where $B_{n}(\xi)$ is the space of random variables of the form
$\beta \cdot \xi (n)$, where $\beta$ is a complex number.
Since $H_{n} (\xi )$ is generated by elements from $H_{n - 1} (\xi)$ and elements of the form $\beta \cdot \xi (n) $, the dimension (dim)
of the space $B_{n}(\xi)$ is either zero or one. But the space
$H_{n} (\xi )$ cannot coincide with $H_{n - 1} (\xi )$ for any  value of $n$.
In fact, if for some $n$ the space $B_{n}(\xi)$ is trivial, then by stationarity $B_{n}(\xi)$ is trivial for all
$n$, and therefore $H(\xi ) = S(\xi)$, contradicting the assumption that the sequence $\xi(n)$ is regular.
Thus the space $B_{n}(\xi)$ has the dimension $1$. Let $\eta(n) $ be a nonzero element of $B_{n}(\xi)$. Take
\[
\varepsilon_{n} = {\frac{{\eta (n)}} {{{\left\| {\eta (n)}
\right\|}}}},
\]

\noindent where ${\left\| {\eta (n)}  \right\|}^{2} = E{\left| {\eta(n)}  \right|}^{2}
> 0$. For fixed $n$ and $k \ge 0$ consider the decomposition
\[
H_{n} (\xi ) = H_{n - k} (\xi ) \oplus B_{n - k + 1}(\xi)
\oplus\cdots\oplus B_{n}(\xi).
\]

\noindent The elements $\varepsilon (n-k),\dots,\varepsilon(n) $
is an orthogonal basis in $B_{n - k + 1}(\xi)\oplus\cdots\oplus B_{n}(\xi)$ and
\begin{equation}
\label{eq122} \xi(n)=\sum\limits_{j = 0}^{k-1}a(j)\varepsilon(n-j)+\text{Proj}\,({\xi}(n)\vert H_{n-k}(\xi)),\quad a(j) = E\xi(n) \overline{\varepsilon(n - j)}.
\end{equation}

\noindent By the Bessel inequality
\[
{\sum\limits_{j = 0}^{\infty}  {{\left| {a(j)}  \right|}^{2}}}
\le {\left\| {\xi (n)}  \right\|}^{2} < \infty .
\]

\noindent The series
\[
{\sum_{j = 0}^{\infty} {a(j) \varepsilon(n -j)}}
\]
converges in the mean square, and then, by (\ref{eq122}),
relation (\ref{eq121}) will be established as soon as we show that
\[
\text{Proj}\,({\xi}(n)\vert H_{n-k}(\xi)){\buildrel {L^{2}} \over
\longrightarrow} 0, \quad k \to \infty.\]

\noindent Consider the case $n = 0$. Denote
$\hat{\xi} _{i}=
\text{Proj}\,({\xi}(0)\vert H_{i}(\xi))$.
 Since
\[
\hat {\xi} _{ - k} = \hat {\xi} _{0} + {\sum\limits_{i = 1}^{k}
{[\hat {\xi }_{ - i} - \hat {\xi} _{ - i + 1} ]}} ,
\]

\noindent and the terms that appear in this sum are orthogonal, we have for every $k \ge 0$
\[
{\sum\limits_{i = 1}^{k} {{\left\| {\hat {\xi} _{ - i} - \hat
{\xi} _{ - i + 1}}  \right\|}^{2}}}  = {\left\| {{\sum\limits_{i =1}^{k} {(\hat {\xi} _{ - i} - \hat {\xi} _{ - i + 1} )}}}
\right\|}^{2} =
\]
\[
={\left\| {\hat {\xi} _{ - k} - \hat {\xi} _{0}}  \right\|}^{2}
\le 4{\left\| {\xi (0)}  \right\|}^{2} < \infty .
\]

\noindent Therefore the limit
${\mathop {\lim }\limits_{k \to \infty}}  \hat {\xi} _{ - k} $ exists (in the mean square sense).
For each $k$ the value $\hat {\xi }_{ - k} \in H_{ - k} (\xi )$, and therefore the limit in question must belong to the space ${\mathop {
\cap} \limits_{k \ge 0}} H_{ - k} (\xi ) = S(\xi )$. But, by assumption, $S(\xi )$ is trivial, and therefore
\[
\text{Proj}\,({\xi}(n)\vert H_{n-k}(\xi)){\buildrel {L^{2}} \over
\longrightarrow} 0, \quad k \to \infty.\]

Sufficiency. Let the nondegenerate stationary sequence $\xi(n)$ have a representation (\ref{eq121}),
where $\varepsilon = \{\varepsilon(n)\}$ is an orthonormal system (not necessarily satisfying the condition $H_{n} (\xi ) = H_{n} (\varepsilon )$, $n \in\mathbb Z)$.
Then $H_{n} (\xi ) \subseteq H_{n} (\varepsilon )$ and therefore
$S(\xi ) = {\mathop { \cap} \limits_{k}} H_{k} (\xi )
\subseteq H_{n} (\varepsilon )$
for every $n$. But $\varepsilon _{n + 1} \bot H_{n} (\varepsilon )$, and therefore $\varepsilon _{n + 1} \bot S(\xi )$
and at the same time $\varepsilon =\{\varepsilon(n)\}$ is a basis in $H(\xi )$. lt follows that $S(\xi )$ is trivial,
and consequently $\xi(n)$ is regular.
This completes the proof of the theorem.
\end{proof}

\begin{zauv}
It follows from the proof of the theorem \ref{teo121} that a nondegenerate sequence $\xi(n)$ is regular
if and only if it admits a representation as one-sided moving average
\begin{equation}
\label{eq123} \xi (n) = {\sum\limits_{k = 0}^{\infty}
{\tilde{a}_{k} \tilde{\varepsilon}_{n - k}}}
\end{equation}
where $\tilde{\varepsilon}=\{\tilde{\varepsilon}(n)\}$ is an orthonormal system which (it is important to emphasize
this !) does not necessarily satisfy the condition $H_n(\xi)=H_n(\tilde{\varepsilon})$, $n\in\mathbb Z$. In this
sense the conclusion of the theorem \ref{teo121} says more, and specifically that for a
regular stationary sequence $\xi(n)$  there exist a sequence of numbers $a=\{a(n)\}$ and an orthonormal system of random variables $\varepsilon=\{{\varepsilon}(n)\}$,
such that not only (\ref{eq123}), but also (\ref{eq121}), is satisfied, with $H_n(\xi)=H_n({\varepsilon})$, $n\in\mathbb Z$.
\end{zauv}

\begin{teo}
Wold expansion. A nondegenerate stationary sequence $\xi = \{\xi (n)\}$ can be represented in the form
\begin{equation}
\label{eq124} \xi (n) = \xi^s(n)+{\sum\limits_{k = 0}^{\infty}
{a(k) \varepsilon(n - k)} },
\end{equation}

\noindent where
\[{\sum_{k = 0}^{\infty}  {{\left| {a(k)}
\right|}^{2}}}  < \infty,\]
and $\varepsilon = \{\varepsilon(n)\}$ is the innovation sequence (for $\xi^{r}(n))$.
\end{teo}

 The significance of the concepts of regular and singular sequences becomes clear if we consider the following (linear)
extrapolation problem, for whose solution the Wold expansion (\ref{eq124}) is
especially useful.

Let $H_{0} (\xi )$ be the closed linear manifold spanned by values of the stationary sequence $\{\xi(k), k=0,-1,-2,\dots\}$.
 Consider the problem of constructing an optimal (least-squares) linear estimator $\hat {\xi}(n) $ of the value $\xi (n) $ of the sequence at point $n>0$ based on observations of the ``past'' $\{\xi(k), k=0,-1,-2,\dots\}$.
The estimate $\hat {\xi}(n) $ is a projection of $\xi (n) $ on the manifold $H_{0} (\xi )$:
\begin{equation}
\label{eq125} \hat {\xi} (n) = \text{Proj}\,({\xi}(n)\vert H_{0}(\xi)).
\end{equation}

Since the sequences $\xi ^{s}(n)$ and $\xi ^{r}(n)$ are orthogonal and
$H_{0} (\xi ) \subseteq H_{0} (\xi ^{s}) \oplus H_{0} (\xi ^{r})$, we obtain, by using
(\ref{eq124})

\[
\begin{array}{rl}
 \hat {\xi} (n) = &\text{Proj}\,(\xi^{s}(n) + \xi^{r}(n) \vert H_{0} (\xi ))
= \text{Proj}\,(\xi^{s}(n) \vert H_{0} (\xi )) + \text{Proj}\,(\xi^{r}(n) \vert H_{0} (\xi )) = \\
 = &\text{Proj}\,(\xi^{s}(n) \vert H_{0} (\xi^{s}) \oplus H_{0} (\xi ^{r})) +
\text{Proj}\,(\xi _{n}^{r} \vert H_{0} (\xi ^{s}) \oplus H_{0} (\xi ^{r})) = \\
 = &\text{Proj}\,(\xi^{s}(n) \vert H_{0} (\xi ^{s})) + \text{Proj}\,(\xi^{r}(n)
\vert H_{0} (\xi ^{r})) = \\
 = &\xi^{s}(n) + \text{Proj}\,({\sum\limits_{k = 0}^{\infty}  {a(k) \varepsilon(n - k)}}  \vert H_{0} (\xi ^{r})). \\
 \end{array}
\]

In (\ref{eq124}) the sequence $\varepsilon =\{\varepsilon (n)\}$ is an innovation sequence for $\xi ^{r} = \{\xi^{r}(n)\}$ and $H_{0} (\xi ^{r}) = H_{0} (\varepsilon )$.
 Therefore
\begin{equation}
\label{eq126} \hat {\xi} (n) = \xi^{s}(n) + \text{Proj}\,\left({\sum\limits_{k = 0}^{\infty} {a(k) \varepsilon (n - k)}}
\vert H_{0} (\varepsilon )\right) = \xi^{s}(n) + {\sum\limits_{k =
n}^{\infty}  {a(k) \varepsilon (n - k)}}
\end{equation}

\noindent and the mean-square error of prediction $\xi (n) $
based on observations $\{\xi(k), k=0,-1,-2,\dots\}$ equals
\begin{equation}
\label{eq127} \sigma _{n}^{2} = E{\left| {\xi (n) - \hat {\xi}
(n)}  \right|}^{2} = {\sum\limits_{k = 0}^{n - 1} {{\left|
{a(k)}  \right|}^{2}}}.
\end{equation}

\noindent Two important conclusions follows from the presented results.

1) If the sequence $\xi(n)$  is singular, then for every $n \ge 1$ the error of the extrapolation
$\sigma _{n}^{2} $ is zero; in other words, we can predict $\xi (n) $ without error from its "past"
$\{\xi(k), k=0,-1,-2,\dots\}$.

2) If the sequence $\xi(n)$  is regular, then $\sigma _{n}^{2} \le
\sigma _{n + 1}^{2} $ and
\begin{equation}
\label{eq128} {\mathop {\lim} \limits_{n \to \infty}}  \sigma
_{n}^{2} = {\sum\limits_{k = 0}^{\infty}  {{\left| {a(k)}
\right|}^{2}}}.
\end{equation}

\noindent Since
\[
{\sum\limits_{k = 0}^{\infty}  {{\left| {a(k)}  \right|}^{2}}}  =
E{\left| {\xi(n)}  \right|}^{2},
\]

\noindent it follows from (\ref{eq128}) and (\ref{eq127}) that if $n$
increases, the prediction of $\xi (n) $ based on observations $\{\xi(k), k=0,-1,-2,\dots\}$ becomes
trivial (reducing simply to $E\xi (n) = 0)$.

\subsection{Conclusions}

In this section we describe properties of stationary stochastic sequence illustrated with the help of some examples.
The spectral representation of stationary stochastic sequences and their correlation functions are described.
Properties of linear transformations of stationary stochastic sequences are described.
The Wold representation of stationary stochastic sequences is analysed. Application to the problem of prediction of stationary sequences is described.

For the detailed exposition of results of the theory of stationary stochastic sequences see books by Gikhman and Skorokhod \cite{Gihman}, Hannan\cite{Hannan},
Rozanov\cite{Rozanov}, Yaglom \cite{Yaglom1,Yaglom2}.

 \section{Estimates for functionals of stationary sequences}\label{est1}

In this section we deal with the problem of the mean square optimal linear estimation of the functionals
 \[A_ {N}  { \xi} = \sum_ {j = 0} ^ {N}  {a} (j) { \xi} (j),  \]
 \[A { \xi} = \sum_ {j = 0} ^ { \infty}  {a} (j)  { \xi} (j) \]
 which depend on the unknown values of a stationary stochastic sequence $ { \xi} (j)$ from the class $ \Xi $
of stationary stochastic sequences satisfying the conditions
\begin {equation} \label {GrindEQ__3_1_}
E  { \xi} (j)=  {0}, \quad E | { \xi} (j) |^{2}  \le P.
\end {equation}
Estimates are based on results of observations of the sequence $ { \xi} (j) $ at points of time $ j =-1,-2,\dots$.

Following the Ulf~Grenander~\cite{Grenander} approach to investigation the problem
of optimal linear estimation of the functional $A_ {N}  { \xi} $ (as well estimation of the functional $A{\xi}$) we consider the problem as a two-person zero-sum game in which the first
player chooses a stationary stochastic sequence $ \xi (j) $ from the class $ \Xi $ of
stationary stochastic sequences such
that the value of the mean square error of estimate of the
functional attains its maximum. The second player is looking for
an estimate of the linear functional which minimizes the value of the
mean square error. It is show that this game has equilibrium
point.
The maximum error gives a moving average stationary sequence which is
least favourable in the given class of stationary sequences.
 The greatest value
of the error and the least favourable sequence are determined by the
largest eigenvalue and the corresponding eigenvector of the
operator determined by the coefficients $ a (j) $ which determine the functionals.

\subsection{The maximum value of the mean-square error of estimate of the functional $A_{N} {\xi}$}

In this subsection the maximum value of the mean-square error of the optimal linear estimate of the functional $A_ {N} { \xi} $ which depend on the unknown values of a stationary stochastic sequence $ { \xi} (j)$ from the class $ \Xi $ is found.

\noindent Let $ \Delta ( \xi , \hat{A}_{N} )=E \left|A_{N} { \xi}- \hat{A}_{N} { \xi} \right|^{2}$ denotes the mean-square  error of an estimate
$ \hat{A}_{N} { \xi}$ of the functional $A_{N} { \xi}$. Denote by $ \Lambda $ the class of all linear estimates of the functional $A_{N}{ \xi}$.
\medskip

\begin{theorem} \label{theo.1.1}
The function $ \Delta ( \xi , \hat{A}_{N} )$ has a saddle point on the set $ \Xi \times \Lambda $. The following relation holds true
 \[ \mathop{ \min} \limits_{ \hat{A}_{N} \in \Lambda} \mathop{ \max} \limits_{ \xi \in \Xi} \Delta ( \xi , \hat{A}_{N} )= \mathop{ \max} \limits_{ \xi \in \Xi} \mathop{ \min} \limits_{ \hat{A}_{N} \in \Lambda} \Delta ( \xi , \hat{A}_{N} )=P \nu_{N}^{2} . \]
\noindent The least favourable in the class $ \Xi $ of
stationary stochastic sequences satisfying conditions \eqref{GrindEQ__3_1_} for the optimal linear estimation of the functional
$A_{N} { \xi}$  is one-sided moving average
sequence of order $N$, which is determined by the formula
 \[ { \xi}(j)= \sum_{u=j-N}^{j} \varphi (j-u) { \varepsilon}(u). \]
Here $ \nu_{N}^{2} $ is the largest eigenvalue of the selfadjoint
compact operator $Q_{N}$ in the space ${\mathbb C}^{(N+1)}$,
determined by the matrix $Q_{N} = \left \{Q_{N} (p,q) \right \}_{p,q=0}^{N}$ that is constructed with the elements
 \[Q_{N} (p,q) = \sum_{u=0}^{ \min (N-p,N-q)} {a}(p+u)\, \overline{a}(q+u), \quad p,q=0,1,\dots,N,\]
$ { \varepsilon}(u) $ is a  stationary stochastic sequence with orthogonal values:
\[E { \varepsilon}(i) \overline{ \varepsilon}(j)= \delta_{ij}, \]
where  $ \delta_{ij}$ is the Kronecker symbol;
$\{ \varphi (u), u= {0,\dots,N}\}$ is determined by the eigenvector that corresponds
to $ \nu_{N}^{2} $, and the condition $ \left \| { \xi}(j) \right \| ^{2} =P.$
\end{theorem}

 \textit{Proof.} Lower bound. Denote by $ \Xi_{R} $ the class of all regular  stationary sequences which satisfy conditions \eqref{GrindEQ__3_1_}. Since $ \Xi_{R} \subset \Xi $, we have
 \begin{equation} \label{GrindEQ__3_2_}
 \mathop{ \max} \limits_{ \xi \in \Xi} \mathop{ \min} \limits_{ \hat{A}_{N} \in \Lambda} \Delta ( \xi , \hat{A}_{N} ) \ge \mathop{ \max} \limits_{ \xi \in \Xi_{R}} \mathop{ \min} \limits_{ \hat{A}_{N} \in \Xi_{R}} \Delta ( \xi , \hat{A}_{N} ).
 \end{equation}
A regular stationary sequence \label{regposl} ${\xi}(j)$ admits the
canonical representation as an  one-sided moving average sequence
 \begin{equation} \label{GrindEQ__3_3_}
{ \xi}(j)= \sum_{u=- \infty}^{j} \varphi (j-u) { \varepsilon}(u) ,
 \end{equation}
where ${ \varepsilon}(u) $
is a standard stationary stochastic sequence with
orthogonal values, $ \{\varphi (u): u=0, 1, \ldots \}$ are coefficients of the canonical
representation \eqref{GrindEQ__3_3_}.  The
sequence $ { \xi}(j) \in \Xi_{R} $ is determined if there are
determined coefficients $ \left \{  \varphi (u): u=0, 1, \ldots \right \}$ such that
\[
 E| { \xi}(j) | ^{2} =   \left| \sum_{u=-
\infty}^{j}  \varphi (j-u) \varepsilon (u) \right|^{2}=
\]
 \[ =  \sum_{u,v=- \infty}^{j}  \varphi (j-u) \overline{ \varphi (j-v)}E \varepsilon(u) \overline{ \varepsilon(v)}= \]
 \begin{equation} \label{GrindEQ__3_4_}
 =  \sum_{u=- \infty}^{j}  \left| \varphi (j-u) \right|^{2}= \sum_{u=0}^{ \infty}   \left| \varphi (u) \right|^{2}\le P.
 \end{equation}
The value of the mean-square error $E \left|A_{N} { \xi}- \hat{A}_{N} { \xi} \right|^{2} $ attains its minimum if we choose an estimate $ \hat{A}_{N} { \xi}$ of the form
 \[ \hat{A}_{N} { \xi}= \sum_{j=0}^{N} {a}(j) \hat{ \xi}(j), \]
where $ \hat{ { \xi}}(j)$ is the optimal estimate of the value
of $ { \xi}(j)$ based on observations of the sequence $ {
\xi}(p)$ at points $p=-1,-2,\dots$. Taking into consideration the canonical representation
\eqref{GrindEQ__3_3_} of the regular sequence and the form of the
optimal estimates of its values
 \begin{equation} \label{GrindEQ__3_5_}
 \hat{ { \xi}}(j)= \sum_{u=- \infty}^{-1} \varphi (j-u){\varepsilon}(u),
 \end{equation}
we can write
 \[ \mathop{ \min} \limits_{ \hat{A}_{N} \in \Lambda} E \left|A_{N} { \xi}- \hat{A}_{N} { \xi} \right|^{2} = \]
 \[= \sum_{i,j=0}^{N} a (i) \overline{a(j)} \sum_{u=0}^{i}  \sum_{v=0}^{j} \;  \varphi (i-u) \overline{ \varphi (j-v)}\,E \varepsilon (u) \overline{ \varepsilon (v)} \]
 \[= \sum_{i,j=0}^{N} \; a (i) \overline{a (j)} \; \sum_{u=0}^{ \min (i,j)}  \varphi (i-u) \overline{ \varphi (j-u)}  = \]
 \begin{equation} \label{GrindEQ__3_6_}
= \sum_{i,j=0}^{N} a(i) \overline{a(j)}R(i,j),
 \end{equation}
where
 \[R(i,j)= \sum_{u=0}^{ \min (i,j)}  \varphi (i-u) \overline{ \varphi (j-u)}.
 \]
The change of variables $p=i-u, \; q=j-u$, gives us a possibility to write \eqref{GrindEQ__3_6_} in a different form
 \begin{equation} \label{GrindEQ__3_7_}
 \mathop{ \min} \limits_{ \hat{A}_{N} \in \Lambda} E \left|A_{N} { \xi}- \hat{A}_{N} { \xi} \right|^{2} = \sum_{p,q=0}^{N}  \varphi (p) \overline{ \varphi (q)} \; Q_N (p,q),
 \end{equation}

\noindent where
 \begin{equation} \label{GrindEQ__3_8_}
Q_N (p,q)= \sum_{u=0}^{ \min (N-p,N-q)}a(p+u) \overline{a(q+u)}  .
 \end{equation}

 \noindent Denote by $Q_{N}$ the operator in the space ${\mathbb C}^{(N+1)}$,
determined by the matrix
$ \left \{Q_{N} (p,q) \right \}_{p,q=0}^{N} $. The operator $Q_{N} $ is
selfadjoint (its matrix is Hermitian) bounded operator. It can be
represented in the form $Q_{N} =A_{N} \cdot A_{N}^{*}$, where the
operator $A_{N}$ is determined by the matrix  $ \left \{A_{N} (p,q) \right \}_{p,q=0}^{N} $ with
 \[
 A_{N} (p,q)= \left \{
 \begin{array}{cc} {a}(p+q),& \; p+q \le N, \\ {0},& \; p+q>N.
 \end{array}
 \right.
 \]

 \noindent For this reason the operator $Q_{N} $ has nonnegative eigenvalues \cite{Akhieser}, \cite{Riesz}.
 It follows from \eqref{GrindEQ__3_7_}, that $ \varphi (p)=0$ for $p>N$. Denote by
 \[
 \vec{{\varphi}}=\{\tilde{ \varphi}(0), \tilde{ \varphi}(1),\dots,\tilde{ \varphi}(N)\},\quad
 \tilde{ \varphi}(p)=P^{-1/2} \varphi (p). \]

 \noindent Condition \eqref{GrindEQ__3_4_} with this notation is of
the form
 \begin{equation} \label{GrindEQ__3_9_}
 \left \| \vec{ \varphi} \right \|^{2}= \sum_{p=0}^{N} | \tilde{ \varphi}(p) |^{2} \le 1 ,
 \end{equation}
where $ \left \| \vec{ \varphi} \right \| $ is the norm in the space
${\mathbb C}^{(N+1)}$. Taking into consideration
\eqref{GrindEQ__3_7_} and
 \eqref{GrindEQ__3_9_}, we can write
 \[ \mathop{ \min} \limits_{ \hat{A}_{N} \in \Lambda} \Delta ( \xi , \hat{A}_{N} )=P \left \langle Q_{N} \vec{ \varphi}, \vec{ \varphi} \right \rangle , \]
where $ \left \langle \cdot , \cdot \right \rangle $ is the inner
product in the space ${\mathbb C}^{(N+1)}$.

Taking into account \eqref{GrindEQ__3_2_}, we  get the following
bound from below for the maximum value of the error \cite{Dunford}, \cite{Gould},
 \cite{Riesz}
 \begin{equation} \label{GrindEQ__3_10_}
 \mathop{ \max} \limits_{ \xi \in \Xi} \mathop{ \min} \limits_{ \hat{A}_{N} \in \Lambda} \Delta ( \xi , \hat{A}_{N} ) \ge P \mathop{ \max} \limits_{ \left \| \vec{ \varphi} \right \| \le 1} \left \langle Q_{N} \vec{ \varphi}, \vec{ \varphi} \right \rangle =P \nu_{N}^{2} ,
 \end{equation}
where $ \nu_{N}^{2} $ is the greatest eigenvalue of the operator
$Q_{N}$.

Upper bound. To find the upper bound of the minimax values of the error we use the inequality
 \begin{equation} \label{GrindEQ__3_11_}
 \mathop{ \min} \limits_{ \hat{A}_{N} \in \Lambda} \mathop{ \max} \limits_{ \xi \in \Xi} \Delta ( \xi , \hat{A}_{N} ) \le \mathop{ \min} \limits_{ \hat{A}_{N} \in \Lambda_{1}} \mathop{ \max} \limits_{ \xi \in \Xi} \Delta ( \xi , \hat{A}_{N} ),
 \end{equation}
where $ \Lambda_{1}$ is the class of all linear estimates of the functional $A_{N} { \xi}$, which have the form
 \begin{equation} \label{GrindEQ__3_12_}
 \hat{A}_{N} { \xi}= \sum_{j=- \infty}^{-1} {c}(j) { \xi}(j).
 \end{equation}
Here $ {c}(j)$ are complex-valued coefficients such that
 \[ \sum_{j=- \infty}^{-1} | {c}(j) |  ^{2} < \infty . \]
Taking into consideration the spectral representations of the
stationary stochastic sequence $\xi(j)$ and the correlation
function of the sequence $\xi(j)$ , we can write
 \[ \Delta ( \xi , \hat{A}_{N} )=E \left|A_{N} { \xi}- \hat{A}_{N} { \xi} \right|^{2} =E \left| \sum_{j=0}^{N} {a}(j){ \xi}(j)- \sum_{j=- \infty}^{-1} {c}(j){ \xi}(j) \right|^{2} = \]
 \[= \int_{- \pi}^{ \pi} |A_{N} (e^{i \lambda} )-C(e^{i \lambda} ) |^{2} F_{\xi}(d \lambda ), \]
where
 \[A_{N} (e^{i \lambda} )= \sum_{j=0}^{N} {a}(j)e^{ij \lambda}, \quad C(e^{i \lambda} )= \sum_{j=- \infty}^{-1} {c}(j)e^{ij \lambda}  . \]
Here $F_{\xi}(d \lambda )$ is the spectral measure \label{spektrmira}
of the stationary stochastic sequence.
Restriction \eqref{GrindEQ__3_1_} is equivalent to the restriction
 \begin{equation} \label{GrindEQ__3_14_}
 \mathop{\max} \limits_{ \xi \in \Xi}\int_{- \pi}^{ \pi}  \, F_{\xi}(d \lambda ) \le P.
 \end{equation}
So we can  write
 \[ \mathop{ \max} \limits_{ \xi \in \Xi} \Delta ( \xi , \hat{A}_{N} )=
\mathop{ \max} \limits_{ \xi \in \Xi} \int_{- \pi}^{ \pi} |A_{N} (e^{i \lambda} )-C(e^{i \lambda} ) |^{2} F_{\xi}(d \lambda) \le \]
  \[
  \le
 \left(\mathop{ \max} \limits_{ \lambda \in [- \pi , \pi ]}
|A_{N} (e^{i \lambda} )-C(e^{i \lambda} ) |^{2}\right)
\mathop{\max} \limits_{ \xi \in \Xi} \int_{- \pi}^{ \pi}  F_{\xi}(d
\lambda )
\]
 \[
  \le
 P \mathop{ \max} \limits_{ \lambda \in [- \pi , \pi ]}
|A_{N} (e^{i \lambda} )-C(e^{i \lambda} ) |^{2}.
\]

\noindent
To calculate
 \[ \mathop{ \max} \limits_{ \lambda \in [- \pi , \pi ]} | A_{N} (e^{i \lambda} )-C(e^{i \lambda} ) | ^{2} \]
consider the class of all  power series
 \[ {f}(z)= \sum_{n=0}^{ \infty} { \alpha}(n)z^{n}  , \]
which are regular in the region $ \left|z \right|<1$ and have fixed first
$N+1$ coefficients
\[ { \alpha}(n)={d}(n),\;n=0,1,\dots,N.
\]
Denote by $ \rho_{N}^{2} $ the greatest eigenvalue of the matrix
 \[H= \left \{H(p,q) \right \}_{p,q=0}^{N}, \]
 \[ H(p,q)= \sum_{j=0}^{ \min (p,q)} {d}(p-j) \overline{d}(q-j), \quad
p,q= 0,1,\dots,N . \]
It follows from the properties of the power series, that \cite{GrenanderS}
 \[ \mathop{ \min} \limits_{ \left \{ { \alpha}(n):n \ge N+1 \right \}} \mathop{ \max} \limits_{ \left|z \right|=1} \left | {f}(z) \right | ^{2} = \rho_{N}^{2} . \]
Since in our case
 \[ {d}(p)= {a}(N-p), \; p= 0,1,\dots,N,
  \]
we have to determine the greatest eigenvalue of the
matrix
 \[G_{N}= \left \{G_{N}(p,q) \right \}_{p,q=0}^{N},
  \]
 \[ G_{N}(p,q)= \sum_{u=0}^{ \min
(p,q)} {a}(N-p+u) \overline{a}(N-q+u). \]

\noindent Denote this greatest eigenvalue by $ \omega_{N}^{2} $. With this
notations we have
 \[ \mathop{ \min} \limits_{ \hat{A}_{N} \in \Lambda_{1}} \mathop{ \max} \limits_{ \xi \in \Xi} \Delta ( \xi , \hat{A}_{N} ) \le P \omega_{N}^{2}. \]
Taking into account \eqref{GrindEQ__3_11_}, we get
 \begin{equation} \label{GrindEQ__3_15_}
 \mathop{ \min} \limits_{ \hat{A}_{N} \in \Lambda} \mathop{ \max} \limits_{ \xi \in \Xi} \Delta ( \xi , \hat{A}_{N} ) \le P \omega_{N}^{2}.
 \end{equation}
Note, that
 \[G_{N} (N-p,N-q)=Q_{N} (p,q). \]
For this reason $ \omega_{N}^{2} = \nu_{N}^{2} $. Comparing \eqref{GrindEQ__3_10_} and \eqref{GrindEQ__3_15_}, we get
 \begin{equation} \label{GrindEQ__3_16_}
 \mathop{ \min} \limits_{ \hat{A}_{N} \in \Lambda} \mathop{ \max} \limits_{ \xi \in \Xi} \Delta ( \xi , \hat{A}_{N} ) \le \mathop{ \max} \limits_{ \xi \in \Xi} \mathop{ \min} \limits_{ \hat{A}_{N} \in \Lambda} \Delta ( \xi , \hat{A}_{N} ).
 \end{equation}
Since the opposite inequality always holds true, only equality is possible in \eqref{GrindEQ__3_16_}.
Proof is complete.

From the proof of the theorem a construction of the optimal minimax estimate follows.

 \begin{cor}
The optimal minimax estimate $ \hat{A}_{N} { \xi}$ of the functional $A_{N} { \xi}$ is of the form

 \[ \hat{A}_{N} { \xi}= \sum_{j=0}^{N} {a}(j) \left( \sum_{u=j-N}^{-1} \varphi(j-u) { \varepsilon}(u) \right), \]
where $ {\varepsilon}(u) $
is a  stationary sequence with orthogonal values, the sequence $
\{\varphi (u),u=0,1,\dots,N\}$ is
uniquely determined by coordinates of the eigenvector of the
operator $Q_{N} $ that corresponds to the greatest eigenvalue $
\nu_{N}^{2} $ and condition $ E\left | { \xi}(j) \right | ^{2}
=P$.
 \end{cor}

 \begin{prikl} \label{prikl.1.1}
Consider the problem of optimal linear stimulation of the functional
 \[A_{1} { \xi}= \xi (0)+  \xi (1) \]
that depends on the unknown values of a stationary sequence $ { \xi}(j)$, that satisfies the conditions
 \[ E{ \xi}(j)=0,\quad E\left | { \xi}(j) \right |^{2} \le 1, \]
based on observations of the sequence $ { \xi}(j)$ at points $j=-1,-2,\dots$.

\noindent Eigenvalues of the operator $Q_{1} $, determined by
equation \eqref{GrindEQ__3_8_}, are equal to $3 \pm \sqrt{5} $.
So the greatest eigenvalue is $ \nu_{1}^{2} =3+ \sqrt{5} $. The eigenvector corresponding to the eigenvalue $ \nu_{1}^{2} =3+ \sqrt{5} $ is of the
form $ \vec\varphi = \left \{ \varphi (0), \varphi(1) \right \}$, where
 \[ \varphi (0)= \sqrt{(5+ \sqrt{5} )/10},\quad \varphi (1)= \sqrt{(5- \sqrt{5} )/10} . \]

\noindent The least favourable stationary sequence $ { \xi}(j)$ is a moving average sequence of the form
 \[ { \xi}(j)= \varphi (0) \varepsilon (j)+ \varphi (1) \varepsilon (j-1)= \]
 \[= \sqrt{(5+ \sqrt{5} )/10} \,\,\varepsilon(j) +
 \sqrt{(5- \sqrt{5} )/10}\,\,\varepsilon(j+1). \]

\noindent The optimal linear minimax estimate $ \hat{A}_{1} { \xi}$ of the functional $A_{1} { \xi}$ is of the form
 \[ \hat{A}_{1} { \xi}=  \varphi (1)\,\, \varepsilon (-1)= \sqrt{(5- \sqrt{5} )/10}\,\, \varepsilon (-1). \]

\noindent The mean-square error of the optimal estimate of the functional
$A_{1} { \xi}$ does not exceed $3+ \sqrt{5}$.
 \end{prikl}

\subsection{The maximum value of the mean square error of estimate of the functional $A {\xi}$}

In this subsection the maximum value of the mean square error of the optimal linear estimate of the functional
\[A { \xi} = \sum_ {j = 0} ^ { \infty}  {a} (j)  { \xi} (j)\]
 which depend on unknown values of a stationary stochastic sequence $ { \xi} (j)$ from the class $ \Xi $ is found.

We will suppose that the sequence $ \{ {a}(j): j=0,1, \ldots \}$ which determines the functional $A  { \xi} $ satisfies conditions
 \begin{equation} \label{GrindEQ__3_17_}
  \sum_{j=0}^{ \infty} \left|a (j) \right|  < \infty , \quad \sum_{j=0}^{ \infty}(j+1) \left | {a}(j) \right | ^{2}  < \infty.
 \end{equation}

 \begin{teo} \label{theo1.2}
The function $ \Delta ( \xi , \hat{A})=E \left|A{ \xi}- \hat{A} { \xi} \right|^{2}$ has a saddle point on the set $ \Xi \times \Lambda $. The following relation holds true
 \[ \mathop{ \min} \limits_{ \hat{A} \in \Lambda} \mathop{ \max} \limits_{ \xi \in \Xi} \Delta ( \xi, \hat{A})= \mathop{ \max} \limits_{ \xi \in \Xi} \mathop{ \min} \limits_{ \hat{A} \in \Lambda} \Delta ( \xi , \hat{A})=P \nu ^{2} . \]

\noindent The least favourable in the class $ \Xi $ of
stationary stochastic sequences satisfying conditions
\eqref{GrindEQ__3_1_} for the optimal linear estimation of the
functional $A { \xi}$  is a  moving average
sequence
 \[ { \xi}(j)= \sum_{u=- \infty}^{j} \varphi(j-u) { \varepsilon}(u). \]
Here $ \nu^{2} $ is the greatest eigenvalue and $ \vec\varphi= \left \{
\varphi(u): u=0,1,\dots\right\}$ is the corresponding eigenvector
of the selfadjoint compact operator in the space $ \ell_{2} $ determined
 by the matrix
 \[Q= \left \{Q(p,q) \right \}_{p,q=0}^{ \infty}, \quad Q(p,q)= \sum_{u=0}^{ \infty} {a}(p+u) \overline{a}(q+u) , \]
$ { \varepsilon}(u)$ is a  stationary sequence with orthogonal values.
\end{teo}

\textit{Proof.} Lower bound. Let $ \xi \in \Xi_{R} $. In this
case the following inequality holds true

 \begin{equation} \label{GrindEQ__3_18_}
 \mathop{ \max} \limits_{ \xi \in \Xi} \mathop{ \min} \limits_{ \hat{A} \in \Lambda} \Delta ( \xi , \hat{A}) \ge \mathop{ \max} \limits_{ \xi \in \Xi_{R}} \mathop{ \min} \limits_{ \hat{A} \in \Lambda} \Delta ( \xi , \hat{A} ).
 \end{equation}
Making use the canonical representation of the regular
stationary sequence \eqref{GrindEQ__3_3_} as a   moving
average sequence and form
 \eqref{GrindEQ__3_5_} of the optimal estimate we have
 \[ \mathop{ \min} \limits_{ \hat{A} \in \Lambda} \Delta ( \xi , \hat{A} )= \mathop{ \min} \limits_{ \hat{A} \in \Lambda} E \left|A { \xi}- \hat{A} { \xi} \right|^{2} =
 \]
 \[
=E \left| \sum_{j=0}^{ \infty} {a}(j) \sum_{u=0}^{j} \varphi (j-u) { \varepsilon}(u) \right|^{2} = \]
 \begin{equation} \label{GrindEQ__3_19_}
= \sum_{p,q=0}^{ \infty} \;  \varphi (p) \overline{ \varphi (q)}Q (p,q),
 \end{equation}
\noindent where
 \begin{equation} \label{GrindEQ__3_20_}
Q(p,q)=  \sum_{u=0}^{ \infty} {a}(p+u) \overline{a(q+u)}.
 \end{equation}

 \noindent Denote by $Q$ the operator in the space $ \ell_{2}$,
determined by the matrix  $Q=
\left \{Q(p,q) \right \}_{p,q=0}^{ \infty}.$ Since the second
condition from \eqref{GrindEQ__3_17_} is satisfied and
 \[ \sum_{p,q=0}^{ \infty} \left | Q(p,q) \right | ^{2}  =  \sum_{p,q=0}^{ \infty} \;  \left| \sum_{u=0}^{ \infty}a (p+u) \overline{a (q+u)} \right|^{2} \le \]
 \[ \le \sum_{p,q=0}^{ \infty} \;  \left( \sum_{u=0}^{ \infty} \left|a (p+u) \right|^{2} \cdot \sum_{u=0}^{ \infty} \left|a (q+u) \right|^{2} \right)= \]
 \[= \left( \sum_{p=0}^{ \infty}  \sum_{u=0}^{ \infty} \left|a (p+u) \right|^{2}
 \right)^{2} =  \left( \sum_{p=0}^{ \infty}(p+1) \left | {a}(p) \right | ^{2} \right)^{2} , \]
\noindent we have
 \[ \left \| Q \right \| \le N(Q) \le \sum_{p=0}^{ \infty}(p+1) \left | {a}(p) \right | ^{2} < \infty, \]

 \noindent where $N(Q)$ is the Hilbert-Schmidt norm \label{normGlSh}
of the operator $Q$. The operator $Q$ is selfadjoint (its matrix is
Hermitian) Hilbert-Schmidt operator.\label{operGlSh} For these
reasons the operator $Q$ is selfadjoint continuous operator
\cite{Akhieser},\cite{Riesz}. It can be represented in the form $Q=A
\cdot A^{*}$, where the operator $A$ is determined by the matrix
 \[A= \left \{A(p,q) \right \}_{p,q=0}^{ \infty},\quad  A(p,q)= {a}(p+q),\, p,q=0,1,\dots,\infty. \]

 \noindent The operator $Q$ has real nonnegative
eigenvalues. Note that the operator $A$
is a Hilbert-Schmidt operator and its Hilbert-Schmidt norm equals to
 \[N(A)= \left( \sum_{p=0}^{ \infty}(p+1) \left | {a}(p) \right | ^{2} \right)^{1/2} . \]

\noindent Let us introduce the notation
 \[ \vec{ \varphi}=\{\tilde{ \varphi}(u): u=0,1,\dots\},\quad  \tilde{ \varphi}(u)= P^{-1/2} \varphi (u) . \]
Making use the introduced notations \eqref{GrindEQ__3_19_} can be written in the form
 \[ \mathop{ \min} \limits_{ \hat{A} \in \Lambda} \Delta ( \xi , \hat{A})=P \left \langle Q \vec{ \varphi}, \vec{ \varphi} \right \rangle . \]
\noindent Taking into consideration \eqref{GrindEQ__3_4_}, we get
 \[ \mathop{ \max} \limits_{ \xi \in \Xi_{R}} \mathop{ \min} \limits_{ \hat{A} \in \Lambda} \Delta ( \xi , \hat{A})=P \mathop{ \max} \limits_{ \left \| \vec{ \varphi} \right \| =1} \left \langle Q \vec{ \varphi}, \vec{ \varphi} \right \rangle =P \nu ^{2} , \]
\noindent where $ \nu^{2}$ is the greatest eigenvalue of the
operator $Q$ and $ \left \langle \cdot , \cdot \right \rangle $ is
the inner product in the space  $ \ell_{2}$.  Making use
\eqref{GrindEQ__3_18_}, we can estimate the maximin value of the
error
 \begin{equation} \label{GrindEQ__3_21_}
 \mathop{ \max} \limits_{ \xi \in \Xi} \mathop{ \min} \limits_{ \hat{A} \in \Lambda} \Delta ( \xi , \hat{A}) \ge P \nu ^{2} .
 \end{equation}

Upper bound. Consider the sequence of operators $Q_{N}$,
determined by matrices \eqref{GrindEQ__3_8_}, \noindent and the operator
$Q$ determined by the matrix with elements \eqref{GrindEQ__3_20_}. Since the second
condition from \eqref{GrindEQ__3_17_} is satisfied, then
 \[N(Q-Q_{N} )= \sum_{p=N+1}^{ \infty}(p+1) \left | {a}(p) \right | ^{2} \to 0, \]
\noindent  as $N \to \infty$. Having in mind that
 \[ \left \| Q-Q_{N} \right \| \le N(Q-Q_{N}), \]
\noindent we have
 \[ \mathop{ \lim} \limits_{N \to \infty} \left \| Q-Q_{N} \right \| =0. \]
\noindent It means that the sequence of operators $Q_{N}$ converge
uniformly to the operator $Q$. That is why
\cite{Dunford}, \cite{Gould}
 \[ \mathop{ \lim} \limits_{N \to \infty} \nu_{N}^{2} = \nu ^{2} , \]
where $ \nu_{N}^{2} $ is the greatest eigenvalue of the operator $Q_{N} $,
and $ \nu ^{2} $ is the greatest eigenvalue of the operator $Q$.
Making use the statement of Theorem \ref{theo.1.1}, we can write
 \begin{equation} \label{GrindEQ__3_22_}
 \mathop{ \min} \limits_{ \hat{A} \in \Lambda} \mathop{ \max} \limits_{ \xi \in \Xi} \Delta ( \xi , \hat{A})=
  \end{equation}
  \[= \mathop{ \lim} \limits_{N \to \infty} \mathop{ \min} \limits_{ \hat{A}_{N} \in \Lambda} \mathop{ \max} \limits_{ \xi \in \Xi} \Delta ( \xi , \hat{A}_{N} )= \mathop{ \lim} \limits_{N \to \infty} P \nu_{N}^{2} =P \nu ^{2} .
\]
Comparing \eqref{GrindEQ__3_22_} and \eqref{GrindEQ__3_21_}, we get
 \[ \mathop{ \min} \limits_{ \hat{A} \in \Lambda} \mathop{ \max} \limits_{ \xi \in \Xi} \Delta ( \xi , \hat{A})=P \nu ^{2} \le \mathop{ \max} \limits_{ \xi \in \Xi} \mathop{ \min} \limits_{ \hat{A} \in \Lambda} \Delta ( \xi , \hat{A}), \]
\noindent where only equality is possible.

This completes the proof of the theorem.

 \begin{cor}
The optimal minimax linear estimate $ \hat{A} { \xi}$ of the functional $A { \xi}$ is of the form

 \[ \hat{A} { \xi}= \sum_{j=0}^{ \infty} {a}(j) \left[ \sum_{u=- \infty}^{-1} \varphi (j-u) { \varepsilon}(u) \right] , \]
\noindent where $ \varepsilon(u)$
is a standard stationary sequence with orthogonal
values, the sequence $ \{\varphi (u): u= 0,1,\dots \}$ is uniquely determined by
coordinates of the eigenvector of the operator $Q$ that corresponds
to the greatest eigenvalue $ \nu^{2} $ and condition $E \left \|
{ \xi}(j) \right \| ^{2} =P$.
 \end{cor}

 \begin{prikl} \label{prikl.1.3}
Consider the problem of optimal linear estimation of the functional
 \[A { \xi}=  \sum_{j=0}^{ \infty}e^{- \lambda j} \xi (j)  , \]
\noindent where $ \lambda >0$,
which depends on the unknown values of a stationary sequence $ { \xi}(j)$, that satisfies the condition
 \[ E{ \xi}(j) =0,\quad E \left| \xi (j) \right|^{2} \le 1, \]
\noindent based on observations of the sequence at points $j=-1,-2,\dots$. Conditions
 \eqref{GrindEQ__3_17_} are satisfied. Elements of the matrix which determines the operator $Q$,  determined by the equation \eqref{GrindEQ__3_20_},
are of the form
 \[Q(p,q)= \sum_{u=0}^{ \infty}a(p+u) \overline{a(q+u)}=e^{- \lambda (p+q)} (1-e^{-2 \lambda} )^{-1}. \]
\noindent Eigenvalues of the operator $Q$ are determined by the system of equations
 \[ \mu \varphi (p)=  \sum_{s=0}^{ \infty}e^{- \lambda (p+s)} (1-e^{-2 \lambda} )^{-1} \varphi (s), \;  \; p=0, 1, \ldots  . \]
\noindent It follows from this system of equations that $ \varphi (p)$ are of the form
 \[ \varphi (p)=Ce^{- \lambda p},\,\, p=0, 1, \ldots  \]
\noindent The constant $C$ is determined by the
condition
\[ \sum_{p=0}^{ \infty}|\varphi (p)|^2=1.
\]
\noindent So we have
 \[C=(1-e^{-2 \lambda} )^{1/2} , \quad \varphi (p)= (1-e^{-2 \lambda} )^{1/2} e^{- \lambda p}. \]
\noindent Substitution of these expressions to the system of equations gives us
 \[ \mu =(1-e^{-2 \lambda} )^{-2} . \]
We can conclude that the least favourable in the class $ \Xi $
 stationary sequence $ { \xi}(j)$ is a
 moving average sequence of the form
 \[ { \xi}(j)= (1-e^{-2 \lambda} )^{1/2} e^{- \lambda j} \sum_{u=- \infty}^{j}e^{ \lambda u} \varepsilon (u) , \]
\noindent where  $\varepsilon (u)$ is a stationary sequence with orthogonal values.

\noindent The optimal linear minimax estimate
$ \hat{A} { \xi}$ of the functional $A { \xi}$ is as follows
 \[ \hat{A} { \xi}= (1-e^{-2 \lambda} )^{1/2} \sum_{j=0}^{ \infty}e^{-2 \lambda j} \left[ \sum_{u=- \infty}^{-1}e^{ \lambda u} \varepsilon (u) \right]. \]

\noindent The value of the mean-square error does not exceed
 \[\Delta=(1-e^{-2 \lambda} )^{-2} . \]
\noindent This value of the mean-square error gives the least favourable stationary sequence.
 \end{prikl}

\subsection{Conclusions}

In this section we propose a method of solution of the mean square optimal linear estimation of the functionals
 $A_ {N}  { \xi} = \sum_ {j = 0} ^ {N}  {a} (j) { \xi} (j)$ and  $A { \xi} = \sum_ {j = 0} ^ { \infty}  {a} (j)  { \xi} (j) $
 which depend on the unknown values of a stationary stochastic sequence $ { \xi} (j)$ from the class $ \Xi $
of stationary stochastic sequences satisfying the conditions
$E {\xi} (j)=  {0},$ $ E | { \xi} (j) |^{2}  \le P.$
Estimates are based on results of observations of the sequence $ { \xi} (j) $ at points of time $ j =-1,-2,\dots$.

Inspired by the Ulf~Grenander~\cite{Grenander} approach to investigation the problem
of optimal linear estimation of the functionals which depend on the unknown values of a stationary stochastic process
we consider the problem as a two-person zero-sum game.
It is show that this game has an equilibrium point.
The maximum error gives a moving average stationary sequence which is
least favourable in the given class of stationary sequences.
 The greatest value
of the error and the least favourable sequence are determined by the
greatest eigenvalue and the corresponding eigenvector of the
operator determined by the coefficients $a(j)$ which determine the functional.

For the corresponding results for stationary stochastic processes with values in a Hilbert space see papers by Moklyachuk~\cite{Moklyachuk:1981a} -- \cite{Moklyachuk:1986}, \cite{Moklyachuk:2008r}.

 \section{Extrapolation problem for functionals of stationary sequences}

 In this section we deal with the problem of the mean-square optimal estimation of the linear functionals
 \[A_N\xi=\sum\limits_{j=0}^{N}a(j)\xi(j),\quad A{\xi}=\sum_{j=0}^{\infty}{a}(j){\xi}(j), \]
 which depend on the unknown values of a stationary stochastic sequence $\xi(j),j\in \mathbb{Z},$
 based on observations of the sequence ${\xi} (j) $ at points of time $ j=-1,-2,\dots$.

The problem is investigated in the case of spectral certainty, where the spectral density of the stationary stochastic sequence $\xi(j)$ is exactly known.
In this case the classical Hilbert space projection method of linear estimation of the functional is applied. Formulas are derived for calculation the value of the mean square error and the spectral characteristic of the mean-square optimal estimate of the linear functional.
In the case of spectral uncertainty, where the spectral density of the stationary stochastic sequence is not exactly known, but a class of admissible spectral densities is given, the minimax-robust
procedure to linear estimation of the functional is applied.
Relations which determine the least favourable spectral densities and the minimax spectral characteristics are proposed for some special sets of admissible spectral densities.

\subsection{The classical Hilbert space projection method of linear extrapolation}

Let $\xi(j),\,j\in \mathbb{Z}$, $E\xi(j)=0$, be a (wide sense) stationary stochastic sequence. We will consider values of $\xi(j),\,j\in \mathbb{Z}$, as elements of the Hilbert space  $H=L_2(\Omega,\mathcal{F},P)$ of complex valued random variables with zero first moment, $E\xi=0$, finite second moment,
 $E|\xi|^2<\infty$, and the inner product $\left(\xi,{\eta}\right)=E\xi\overline{\eta}$.
The correlation function $R(k)=\left(\xi(j+k),\xi(j)\right)=E\xi(j+k)\overline{\xi(j)}$ of the stationary stochastic sequence $\xi(j),\,j\in \mathbb{Z}$, admits
 the spectral representation
\[
 R(k)=\int\limits_{-\pi}^{\pi}e^{ik\lambda}F(d\lambda),
 \]
where $F(d\lambda)$ is the spectral measure of the sequence.
We will consider regular stationary stochastic sequences with absolutely continuous spectral measures and the correlation functions of the form
\[
 R(k)=\,\frac{1}{2\pi}\int\limits_{-\pi}^{\pi}e^{ik\lambda}f(\lambda)d\lambda,
 \]
where $f(\lambda)$ is the spectral density function of the sequence $\xi(j)$
that satisfies the regularity condition
\begin {equation} \label{extrreg}
 \int\limits_{-\pi}^{\pi} \ln{(f(\lambda))}d\lambda >-\infty.
\end{equation}

This condition is necessary and sufficient in order that the error-free extrapolation of the unknown values of the sequence is impossible \cite{Gihman}.

\noindent Suppose that coefficients $ {a}(j)$, which determine
the functional $A { \xi}$, satisfy conditions
 \begin{equation} \label{GrindEQ__3_23_}
 \sum_{j=0}^{ \infty}  \left|a (j) \right|  < \infty , \quad \sum_{j=0}^{ \infty}(j+1) \left | {a}(j) \right | ^{2}  < \infty .
 \end{equation}
\noindent In this case the functional $A { \xi}$ has the second moment and the
Hilbert-Schmidt operator $Q$, determined in the previous section, has its Hilbert-Schmidt norm finite.

\noindent Let the sequence $ { \xi}(j)$ admits the canonical
representation \label{kanrozkl} as a moving average
sequence
 \begin{equation} \label{GrindEQ__3_24_}
 { \xi}(j)= \sum_{u=- \infty}^{j}d(j-u) { \varepsilon}(u),
 \end{equation}
where the sequence $d(u)$ satisfies condition
\[\sum_{u=0}^{ \infty}|d(u)|^2<\infty,
\]
and where $ {\varepsilon}(u) $
is a standard  stationary white noise sequence
\label{noiseposl}
 \[E { \varepsilon}(i) \overline{ \varepsilon}(j)= \delta_{ij}, \]
\noindent $ \delta_{ij}$ is the Kronecker symbol.

In this case the  spectral density $f ( \lambda ) $ of the stationary sequence $ { \xi}(j)$ admits the canonical factorization \label{canfaktor}
 \begin{equation} \label{GrindEQ__3_25_}
f( \lambda )= \varphi ( \lambda ) \overline{\varphi ( \lambda )}, \quad \varphi ( \lambda )= \sum_{k=0}^{ \infty}d(k)e^{-ik \lambda}.
 \end{equation}
\noindent Denote by $L_{2} (f)$ the Hilbert space of functions
$ {a}( \lambda )$ such that
 \[  \int_{- \pi}^{ \pi} \, a ( \lambda ) \overline{a ( \lambda )} \, f ( \lambda ) \, d \lambda  < \infty. \]

\noindent Denote by $L_{2}^{-} (f)$ the subspace of the space $L_{2} (f)$ generated by functions
 \[\{e^{in \lambda},  n=-1,-2,\dots\}. \]
\noindent Every linear estimate $ \hat{A} { \xi}$ \label{linest} of the functional $A { \xi}$ based on observations of the sequence $ { \xi}(j)$ at points $j=-1,-2,\dots$ is of the form
 \[ \hat{A} { \xi}= \int_{- \pi}^{ \pi}h(e^{i \lambda} )Z_{ \xi}(d \lambda) , \]
where $Z_{ \xi} ( \Delta )$ is the orthogonal random measure \label{vypmira} of the sequence $ { \xi}(j)$:
\[
E \left(Z_{ \xi} ( \Delta_1 )\overline{Z_{ \xi} ( \Delta_2 )}\right)=\frac{1}{2 \pi} \int_{\Delta_1\cap\Delta_2} f( \lambda )d \lambda,
\]
$h(e^{i \lambda} ) $ is the spectral characteristic of the estimate $ \hat{A} { \xi}$ which belongs to the subspace $L_{2}^{-} (f)$. The mean-square error \label{error} of the linear estimate $ \hat{A} { \xi}$ of the functional $A { \xi}$ based on observations of the sequence $ { \xi}(j)$ at points $j=-1,-2,\dots$, is calculated by the formula
 \[  \Delta (h,f)=M \left|A { \xi}- \hat{A}{ \xi} \right|^{2}= \]
 \[ = \frac{1}{2 \pi} \int_{- \pi}^{ \pi}   \left|A(e^{i \lambda} )-h(e^{i \lambda} )\right|^{2} f( \lambda )d \lambda  , \]
\noindent where
 \[A(e^{i \lambda} )= \sum_{j=0}^{ \infty} {a}(j)e^{ij \lambda}.\]

\noindent If the  stationary stochastic sequence $
{ \xi}(j)$ admits the canonical representation in the form
of a  moving average sequence \eqref{GrindEQ__3_24_},
then the optimal estimate of the functional $A { \xi}$  is
determined by the spectral characteristic \label{spchar} $h(f) \in
L_{2}^{-}(f)$ such that
 \begin{equation} \label{GrindEQ__3_26_}
 \Delta (h(f),f)= \mathop{ \min} \limits_{h \in L_{2}^{-} (f)} \Delta (h,f)= \left \| Ad \right \| ^{2} ,
 \end{equation}
\noindent where
 \[ \left \| Ad \right \| ^{2} = \sum_{k=0}^{ \infty} \left \| (Ad)_k \right \| ^{2}  , \quad \left(Ad \right)_{k}= \sum_{l=0}^{ \infty}{a}(k+l)d(l) . \]

\noindent Note, that $ \left \| Ad \right \| ^{2} < \infty $ under
the conditions \eqref{GrindEQ__3_23_}. The spectral characteristic
$h(f)$ of the optimal estimate is calculated by the formula
 \begin{equation} \label{GrindEQ__3_27_}
h(f)=A(e^{i \lambda} )- {\varphi}^{-1} ( \lambda )\, r(e^{i \lambda} ),
 \end{equation}

\noindent where
 \begin{equation} \label{r}
  r(e^{i \lambda} )= \sum_{k=0}^{ \infty}(Ad)_{k} e^{ik \lambda}.
  \end{equation}

For the functional $A_{N} { \xi}$ the mean-square error and
the spectral characteristic of the optimal estimate of the functional is calculated by the formulas
 \begin{equation} \label{GrindEQ__3_28_}
 \Delta_{N} (h(f),f)= \left \| A_{N} d \right \| ^{2} ,
 \end{equation}
 \begin{equation} \label{GrindEQ__3_29_}
h_{N}(f)=A_{N}(e^{i \lambda} )- {\varphi}^{-1} ( \lambda )\, r_{N}(e^{i \lambda} ),
 \end{equation}

\noindent where
 \begin{equation} \label{rN}
 A_{N} (e^{i \lambda} )= \sum_{j=0}^{N} {a}(j)e^{ij \lambda}  , \quad r_{N} (e^{i \lambda} )= \sum_{k=0}^{N}(A_{N} d)_{k} e^{ik \lambda},
 \end{equation}
 \[ \left \| A_{N} d \right \| ^{2} = \sum_{k=0}^{N} \left \| (A_{N} d)_{k} \right \| ^{2} , \quad(A_{N} d)_{k} = \sum_{l=0}^{N-k}{a}(k+l)d(l),k= 0,1,\dots,N. \]

\noindent As a corollary, from the formula \eqref{GrindEQ__3_28_}
one can find the following formula for calculation the mean-square
error of optimal estimates $ \hat{ \xi} (j)$ of the unknown
values $ \xi (j)$
 \begin{equation} \label{GrindEQ__3_30_}
E \left| \xi (j)- \hat{ \xi} (j) \right|^{2} = \sum_{u=0}^{j} \left | d (u) \right | ^{2}  ,
 \end{equation}
where  $d(u)$ are determined from equations of factorization \eqref{GrindEQ__3_25_} of the density $f( \lambda )$.

Thus we came to conclusion that the following theorem holds true.

 \begin{teo} \label{theo.1.3}
If conditions \eqref{GrindEQ__3_23_} are satisfied and the density $f( \lambda)$ admits the canonical factorization \eqref{GrindEQ__3_25_}, then
the mean-square error of the optimal linear estimate of the functional
$A { \xi}$ based on observations of the sequence $ { \xi}(j)$
 at points $j=-1,-2,\dots$, is calculated by formula
 \eqref{GrindEQ__3_26_} (by formula \eqref{GrindEQ__3_28_} for the functional $A_{N}
 { \xi}$).
The spectral characteristic of the optimal linear estimate of the functional is calculated by formula
 \eqref{GrindEQ__3_27_} $($ by formula
 \eqref{GrindEQ__3_29_} for the functional $A_{N} { \xi}$ $)$.
 \end{teo}

\subsection{Minimax-robust method of linear extrapolation of functionals}

\noindent Formulas \eqref{GrindEQ__3_26_} -- \eqref{GrindEQ__3_30_}
can be applied to calculate the spectral characteristic and the
mean-square error of the optimal linear estimate of the functional
$A { \xi}$ only in the case where the spectral
density $f( \lambda )$ of the  stationary stochastic
sequence $ { \xi}(j)$ is exactly known. In the case where the
spectral density $f( \lambda )$ is not exactly known, but, instead, a set
$D$ of admissible spectral densities is specified, the
minimax approach to the problem of the optimal linear estimate of
the functional which depends on the unknown values of the stationary sequence is
reasonable. Under this approach one finds an estimate of the
functional which is optimal and minimize the mean-square error for
all spectral densities from a given class $D$ simultaneously.
\medskip

 \begin{definition}
\noindent A spectral density $f^{0} ( \lambda )$ is called the least favourable \label{unfavor1} in the class $D$ for the optimal linear extrapolation of the functional $A { \xi}$ if the following relation holds true
 \[ \Delta (h(f^{0} ),f^{0} )= \mathop{ \max} \limits_{f \in D} \Delta (h(f),f)= \mathop{ \max} \limits_{f \in D} \mathop{ \min} \limits_{h \in L_{2}^{-} (f)} \Delta (h,f). \]
 \end{definition}

\noindent Taking into consideration relations \eqref{GrindEQ__3_26_} -- \eqref{GrindEQ__3_30_}, we can verify that the following statements hold true

 \begin{teo} \label{theo.1.4}
The spectral density $f^{0} ( \lambda ) \in D$ is the least favourable in the class $D$
for the optimal linear extrapolation of the functional $A { \xi}$, if it admits the canonical factorization
\begin{equation} \label{factor1}
 f^{0} ( \lambda )= \left( \sum_{k=0}^{ \infty}d^{0} (k)e^{-ik \lambda} \right) \cdot \overline{\left( \sum_{k=0}^{ \infty}d^{0} (k)e^{-ik \lambda} \right)},
\end{equation}
\noindent where $d^{0} = \left \{d^{0} (k):k=0, \, 1, \, \ldots \right \}$ is a solution to the constrained optimization problem
 \begin{equation} \label{GrindEQ__3_31_}
  \left \| Ad \right \| ^{2} \to \max,
\end{equation}
\[
 f( \lambda )= \left( \sum_{k=0}^{ \infty}d(k)e^{-ik \lambda} \right) \cdot \overline{\left( \sum_{k=0}^{ \infty}d(k)e^{-ik \lambda} \right)} \in D.
\]
 \noindent The sequence $ { \xi}(j)$ in this case admits the canonical one-sided moving average representation
 \begin{equation} \label{GrindEQ__3_33_}
 { \xi}(j)= \sum_{u=-\infty}^{j}d^0(j-u) { \varepsilon}(u) .
 \end{equation}
 \end{teo}

 \begin{teo} \label{theo.1.5}
The spectral density $f^{0}(\lambda)\in D$ is the least favourable in the class $D$
for the optimal linear extrapolation of the functional $A_N { \xi}$, if it admits the canonical factorization
\begin{equation} \label{factorN}
f^{0} ( \lambda )= \left( \sum_{k=0}^{N}d^{0} (k)e^{-ik \lambda} \right) \cdot \overline{\left( \sum_{k=0}^{N}d^{0} (k)e^{-ik \lambda} \right)},
\end{equation}
\noindent where $d^{0} = \left \{d^{0} (k): k=0,1,\dots,N \right \}$
is a solution to the constrained optimization problem
 \begin{equation} \label{GrindEQ__3_32_}
\left \| A_{N} d \right \| ^{2} \to \max,
 \end{equation}
\[
f( \lambda )= \left( \sum_{k=0}^{N}d(k)e^{-ik \lambda} \right) \cdot\overline{
\left( \sum_{k=0}^{N}d(k)e^{-ik \lambda} \right)} \in D.
\]
\noindent The sequence $ { \xi}(j)$ in this case admits the canonical one-sided moving average representation of order $N+1$:
 \begin{equation} \label{GrindEQ__3_33N_}
 { \xi}(j)= \sum_{u=j-N}^{j}d^0(j-u) { \varepsilon}(u) .
 \end{equation}
 \end{teo}

 \begin{definition}
The spectral characteristic $h^{0} (e^{i \lambda})$ of the optimal linear extrapolation of the functional $A { \xi}$ is called minimax \label{minmaxch1} (robust) if the following conditions hold true
 \[h^{0} (e^{i \lambda} ) \in H_{D} = \bigcap_{f \in D}L_{2}^{-} (f), \quad \mathop{ \min} \limits_{h \in H_{D}} \mathop{ \max} \limits_{f \in D} \Delta (h,f)= \mathop{ \max} \limits_{f \in D} \Delta (h^{0},f). \]
 \end{definition}

The least favourable spectral density $f^{0} ( \lambda ) \in D$
and the minimax (robust) spectral characteristic $h^{0} (e^{i \lambda} ) \in H_{D} $ form a saddle point of the function
$ \Delta (h,f)$ on the set $H_{D}\times{D}$.
The saddle point inequalities
 \[ \Delta (h,f^{0} ) \ge \Delta (h^{0} ,f^{0} ) \ge \Delta (h^{0} ,f) \quad \forall f \in D \; \forall h \in H_{D} \]
hold true if $h^{0} =h(f^{0})$ and $h(f^{0} )\in H_{D}$, where
$f^{0}$ is a solution to the constrained optimization problem
 \[ \Delta (h(f^{0} ),f^{0} )= \mathop{ \max} \limits_{f \in D} \Delta (h(f^{0} ),f). \]
If we have found a solution $f^{0} ( \lambda )$ to this problem, then the
minimax(robust) spectral characteristics can be calculated by
formulas \eqref{GrindEQ__3_27_}, \eqref{GrindEQ__3_29_} if the
condition $h(f^{0} ) \in H_{D}$ holds true.

The spectral density $f^{0} ( \lambda )$ is a solution to the following constrained optimization problem
 \begin{equation} \label{GrindEQ__3_34_}
 \Delta (f)=- \Delta (h(f^{0} ),f) \to \inf,\quad f( \lambda ) \in D,
 \end{equation}
\[
 \Delta (h(f^{0} ),f)= \frac{1}{2 \pi} \int_{- \pi}^{ \pi}\frac{|r(e^{i \lambda} )|^2}{|\varphi^0 ( \lambda )|^2}f( \lambda )d \lambda,
\]

\noindent where $r(e^{i \lambda} )$ is calculated by formulas
\eqref{r}, \eqref{rN} with $f( \lambda)=f^{0} ( \lambda )$.

The constrained optimization problem
 \eqref{GrindEQ__3_34_} is equivalent to the following unconstrained optimization problem \cite{Moklyachuk:2008},\cite{Pshenychn}
 \begin{equation} \label{GrindEQ__3_35_}
 \Delta_{D} (f)=- \Delta (h(f^{0} ),f)+ \delta (f|D) \to \inf ,
 \end{equation}
where $ \delta (f|D)$ is the indicator function of the set $D$. Solution of the problem \eqref{GrindEQ__3_35_} is determined by the condition $0 \in \partial \Delta_{D} (f^{0} )$, where $ \partial \Delta_{D} (f^{0} )$ is the subdifferential on the convex functional $ \Delta_{D} (f)$ at the point $f^{0} $. With the help of conditions \eqref{GrindEQ__3_34_}, \eqref{GrindEQ__3_35_} we can find the least favourable spectral densities for concrete classes of spectral densities.

Note, that the form of the functional $\Delta(h(f_0),f)$  is convenient for application the Lagrange method of indefinite multipliers for
finding solution to the problem (\ref{GrindEQ__3_34_}).
Making use the method of Lagrange multipliers and the form of
subdifferentials of the indicator functions of certain classes
of spectral densities we describe relations that determine least favourable spectral densities in some special classes
of spectral densities \cite{Moklyachuk:2008r}, \cite{Moklyachuk:2012}.

 \subsection{Least favourable spectral densities in the class $D_{0}$}

Consider the problem of the optimal estimation of the functionals $A\xi=\sum_{j=1}^{\infty}a(j)\xi(j)$ and $A_N\xi=\sum_{j=1}^{N}a(j)\xi(j)$  which depends on the unknown values of a stationary stochastic sequence $\xi(j)$
from observations of the sequence $\xi(j)$ at points of time $j=-1,-2,\dots$
in the case where the spectral density  $f(\lambda)$ is from the class $D_0$
of spectral densities which are characterized by restrictions on the first
moment of the density
 \[D_{0} = \left \{f( \lambda ): \, \frac{1}{2 \pi} \int_{- \pi}^{ \pi} \; f( \lambda )d \lambda \leq P_0 \right \}, \]
\noindent where $P_0,P_0>0,$ is a given number.
This class of spectral densities describes stationary sequences with restriction on the dispersion $E|\xi(j)|^2\leq P_0$.

We can apply the method of Lagrange multipliers to find solution to the optimization problem (\ref{GrindEQ__3_34_}).
We get the following relations that determine the least favourable spectral density  $f^0\in D_0$
\begin{equation} \label{rD0}
| r(e^{i \lambda} )|^2= \alpha^{2}|\varphi^0 ( \lambda )|^2,
\end{equation}
where $\alpha ^{2}$ is the Lagrange multiplier.

\noindent This relation can be rewritten in the following way
\begin{equation} \label{GrindEQ__3_36_}
 \left( \sum_{k=0}^{ \infty}(Ad^0)_{k} e^{ik \lambda} \right) \cdot \overline{\left( \sum_{k=0}^{ \infty}(Ad^0)_{k} e^{ik \lambda} \right)}= \alpha ^{2}
  \left( \sum_{k=0}^{ \infty}d^0(k)e^{-ik \lambda} \right) \cdot \overline{\left( \sum_{k=0}^{ \infty}d^0(k)e^{-ik \lambda} \right)}.
\end{equation}

\noindent The unknown $ \alpha ^{2} $ and $d^0= \left \{d^0(k):k=0, 1, \ldots \right \}$ are determined with the help of equations of the
canonical factorization \eqref{factor1} of the density $f^{0}(\lambda)$, solution of the constrained optimization problem \eqref{GrindEQ__3_31_} and restrictions
imposed on densities from the class of admissible spectral densities $D_0$.

For all solutions $d= \left \{d(k):k=0, 1, \ldots \right \}$  of the equation
\begin{equation}\label{eqAd}
Ad=\alpha \overline{d}, \,\alpha\in \mathbb C,
\end{equation}
which satisfy condition
\begin{equation}\label{nd}
\|d\|^2= \sum_{k=0}^{\infty}|d(k)|^2=P_0 ,
\end{equation}
the following equality holds true
\begin{equation} \label{lfD0}
 f^0(\lambda)=\left| \sum_{k=0}^{ \infty}d(k)e^{-ik \lambda} \right|^2=\left|c \sum_{k=0}^{ \infty}(Ad)_{k} e^{ik \lambda} \right|^2.
 \end{equation}
Solution $d^0= \left \{d^0(k):k=0, 1, \ldots \right \}$  of the equation \eqref{eqAd} which satisfies condition \eqref{nd} and gives the maximum value
 $\|Ad^0\|^2=\nu_0 P_0$ of the quantity $\|Ad\|^2$ determines the least favourable spectral density
 \begin{equation} \label{lfdD0}
 f^0(\lambda)=\left| \sum_{k=0}^{ \infty}d^0(k)e^{-ik \lambda} \right|^2,
 \end{equation}
which is the spectral density of the one-sided  moving-average sequence
 \begin{equation} \label{maD0}
 { \xi}(j)= \sum_{u=-\infty}^{j}d^0(j-u) { \varepsilon}(u) .
 \end{equation}

\noindent Thus the following statement fulfilled.

 \begin{teo} \label{theo.1.6}
The least favourable in the class $D_{0}$  for the optimal
extrapolation of the functional $A {\xi}$ is the spectral density
\eqref{lfdD0} of the one-sided  moving-average sequence \eqref{maD0}
which is determined by solution $d^0= \left \{d^0(k):k=0, 1, \ldots \right \}$  of the equation \eqref{eqAd} which satisfies condition \eqref{nd} and gives the maximum value
 $\|Ad^0\|^2=\nu^2 P_0$ of the quantity $\|Ad\|^2$.
\noindent The minimax spectral characteristic of the optimal estimate of the functional $A{\xi}$ is calculated by the formula \eqref{GrindEQ__3_27_}.
 \end{teo}

For the functional $A_N\xi=\sum_{j=1}^{N}a(j)\xi(j)$ the corresponding relation
has the following form
 \begin{equation}  \label{rND0}
\left|r_{N} (e^{i \lambda} )\right|^2= \alpha ^{2} \left|\varphi ^{0} ( \lambda )\right|^2,
 \end{equation}

\noindent and the equality
 \begin{equation}  \label{ArND0}
\left|r_{N} (e^{i \lambda} )\right|^2= \left| \sum_{k=0}^{N}(A_{N} d)_{k} e^{ik \lambda} \right|^2 =  \left| \sum_{k=0}^{N}( \tilde{A}_{N} d)_{k} e^{-ik \lambda}
\right|^2,
 \end{equation}
\noindent holds true, where
 \[( \tilde{A}_{N} d)_{k} = \sum_{u=0}^{k}{a}(N-k+u)d(u), \;  k=0,1,\dots, N. \]

\noindent For all solutions $d= \left \{d(k): k=0,1,\dots, N\right \}$ to  equations
\begin{equation}\label{eqANd1}
A_Nd=\alpha \overline{d}, \,\alpha\in \mathbb C,
\end{equation}
\begin{equation}\label{eqANd2}
\tilde{A}_{N}d=\beta {d}, \,\beta\in \mathbb C,
\end{equation}
\noindent which satisfy condition
\begin{equation}\label{ndN}
\|d\|^2= \sum_{k=0}^{N}|d(k)|^2=P_0 ,
\end{equation}
the following equality holds true
\begin{equation} \label{lfD0N}
 f^0(\lambda)=\left| \sum_{k=0}^{N}d(k)e^{-ik \lambda} \right|^2=\left|c \sum_{k=0}^{N}(Ad)_{k} e^{ik \lambda} \right|^2.
 \end{equation}

Solutions $d^0= \left \{d^0(k):k=0, 1, \ldots, N\right \}$  of the equations \eqref{eqANd1}, \eqref{eqANd2} which satisfy condition \eqref{ndN} and gives the maximum values
 $\|A_Nd^0\|^2=\|\tilde{A}_{N}d^0\|^2=\nu_N^2 P_0$ of the quantity $\|A_Nd\|^2$ determines the least favourable spectral density
 \begin{equation} \label{lfdD0N}
 f^0(\lambda)=\left| \sum_{k=0}^{N}d^0(k)e^{-ik \lambda} \right|^2,
 \end{equation}
which is the spectral density of the one-sided  moving-average sequence of order $N$
 \begin{equation} \label{maD0N}
 { \xi}(j)= \sum_{u=j-N}^{j}d^0(j-u) { \varepsilon}(u) .
 \end{equation}

\noindent Thus the following statement holds true.

\begin{teo} \label{theo.1.7}
The least favourable in the class $D_{0}$  for the optimal
extrapolation of the functional $A_N {\xi}$ is the spectral density
\eqref{lfdD0N} of the one-sided  moving-average sequence \eqref{maD0N}
which is determined by solutions $d^0= \left \{d^0(k):k=0, 1, \ldots,N \right \}$  of the equations \eqref{eqANd1}, \eqref{eqANd2} which satisfy condition \eqref{ndN}  and gives the maximum value
$\|A_Nd^0\|^2=\|\tilde{A}_{N}d^0\|^2=\nu_N^2 P_0$ of the quantity $\|A_Nd\|^2$.
\noindent The minimax spectral characteristic of the optimal estimate of the functional $A_N{\xi}$ is calculated by the formula \eqref{GrindEQ__3_29_}.
 \end{teo}

\begin{cor}
The spectral density \eqref{lfdD0N} of the one-sided  moving-average sequence \eqref{maD0N}, where the sequence  $d^0= \left \{d^0(k):k=0, 1, \ldots,N \right \}$ satisfy condition \eqref{ndN},
 is the least favourable
in the class $D_{0}$ for the optimal extrapolation of the
functional $ {a}(N) { \xi}(N)$.
\end{cor}

\subsection{Least favourable spectral densities in the class $D_{M}$}

Consider the problem of the optimal estimation of the functionals $A\xi=\sum_{j=1}^{\infty}a(j)\xi(j)$ and $A_N\xi=\sum_{j=1}^{N}a(j)\xi(j)$  which depends on the unknown values of a stationary stochastic sequence $\xi(j)$
from observations of the sequence $\xi(j)$ at points of time $j=-1,-2,\dots$
in the case where the spectral density  $f(\lambda)$ is from the class $D_M$
of spectral densities which are characterized by restrictions on the moments of the density
 \[D_{M} = \left \{f( \lambda )\left|\, \frac{1}{2 \pi} \int_{- \pi}^{ \pi} f( \lambda )\cos(m\lambda)d\lambda =r(m), \, m= 0,1,\dots,M \right.\right \}, \]
 where $r(0)=P_0$ and $r(0),r(1),\dots,r(M)$ is a strictly positive sequence. The moment problem in this case has not uniquely determined solutions and the class $D_M$ contains an infinite number of densities \cite{Krein}.

We can apply the method of Lagrange multipliers to find solution to the optimization problem (\ref{GrindEQ__3_34_}) in the case $D= D_M$.
We get the following relations that determine the least favourable spectral density  $f^0\in D_M$
\begin{equation} \label{rDM}
| r(e^{i \lambda} )|^2(f^0(\lambda))^{-1}= \sum\limits_{m=0}^M \alpha_m \cos (m\lambda)=\left|\sum\limits_{m=0}^M p(m) e^{im\lambda}\right|^2,
\end{equation}
where $\alpha_m, m=0,1,\ldots,M $ are the Lagrange multipliers.

It follows from this relation that the least favourable spectral density  $f^0\in D_M$ is of the form
\begin{equation}\label{lfDM}
f_0(\lambda)=\frac{\left|\sum\limits_{k=0}^{\infty}(Ad)_ke^{ik\lambda}\right|^2}{\left|\sum\limits_{m=0}^M p(m) e^{-im\lambda}\right|^2}.
\end{equation}

 \noindent The unknown $ \{p(m): m=0,1,\dots,M\}$ and
$d= \left \{d(k):k=0, \, 1, \, \ldots \right \}$ are determined by
equations of the canonical factorization \eqref{factor1} of the density $f^{0} ( \lambda)$,
 solution of the constrained optimization problem \eqref{GrindEQ__3_31_} and restrictions
imposed on densities from the class of admissible spectral densities $D_M$.

Denote by $\nu_M P_0$ the maximum value $\|Ad^0\|^2$ of the quantity $\|Ad\|^2$, where
$d= \left \{d(k):k=0, \, 1, \, \ldots \right \}$ satisfies condition \eqref{nd} and are determined by
equations of the canonical factorization \eqref{factor1} of the density $f^{0} ( \lambda)$
 and restrictions
imposed on densities from the class of admissible spectral densities $D_M$.

Denote by $\nu_M^{+} P_0$ the maximum value $\|Ad^0\|^2$ of the quantity $\|Ad\|^2$, where
$d= \left \{d(k):k=0,  1, \ldots \right \}$ satisfies condition \eqref{nd} and are determined by
equations of canonical factorization \eqref{factor1} of the density \eqref{lfDM}
and restrictions
imposed on densities from the class of admissible spectral densities $D_M$.

\noindent The following statement is true.

\begin{teo} \label{theo.1.8}
If there exists a sequence  $d^{0} = \left \{d^0(k):k=0,  1, \ldots \right \}$ which satisfies condition \eqref{nd}
and such that $\nu_MP_0 =\nu_M^{+}P_0=\|Ad^0\|^2$, then the least favourable in the class $D_{M}$  for the optimal
extrapolation of the functional $A {\xi}$ is the spectral density
\eqref{lfdD0} of the one-sided  moving-average sequence \eqref{maD0}.
If $\nu_M<\nu_M^{+}$, then the least favourable in the class $D_{M}$  for the optimal
extrapolation of the functional $A {\xi}$ is the spectral density
\eqref{lfDM} which admits the canonical factorization \eqref{factor1}.
The unknown $ \{p(m): m=0,1,\dots,M\}$ and
$d= \left \{d(k):k=0, \, 1, \, \ldots \right \}$ are determined by
equations of canonical factorization \eqref{factor1} of the density $f^{0} ( \lambda)$,
 solution of the constrained optimization problem \eqref{GrindEQ__3_31_} and restrictions
imposed on densities from the class of admissible spectral densities $D_M$.
 The minimax spectral characteristic of the optimal estimate of the functional $A { \xi}$ is calculated by the formula \eqref{GrindEQ__3_27_}.
 \end{teo}

Consider now the problem of the optimal estimation of the functional  $A_N\xi=\sum_{j=1}^{N}a(j)\xi(j)$.
It follows from the relation \eqref{rDM} that in this case the least favourable spectral density  $f^{0} ( \lambda)\in D_M$ is of the form
\begin{equation}\label{NlfDM}
f_0(\lambda)=\frac{\left|\sum\limits_{k=0}^{N}(A_Nd)_ke^{ik\lambda}\right|^2}{\left|\sum\limits_{m=0}^M p(m) e^{-im\lambda}\right|^2}.
\end{equation}
These densities are spectral densities of the autoregressive-moving-average $ARMA(M,N)$ sequences
 \begin{equation} \label{armaD0N}
 \sum\limits_{m=0}^M p(m) \xi(n-m)= \sum_{k=0}^{N}(A_Nd)_k{\varepsilon}(n-k).
 \end{equation}

\noindent The unknown $ \{p(m): m=0,1,\dots,M\}$ and
$d= \left \{d(k):k=0, 1, \ldots,N \right \}$ are determined by
equations of canonical factorization \eqref{factorN}  of the density $f^{0} ( \lambda)$,
 solution of the constrained optimization problem \eqref{GrindEQ__3_32_} and restrictions
imposed on densities from the class of admissible spectral densities $D_M$.

Denote by $\nu_{MN} P_0$ the maximum value $\|A_Nd^0\|^2$ of the quantity $\|A_Nd\|^2$, where
$d= \left \{d(k):k=0, 1, \ldots,N \right \}$ are determined by equations \eqref{eqANd1}, \eqref{eqANd2}, condition \eqref{ndN} and
equations of canonical factorization \eqref{factorN}  of the density $f^{0} ( \lambda)\in D_M$.

Denote by $\nu_{MN}^{+} P_0$ the maximum value $\|A_Nd^0\|^2$ of the quantity $\|A_Nd\|^2$, where
$d= \left \{d(k):k=0,  1, \ldots,N \right \}$ satisfies condition \eqref{ndN},
equations of canonical factorization \eqref{factor1}  of the density \eqref{NlfDM} and restrictions
imposed on densities from the class of admissible spectral densities $f^{0} ( \lambda)\in D_M$.

Thus the following statement holds true.

\begin{teo} \label{theo.1.9}
If there exists a solution $d^0= \left \{d^0(k):k=0, 1, \ldots,N \right \}$ of the equation \eqref{eqANd1}, or the equation \eqref{eqANd2}, which satisfy condition \eqref{ndN}  and
such that  $\nu_{MN}P_0 =\nu_{MN}^{+}P_0=\|A_Nd^0\|^2$, then the least favourable in the class $D_{M}$  for the optimal
extrapolation of the functional $A_N {\xi}$ is the spectral density
\eqref{lfdD0N} of the one-sided  moving-average sequence \eqref{maD0N}.
If $\nu_{MN}<\nu_{MN}^{+}$, then the least favourable in the class $D_{M}$  for the optimal
extrapolation of the functional $A_N {\xi}$ is the spectral density
\eqref{NlfDM} of the autoregressive-moving-average $ARMA(M,N)$ sequences
 \eqref{armaD0N}.
The unknown $ \{p(m): m=0,1,\dots,M\}$ and
$d= \left \{d(k):k=0, \, 1, \, \ldots \right \}$ are determined by
equations of canonical factorization \eqref{factor1} of the density $f^{0} ( \lambda)$,
 solution of the constrained optimization problem \eqref{GrindEQ__3_31_} and restrictions
imposed on densities from the class of admissible spectral densities $D_M$.
 The minimax spectral characteristic of the optimal estimate of the functional $A_N { \xi}$ is calculated by the formula \eqref{GrindEQ__3_29_}.
 \end{teo}

 \subsection{Least favourable spectral densities in the class $D_{v}^{u}$}

\noindent Consider the problem of minimax estimation of the functionals
$A { \xi}$ and $A_{N} { \xi}$ which depend on the unknown values of a  stationary stochastic sequence $ { \xi}(j)$ for the sets of spectral densities that describe the ``strip'' model \label{smuga} of stationary stochastic sequences
 \[D_{v}^{u} = \left \{f( \lambda )\left| v( \lambda ) \le f( \lambda ) \le u( \lambda ), \; \frac{1}{2 \pi} \int_{- \pi}^{ \pi} f( \lambda )d \lambda  =P_0 \right.\right \}, \]

\noindent where $v( \lambda ), u( \lambda)$ are given bounded spectral densities.

From the condition $0 \in \partial\Delta_D(f_0)$ for $D=D_v^{u}$ we find the following
equation which determines the least favourable spectral density
for the optimal estimation of the functional $A { \xi}$
\begin{equation}\label{eqDvu}
\left|\sum\limits_{k=0}^{\infty}(Ad)_ke^{ik\lambda}\right|^2=(\psi_1(\lambda)+\psi_2(\lambda)+\alpha_0)\left|\varphi^0(\lambda)\right|^2,
\end{equation}
where $\psi_1(\lambda)\geq 0$ and $\psi_1(\lambda)=0$ if $f_0(\lambda)\geq v(\lambda);$ $\psi_2(\lambda)\leq 0$ and $\psi_2(\lambda)=0$ if $f_0(\lambda)\leq u(\lambda).$

From this equation we find that the least favourable spectral density for the optimal estimation of the functional $A { \xi}$ is of the form
\begin{equation}\label{lfDvu}
f_0(\lambda)=\max\left\{v(\lambda),\min\left\{u(\lambda), c\left|\sum\limits_{k=0}^{\infty}(Ad)_ke^{ik\lambda}\right|^2\right\}\right\}.
\end{equation}

Denote by $\nu_u P_0$ the maximum value $\|Ad^0\|^2$ of the quantity $\|Ad\|^2$, where
$d= \left \{d(k):k=0, \, 1, \, \ldots \right \}$ are solutions of the equation \eqref{eqAd}
which satisfy condition \eqref{nd}, the inequality
\begin{equation}\label{ieqDvu}
v(\lambda)\leq \left|\sum\limits_{k=0}^{\infty}d(k)e^{-ik\lambda}\right|^2\leq u(\lambda),
\end{equation}
and determine the canonical factorization \eqref{factor1} of the density $f^{0} ( \lambda)\in D_v^{u}$.

Denote by $\nu_u^{+} P_0$ the maximum value $\|Ad^0\|^2$ of the quantity $\|Ad\|^2$, where
$d= \left \{d(k):k=0,  1, \ldots \right \}$ satisfies condition \eqref{nd} and determine the canonical factorization \eqref{factor1} of the density \eqref{lfDvu}
from the class of admissible spectral densities $D_v^{u}$.

Thus the following theorem holds true.

\begin{teo} \label{theo1.10}
If there exists a solution of the equation \eqref{eqAd}
which satisfy condition \eqref{nd} and such that
$\nu_u P_0=\nu_u^{+} P_0=\|Ad^0\|^2$, then the spectral density \eqref{lfdD0} of the one-sided  moving-average sequence \eqref{maD0}
is the least favourable in the set $D_{v}^{u}$
for the optimal extrapolation of the functional $A{ \xi}$.
If $\nu_u<\nu_u^{+}$, then the least favourable in the class $D_{v}^{u}$  for the optimal
extrapolation of the functional $A {\xi}$ is the spectral density
\eqref{lfDvu} which admits the canonical factorization \eqref{lfdD0}.
The sequence $d= \left \{d(k):k=0,  1, \ldots \right \}$ is determined by
the optimisation problem \eqref{GrindEQ__3_31_} and restrictions
imposed on densities by the given set of admissible spectral
densities.
\noindent The minimax spectral characteristic of the optimal estimate of the functional $A{ \xi}$ is calculated by the formula \eqref{GrindEQ__3_27_}.
 \end{teo}

For the functional $A_{N} { \xi}$
the least favourable spectral density
for the optimal estimation of the functional is of the form
\begin{equation}\label{lfDvuN}
f_0(\lambda)=\max\left\{v(\lambda),\min\left\{u(\lambda), c\left|\sum\limits_{k=0}^{N}(A_Nd)_ke^{ik\lambda}\right|^2\right\}\right\}.
\end{equation}

Denote by $\nu_{uN} P_0$ the maximum value $\|A_Nd^0\|^2$ of the quantity $\|A_Nd\|^2$, where
$d= \left \{d(k):k=0, 1,\ldots,N \right \}$ are solutions of the equations \eqref{eqANd1}, \eqref{eqANd2},
which satisfy condition \eqref{ndN}, the inequality
\begin{equation}\label{ieqDvu}
v(\lambda)\leq \left|\sum\limits_{k=0}^{N}d(k)e^{-ik\lambda}\right|^2\leq u(\lambda),
\end{equation}
and determine the canonical factorization \eqref{factorN} of the density $f^{0} ( \lambda)\in D_v^{u}$.

Denote by $\nu_{uN}^{+} P_0$ the maximum value $\|A_Nd^0\|^2$ of the quantity $\|A_Nd\|^2$, where
$d= \left \{d(k):k=0,  1, \ldots \right \}$ satisfies condition \eqref{nd} and determine the canonical factorization \eqref{factor1} of the density \eqref{lfDvuN}
from the class of admissible spectral densities $D_v^{u}$.

The following theorem holds true.

\begin{teo} \label{theo1.11}
If there exists a solution of the equation \eqref{eqANd1}, or equation \eqref{eqANd2},
which satisfy condition \eqref{ndN} and such that
$\nu_{uN} P_0=\nu_{uN}^{+} P_0=\|A_Nd^0\|^2$,
then the spectral density \eqref{lfdD0N} of the one-sided  moving-average sequence \eqref{maD0N}
is the least favourable in the set $D_{v}^{u}$
for the optimal extrapolation of the functional $A_N{ \xi}$.
If $\nu_{uN}<\nu_{uN}^{+}$, then the least favourable in the class $D_{v}^{u}$  for the optimal
extrapolation of the functional $A_N {\xi}$ is the spectral density
\eqref{lfDvuN} which admits the canonical factorization \eqref{lfdD0}.
The sequence $d= \left \{d(k):k=0,  1, \ldots \right \}$ is determined by
the optimisation problem \eqref{GrindEQ__3_31_} and restrictions
imposed on densities by the given set of admissible spectral
densities.
\noindent  The minimax spectral characteristic of the optimal estimate of the functional $A_N{ \xi}$ is calculated by the formula \eqref{GrindEQ__3_29_}.
 \end{teo}

 \subsection{Least favourable spectral densities in the class $D_{\varepsilon}$}

Consider the problem of minimax estimation of the functionals
$A { \xi}$ and $A_{N} { \xi}$ which depend on the unknown values of a stationary stochastic sequence $ { \xi}(j)$ for the set of spectral densities that describes the ``$ \varepsilon$-- contamination" model \label{contam} of stationary stochastic sequences
 \[
 D_{ \varepsilon} = \left \{f( \lambda )\left| \; f( \lambda )=(1- \varepsilon )w( \lambda )+ \varepsilon u( \lambda ), \,\,  \frac{1}{2 \pi} \int_{- \pi}^{ \pi} f ( \lambda )d \lambda  =P_0 \right.\right\},
\]
\noindent where $w( \lambda )$ is a known
spectral density, and $u( \lambda )$ is an
unknown spectral density.
 From the condition $0 \in \partial \Delta_{D}
(f^{0} )$ for the functional $A { \xi}$  we find the following
equations which determine the least favourable spectral densities
for the optimal estimation of the functional $A { \xi}$ for
the given set of admissible spectral densities
\begin{equation}\label{eqDe}
\left|\sum\limits_{k=0}^{\infty}(Ad)_ke^{ik\lambda}\right|^2=(\psi_1(\lambda)+\alpha_0^{-1})\left|\varphi^0(\lambda)\right|^2,
\end{equation}
where $\psi_1(\lambda)\geq 0$ and $\psi_1(\lambda)=0$ if $f_0(\lambda)\geq (1- \varepsilon )w( \lambda ).$

From this equation we find that the least favourable spectral density for the optimal estimation of the functional $A { \xi}$ is of the form
\begin{equation}\label{lfDe}
f_0(\lambda)=\max\left\{(1- \varepsilon )w( \lambda ), \alpha_0\left|\sum\limits_{k=0}^{\infty}(Ad)_ke^{ik\lambda}\right|^2\right\}.
\end{equation}

 Denote by $\nu_{\varepsilon} P_0$ the maximum value $\|Ad^0\|^2$ of the quantity $\|Ad\|^2$, where
$d= \left \{d(k):k=0, \, 1, \, \ldots \right \}$ are solutions of the equation \eqref{eqAd}
which satisfy condition \eqref{nd}, the inequality
\begin{equation}\label{ieqDe}
\left|\sum\limits_{k=0}^{\infty}d(k)e^{-ik\lambda}\right|^2\geq (1- \varepsilon )w( \lambda ),
\end{equation}
and determine the canonical factorization \eqref{factor1} of the density $f^{0} ( \lambda)\in D_{ \varepsilon}$.

Denote by $\nu_{\varepsilon}^{+} P_0$ the maximum value $\|Ad^0\|^2$ of the quantity $\|Ad\|^2$, where
$d= \left \{d(k):k=0,  1, \ldots \right \}$ satisfies condition \eqref{nd} and determine the canonical factorization \eqref{factor1} of the density \eqref{lfDe}
from the class of admissible spectral densities $D_{ \varepsilon}$.

The following theorem holds true.

 \begin{teo} \label{theo1.12}
 If there exists a solution of the equation \eqref{eqAd}
which satisfy condition \eqref{nd} and such that
$\nu_{\varepsilon} P_0=\nu_{\varepsilon}^{+} P_0=\|Ad^0\|^2$, then the spectral density \eqref{lfdD0} of the one-sided  moving-average sequence \eqref{maD0}
is the least favourable in the set $D_{ \varepsilon}$
for the optimal extrapolation of the functional $A{ \xi}$.
If $\nu_{\varepsilon}<\nu_{\varepsilon}^{+}$, then the least favourable in the class $D_{ \varepsilon}$  for the optimal
extrapolation of the functional $A {\xi}$ is the spectral density
\eqref{lfDe} which admits the canonical factorization \eqref{lfdD0}.
The sequence $d= \left \{d(k):k=0,  1, \ldots \right \}$ is determined by
the optimisation problem \eqref{GrindEQ__3_31_} and restrictions
imposed on densities by the given set of admissible spectral
densities.
\noindent  The minimax spectral characteristic of the optimal estimate of the functional $A{ \xi}$ is calculated by the formula \eqref{GrindEQ__3_27_}.
 \end{teo}

For the functional $A_{N} { \xi}$
the least favourable spectral density
for the optimal estimation of the functional is of the form
\begin{equation}\label{lfDeN}
f_0(\lambda)=\max\left\{(1- \varepsilon )w( \lambda ), \alpha_0\left|\sum\limits_{k=0}^{N}(A_Nd)_ke^{ik\lambda}\right|^2\right\}.
\end{equation}

Denote by $\nu_{\varepsilon}^N P_0$ the maximum value $\|A_Nd^0\|^2$ of the quantity $\|A_Nd\|^2$, where
$d= \left \{d(k):k=0, 1,\ldots,N \right \}$ are solutions of the equations \eqref{eqANd1}, \eqref{eqANd2},
which satisfy condition \eqref{ndN}, the inequality
\begin{equation}\label{ieqDeN}
 \left|\sum\limits_{k=0}^{N}d(k)e^{-ik\lambda}\right|^2\geq (1- \varepsilon )w( \lambda ),
\end{equation}
and determine the canonical factorization \eqref{factorN} of the density $f^{0} ( \lambda)\in D_{ \varepsilon}$.

Denote by $\nu_{\varepsilon}^{N+} P_0$ the maximum value $\|A_Nd^0\|^2$ of the quantity $\|A_Nd\|^2$, where
$d= \left \{d(k):k=0,  1, \ldots \right \}$ satisfies condition \eqref{nd} and determine the canonical factorization \eqref{factor1} of the density \eqref{lfDvuN}
from the class of admissible spectral densities $D_{ \varepsilon}$.

The following theorem holds true.

\begin{teo} \label{theo1.11}
If there exists a solution of the equation \eqref{eqANd1}, or the equation, \eqref{eqANd2},
which satisfy condition \eqref{ndN} and such that
$\nu_{\varepsilon}^N P_0=\nu_{\varepsilon}^{N+} P_0=\|A_Nd^0\|^2$,
then the spectral density \eqref{lfdD0N} of the one-sided  moving-average sequence \eqref{maD0N}
is the least favourable in the set $D_{ \varepsilon}$
for the optimal extrapolation of the functional $A_N{ \xi}$.
If $\nu_{\varepsilon}^N<\nu_{\varepsilon}^{N+}$, then the least favourable in the class $D_{ \varepsilon}$  for the optimal
extrapolation of the functional $A_N {\xi}$ is the spectral density
\eqref{lfDeN} which admits the canonical factorization \eqref{lfdD0}.
The sequence $d= \left \{d(k):k=0,  1, \ldots \right \}$ is determined by
the optimisation problem \eqref{GrindEQ__3_31_} and restrictions
imposed on densities by the given set of admissible spectral
densities.
\noindent  The minimax spectral characteristic of the optimal estimate of the functional $A_N{ \xi}$ is calculated by the formula \eqref{GrindEQ__3_29_}.
 \end{teo}

 \subsection{Least favourable spectral densities in the class $D_{1 \varepsilon}$}

 Consider the problem of minimax estimation of the functionals
$A { \xi}$ and $A_{N} { \xi}$ which depend on the unknown values of a  stationary stochastic sequence
$ { \xi}(j)$ for the set of spectral densities that describes the model of "$ \varepsilon$-- neighborhood" \label{okilL1} in the space $L_{1}$  of a stationary stochastic sequence
 \[D_{1 \varepsilon} = \left \{f( \lambda)\left|\, \frac{1}{2 \pi} \int_{- \pi}^{ \pi} \left|f(\lambda)-v(\lambda) \right|d \lambda \le \varepsilon \right.\right \}, \]

\noindent where $ \varepsilon$ is a given numbers,  $v( \lambda)$ is a given spectral density.

From the condition $0 \in \partial \Delta_{D} (f^{0} )$  we find that the least favourable spectral density for the optimal
estimation of the functional $A {\xi}$ is of the form
\begin{equation}\label{lfDe1}
f_0(\lambda)=\max\left\{v(\lambda), c\left|\sum\limits_{k=0}^{\infty}(Ad)_ke^{ik\lambda}\right|^2\right\}.
\end{equation}

Denote by $\nu_{1 \varepsilon} P_0$ the maximum value $\|Ad^0\|^2$ of the quantity $\|Ad\|^2$, where
$d= \left \{d(k):k=0, \, 1, \, \ldots \right \}$ are solutions of the equation \eqref{eqAd}
which satisfy condition \eqref{nd}, the inequality
\begin{equation}\label{ieqDe1}
\left|\sum\limits_{k=0}^{\infty}d(k)e^{-ik\lambda}\right|^2\geq v(\lambda),
\end{equation}
and determine the canonical factorization \eqref{factor1} of the density $f^{0} ( \lambda)\in D_{1 \varepsilon}$.

Denote by $\nu_{1 \varepsilon}^{+} P_0$ the maximum value $\|Ad^0\|^2$ of the quantity $\|Ad\|^2$, where
$d= \left \{d(k):k=0,  1, \ldots \right \}$ satisfies condition \eqref{nd} and determine the canonical factorization \eqref{factor1} of the density \eqref{lfDe1}
from the class of admissible spectral densities $D_{1 \varepsilon}$.

The following theorem holds true.

 \begin{teo} \label{theo1.14}
 If there exists a solution of the equation \eqref{eqAd}
which satisfy condition \eqref{nd} and such that
$\nu_{1\varepsilon} P_0=\nu_{1\varepsilon}^{+} P_0=\|Ad^0\|^2$, then the spectral density \eqref{lfdD0} of the one-sided  moving-average sequence \eqref{maD0}
is the least favourable in the set $D_{1 \varepsilon}$
for the optimal extrapolation of the functional $A{ \xi}$.
If $\nu_{1\varepsilon}<\nu_{1\varepsilon}^{+}$, then the least favourable in the class $D_{1 \varepsilon}$  for the optimal
extrapolation of the functional $A {\xi}$ is the spectral density
\eqref{lfDe1} which admits the canonical factorization \eqref{lfdD0}.
The sequence $d= \left \{d(k):k=0,  1, \ldots \right \}$ is determined by
the optimisation problem \eqref{GrindEQ__3_31_} and restrictions
imposed on densities by the given set of admissible spectral
densities.
\noindent  The minimax spectral characteristic of the optimal estimate of the functional $A{ \xi}$ is calculated by the formula \eqref{GrindEQ__3_27_}.
 \end{teo}

For the functional $A_{N} { \xi}$
the least favourable spectral density
for the optimal estimation of the functional is of the form
\begin{equation}\label{lfDe1N}
f_0(\lambda)=\max\left\{v(\lambda), c\left|\sum\limits_{k=0}^{N}(A_Nd)_ke^{ik\lambda}\right|^2\right\}.
\end{equation}

Denote by $\nu_{1\varepsilon}^N P_0$ the maximum value $\|A_Nd^0\|^2$ of the quantity $\|A_Nd\|^2$, where
$d= \left \{d(k):k=0, 1,\ldots,N \right \}$ are solutions of the equations \eqref{eqANd1}, \eqref{eqANd2},
which satisfy condition \eqref{ndN}, the inequality
\begin{equation}\label{ieqDe1N}
 \left|\sum\limits_{k=0}^{N}d(k)e^{-ik\lambda}\right|^2\geq v( \lambda ),
\end{equation}
and determine the canonical factorization \eqref{factorN} of the density $f^{0} ( \lambda)\in D_{1 \varepsilon}$.

Denote by $\nu_{1\varepsilon}^{N+} P_0$ the maximum value $\|A_Nd^0\|^2$ of the quantity $\|A_Nd\|^2$, where
$d= \left \{d(k):k=0,  1, \ldots \right \}$ satisfies condition \eqref{nd} and determine the canonical factorization \eqref{factor1} of the density \eqref{lfDe1N}
from the class of admissible spectral densities $D_{1 \varepsilon}$.

The following theorem holds true.

\begin{teo} \label{theo1.15}
If there exists a solution of the equation \eqref{eqANd1}, or equation \eqref{eqANd2},
which satisfy condition \eqref{ndN} and such that
$\nu_{1\varepsilon}^{N} P_0=\nu_{1\varepsilon}^{N+} P_0=\|A_Nd^0\|^2$,
then the spectral density \eqref{lfdD0N} of the one-sided moving-average sequence \eqref{maD0N}
is the least favourable in the set $D_{1 \varepsilon}$
for the optimal extrapolation of the functional $A_N{ \xi}$.
If $\nu_{1\varepsilon}^{N}<\nu_{1\varepsilon}^{N+}$, then the least favourable in the class $D_{1 \varepsilon}$  for the optimal
extrapolation of the functional $A_N {\xi}$ is the spectral density
\eqref{lfDe1N} which admits the canonical factorization \eqref{lfdD0}.
The sequence $d= \left \{d(k):k=0,  1, \ldots \right \}$ is determined by
the optimisation problem \eqref{GrindEQ__3_31_} and restrictions
imposed on densities by the given set of admissible spectral
densities.
\noindent  The minimax spectral characteristic of the optimal estimate of the functional $A_N{ \xi}$ is calculated by the formula \eqref{GrindEQ__3_29_}.
 \end{teo}

 \subsection{Least favourable spectral densities in the class $D_{2 \varepsilon}$}

Consider the problem of minimax estimation of the functionals
$A { \xi}$ and $A_{N} { \xi}$ which depend on the unknown values of a  stationary stochastic sequence
$ { \xi}(j)$ for the set of spectral densities that describes the model of "$ \varepsilon$-- neighborhood" \label{okilL1} in the space $L_{2}$  of a stationary stochastic sequence
 \[D_{2 \varepsilon} = \left \{f( \lambda)\left|\, \frac{1}{2 \pi} \int_{- \pi}^{ \pi} \left|f(\lambda)-v(\lambda) \right|^2d \lambda \le \varepsilon \right.\right \}, \]

\noindent where $ \varepsilon$ is a given numbers,  $v( \lambda)$ is a given spectral density.

 From the condition $0 \in \partial \Delta_{D} (f^{0} )$  we find that the least favourable spectral density for the optimal
estimation of the functional $A {\xi}$ is of the form
\begin{equation}\label{lfDe2}
f_0(\lambda)=\frac{v(\lambda)}{2} +\left(\frac{(v(\lambda))^2}{4}+\left|\sum\limits_{k=0}^{\infty}(Ad)_ke^{ik\lambda}\right|^2\right)^{1/2}.
\end{equation}
The sequence $d= \left \{d(k):k=0,  1, \ldots \right \}$ is determined by
equations of the canonical factorization \eqref{lfdD0} of the density \eqref{lfDe2}, optimisation problem \eqref{GrindEQ__3_31_} and restriction
\begin{equation}\label{De2}
\frac{1}{2 \pi} \int_{- \pi}^{ \pi} \left|f(\lambda)-v(\lambda) \right|^2d \lambda = \varepsilon.
\end{equation}
The following theorem holds true.

\begin{teo} \label{theo1.17}
The least favourable in the set $D_{2 \varepsilon}$ spectral density for the optimal extrapolation of the functional $A{ \xi}$ is determined by
equation \eqref{lfDe2}.
The sequence $d= \left \{d(k):k=0,  1, \ldots \right \}$ is determined by
equations of the canonical factorization \eqref{lfdD0} of the density \eqref{lfDe2}, optimisation problem \eqref{GrindEQ__3_31_} and restriction
\eqref{De2}
imposed on densities by the given set of admissible spectral
densities.
The minimax spectral characteristic of the optimal estimate of the functional $A{ \xi}$ is calculated by the formula \eqref{GrindEQ__3_27_}.
\end{teo}

For the functional $A_{N} { \xi}$
the least favourable spectral density
for the optimal estimation of the functional is of the form
\begin{equation}\label{lfDe2N}
f_0(\lambda)=\frac{v(\lambda)}{2} +\left(\frac{(v(\lambda))^2}{4}+\left|\sum\limits_{k=0}^{N}(A_Nd)_ke^{ik\lambda}\right|^2\right)^{1/2}.
\end{equation}
The sequence $d= \left \{d(k):k=0,1,\ldots,N \right \}$ is determined by
equations of the canonical factorization \eqref{lfdD0} of the density \eqref{lfDe2N}, optimisation problem \eqref{GrindEQ__3_31_} and restriction
\eqref{De2}.

\begin{teo} \label{theo1.18}
The least favourable in the set $D_{2 \varepsilon}$ spectral density for the optimal extrapolation of the functional $A_N{ \xi}$ is determined by
equation \eqref{lfDe2N}.
The sequence $d= \left \{d(k):k=0,1,\ldots,N \right \}$ is determined by
equations of the canonical factorization \eqref{lfdD0} of the density \eqref{lfDe2N}, optimisation problem \eqref{GrindEQ__3_31_} and restriction
\eqref{De2}
imposed on densities by the given set of admissible spectral
densities.
The minimax spectral characteristic of the optimal estimate of the functional $A{ \xi}$ is calculated by the formula \eqref{GrindEQ__3_29_}.
\end{teo}

\subsection{Conclusions}

In this section  we propose methods of solution of the problem of the mean-square optimal linear estimation of the functionals $A\xi=\sum\limits_{j=0}^{\infty}a(j)\xi(j)$ and $A_N\xi=\sum\limits_{j=0}^{N}a(j)\xi(j)$  which depend on the unknown values of the stationary stochastic sequence $\xi(j)$.
Estimates are based on observations of the sequence $\xi(j)$ at points of time $j=-1,-2,\dots$.
The problem is investigated in the case of spectral certainty, where the spectral densities of the stationary stochastic sequence $\xi(j)$ is exactly known.
In this case the classical Hilbert space projection method of linear estimation is applied. Formulas are derived for calculation the value of the mean square errors and the spectral characteristics of the mean-square optimal estimates of the linear functionals.
In the case of spectral uncertainty, where the spectral density of the stationary stochastic sequence is not exactly known, but a class of admissible spectral densities is given the minimax-robust
procedure to linear estimation of the functionals is applied.
Relations which determine the least favourable spectral densities and the minimax spectral characteristics are proposed for some given sets of admissible spectral densities.

The minimax-robust approach to the problem of one step ahead optimal prediction of the stationary stochastic sequences as well as estimation of one missed value of the sequences based on convex optimization methods was initiated in papers by Franke \cite{Franke1984,Franke1985},  Franke and Poor \cite{FrankePoor}. See also papers by Hosoya \cite{Hosoya}, Taniguchi \cite{Taniguchi1981}, and survey paper by Kassam and Poor \cite{Kassam}.

For the relative results on the mean-square optimal linear extrapolation of linear functionals for stationary stochastic sequences and processes see papers by Moklyachuk~\cite{Moklyachuk:1986} -- \cite{Moklyachuk:1992b}, \cite{Moklyachuk:2008r}, book by Moklyachuk and Masyutka~\cite{Moklyachuk:2012}.

\section{Extrapolation problem for stationary sequences from observations with noise}

In this section we deal with the problem of the mean-square optimal estimation of the linear
functionals $A\xi=\sum\limits_{j=0}^{\infty}a(j)\xi(j)$ and $A_N\xi=\sum\limits_{j=0}^{N}a(j)\xi(j)$ which depend on the unknown values of a stationary stochastic sequence $\xi(j),j\in \mathbb{Z},$
from observations of the sequence $\xi(j)+\eta(j)$ at points of time $j=-1,-2,\dots$, where $\eta(j)$ is an uncorrelated with $\xi(j)$ stationary stochastic sequence.
The problem is investigated in the case of spectral certainty, where the spectral densities of the stationary stochastic sequences $\xi(j)$ and $\eta(j)$ are exactly known.
In this case the classical Hilbert space projection method of linear estimation is applied. Formulas are derived for calculation the value of the mean square errors and the spectral characteristics of the mean-square optimal estimates of the linear functionals.
In the case of spectral uncertainty, where the spectral densities of the stationary stochastic sequences are not exactly known, but a class of admissible spectral densities is given the minimax-robust
procedure to linear estimation of the functional is applied.
Relations which determine the least favourable spectral densities and the minimax spectral characteristics are proposed for some special sets of admissible spectral densities.

\subsection{The classical Hilbert space projection method of linear extrapolation}

Let  $\xi(j), j\in \mathbb{Z}$, and $\eta(j), j\in  \mathbb{Z}$,  be (wide sense) stationary stochastic sequences with zero mathematical expectations $E\xi(j)=0$, $E\eta(j)=0$.
The correlation functions $R_{\xi}(k)=E\xi(j+k)\overline{\xi(j)}$ and $R_{\eta}(k)=E\eta(j+k)\overline{\eta(j)}$
 of stationary stochastic sequences $\xi(j), j\in \mathbb{Z}$, and $\eta(j), j\in  \mathbb{Z}$,
admit  the spectral representations
\[
 R_{\xi}(k)=\int\limits_{-\pi}^{\pi}e^{ik\lambda}F(d\lambda), \quad R_{\eta}(k)=\int\limits_{-\pi}^{\pi}e^{ik\lambda}G(d\lambda),
 \]
where $F(d\lambda)$ and $G(d\lambda)$ are spectral measures of the sequences.
We will consider stationary stochastic sequences with absolutely continuous spectral measures $F(d\lambda)$ and $G(d\lambda)$ and the correlation functions of the form
\[
 R_{\xi}(k)=\,\frac{1}{2\pi}\int\limits_{-\pi}^{\pi}e^{ik\lambda}f(\lambda)d\lambda, \quad R_{\eta}(k)=\,\frac{1}{2\pi}\int\limits_{-\pi}^{\pi}e^{ik\lambda}g(\lambda)d\lambda,
 \]
where $f(\lambda)$ and $g(\lambda)$ are the spectral density functions of the sequences $\xi(j), j\in \mathbb{Z}$, and $\eta(j), j\in  \mathbb{Z}$, correspondingly.

We will suppose that the spectral density functions $f(\lambda)$ and $g(\lambda)$
satisfy the minimality condition
\begin{equation}\label{extrminimalAn}
\int\limits_{-\pi}^{\pi}\frac{1}{f(\lambda)+g(\lambda)}d\lambda<\infty.
\end{equation}
Under this condition the error-free extrapolation of the unknown values of the sequence $\xi(j)+\eta(j)$ is impossible.

The stationary stochastic sequences $\xi(j)$ and $\eta(j)$ admit
 the spectral representations
\begin{equation*}
\xi(j)=\int\limits_{-\pi}^{\pi}e^{ij\lambda}dZ_{\xi}(\lambda), \hspace{1cm}
\eta(j)=\int\limits_{-\pi}^{\pi}e^{ij\lambda}dZ_{\eta}(\lambda),
\end{equation*}
where $Z_{\xi}(d\lambda)$ and $Z_{\eta}(d\lambda)$ are orthogonal stochastic measures of the sequences $\xi(j)$ and $\eta(j)$ such that
\[
E{Z_{\xi}}(\Delta_1)\overline{{Z_{\xi}}(\Delta_2)}=F(\Delta_1\cap\Delta_2)=\,\frac{1}{2\pi}\int_{\Delta_1\cap\Delta_2}f(\lambda)d\lambda,
\]
\[
E{Z_{\eta}}(\Delta_1)\overline{{Z_{\eta}}(\Delta_2)}=G(\Delta_1\cap\Delta_2)=\,\frac{1}{2\pi}\int_{\Delta_1\cap\Delta_2}g(\lambda)d\lambda.
\]

Consider the problem of the mean-square optimal estimation of the linear
functional
\[A\xi=\sum\limits_{j=0}^{\infty}a(j)\xi(j)\]
 which depends on the unknown values of a stationary stochastic sequence $\xi(j),j=0,1,\dots,$  based on observations of the sequence
$\xi(j)+\eta(j)$ at points of time $j=-1,-2,\dots$.

We will suppose that the sequence $ \{ {a}(j): j=0,1, \ldots \}$ which determines the functional $A  { \xi} $ satisfies the following conditions
 \begin{equation} \label{condAn}
  \sum_{j=0}^{ \infty} \left|a (j) \right|  < \infty , \quad \sum_{j=0}^{ \infty}(j+1) \left | {a}(j) \right | ^{2}  < \infty.
 \end{equation}

It follows from the spectral representation of the sequence $\xi(j)$
that we can represent the functional $A\xi$ in the  form
\begin{equation}\label{extrAn}
A\xi=\int\limits_{-\pi}^{\pi}A(e^{i\lambda}){Z}_{\xi}(d\lambda),
 \end{equation}
where
\begin{equation*}
  A(e^{i\lambda})=\sum\limits_{j=0}^{\infty}a(j)e^{ij\lambda }.
 \end{equation*}

Denote by $\hat{A}\xi$  the mean square optimal linear estimate of the functional  $A\xi$ from observations of the sequence
 $\xi(j)+\eta(j)$ at points of time  $j=-1,-2,\dots$.
Denote by $\Delta(f,g)=E\left|A\xi-\hat{A}\xi\right|^2$ the mean square error of the estimate $\hat{A}\xi$. To find the estimate $\hat{A}\xi$
we will use the Hilbert space projection method proposed by Kolmogorov~\cite{Kolmogorov}.
We will consider $\xi(j), j\in \mathbb{Z}$, and $\eta(j), j\in  \mathbb{Z}$, as elements of the Hilbert space  $H=L_2(\Omega,\mathcal{F},P)$
 of complex valued random variables with zero first moment, $E\xi=0$, finite second moment,
 $E|\xi|^2<\infty$, and the inner product $\left(\xi,{\eta}\right)=E\xi\overline{\eta}$.

Denote by $H^0(\xi+\eta)$ the subspace  of the Hilbert space  $H=L_2(\Omega,\mathcal{F},P)$ generated by elements $\{\xi(j)+\eta(j):j=-1,-2,\dots\}$.
Denote by $L_2(f+g)$ the Hilbert space of complex-valued functions  that are square-integrable  with respect to the measure whose density is $f(\lambda)+g(\lambda).$
Denote by $L_2^0(f+g)$ the subspace of $L_2(f+g)$ generated by functions $\{e^{ij\lambda}:j=-1,-2,\dots\}.$

The mean square optimal linear estimate $\hat{A}\xi$ of the functional  $A\xi$ based on observations of the sequence $\xi(j)+\eta(j)$ at points of time  $j=-1,-2,\dots$
is an element of the $H^0(\xi+\eta)$. It can be represented in the form
\begin{equation}\label{hatAn}
\hat{A}\xi=\,\int\limits_{-\pi}^{\pi}h(e^{i\lambda}) (Z_{\xi}(d\lambda)+Z_{\eta}(d\lambda)),
 \end{equation}
where $h(e^{i\lambda}) \in L_2^0(f+g)$  is the spectral characteristic of the estimate $\hat{A}\xi.$

The mean square error $\Delta(h;f,g)$ of the estimate $\hat{A}\xi$ is given by the formula
\begin{equation*}
 \Delta(h;f,g)=E\left|A\xi-\hat{A}\xi\right|^2=\frac{1}{2\pi}\int\limits_{-\pi}^{\pi}\left|A(e^{i\lambda})-h(e^{i\lambda})\right|^2 f(\lambda)d\lambda
 +\frac{1}{2\pi}\int\limits_{-\pi}^{\pi}\left|h(e^{i\lambda})\right|^2 g(\lambda)d\lambda.
\end{equation*}

The Hilbert space projection method proposed by A. N. Kolmogorov \cite{Kolmogorov} makes it possible to find the spectral characteristic  $h(e^{i\lambda})$ and the mean square error $\Delta(h;f,g)$ of the optimal linear estimate of the functional  $A\xi$ in the case where the spectral densities $f(\lambda)$ and $g(\lambda)$ of the sequences $\xi(j),j\in\mathbb{Z},$ and $\eta(j),j\in\mathbb{Z},$ are exactly known and the minimality condition (\ref{extrminimalAn}) is satisfied. The spectral characteristic is determined by the following conditions:
\begin{equation*} \begin{split}
1)& \hat{A}\xi \in H^0(\xi+\eta), \\
2)& A\xi-\hat{A}\xi \bot  H^0(\xi+\eta).
\end{split} \end{equation*}

It follows from the second condition that the following equations should be satisfied
\begin{equation*} \begin{split}
&E\left(A\xi-\hat{A}\xi\right)\left(\overline{\xi(j)}+\overline{\eta(j)}\right)=\\
&=\frac{1}{2\pi}\int\limits_{-\pi}^{\pi} \left(A(e^{i\lambda})- h(e^{i\lambda})\right)e^{-ij\lambda}f(\lambda)d\lambda-\int\limits_{-\pi}^{\pi}  h(e^{i\lambda})e^{-ij\lambda}g(\lambda)d\lambda=0, \,\,j=-1,-2,\dots.
\end{split} \end{equation*}

 The last equations are equivalent to equations
 \begin{equation*}
\int\limits_{-\pi}^{\pi} \left[A(e^{i\lambda})f(\lambda)- h(e^{i\lambda})(f(\lambda)  +g(\lambda))\right]e^{-ij\lambda}d\lambda=0, \,\,j=-1,-2,\dots.
 \end{equation*}

It follows from these equations that the function $\left[A(e^{i\lambda})f(\lambda)- h(e^{i\lambda})(f(\lambda)  +g(\lambda))\right]$ is of the form
\begin{equation}\label{extrAn}
A(e^{i\lambda})f(\lambda)- h(e^{i\lambda})(f(\lambda)  +g(\lambda))=C(e^{i\lambda}),
 \end{equation}
\[C(e^{i\lambda})=\sum\limits_{j=0}^{\infty}c(j)e^{ij\lambda},
\]
where $c(j), j=0,1,\dots$ are unknown coefficients that we have to find.

 From the relation (\ref{extrAn}) we deduce that the spectral characteristic $ h(e^{i\lambda})$
 of the optimal linear estimate of the functional  $A\xi$
 is of the form
 \begin{equation} \label{extrspectharAAnAn}
 \begin{split}
h(e^{i\lambda})=&\frac{A(e^{i\lambda})f(\lambda)-C(e^{i\lambda})}{f(\lambda)+g(\lambda)}=\\
=A(e^{i\lambda})-&\frac{A(e^{i\lambda})g(\lambda)+C(e^{i\lambda})}{f(\lambda)+g(\lambda)}.
\end{split} \end{equation}

It follows from the first condition, which determines the spectral characteristic $h(e^{i\lambda}) \in L_2^0(f+g)$
 of the optimal linear estimate of the functional  $A\xi$
that the Fourier coefficients of the function  $h(e^{i\lambda})$ are equal to zero for $j=0,1,\dots$, namely
\[
\frac{1}{2\pi}\int\limits_{-\pi}^{\pi} h(e^{i\lambda})e^{-ij\lambda }d\lambda=0, \,\, j =0,1,\dots
\]
Using the last relations and (\ref{extrspectharAAnAn}) we get the following system of equations

\begin{equation*} \begin{split}
\int\limits_{-\pi}^{\pi}\left(A(e^{i\lambda})\frac{f(\lambda)}{f(\lambda)+g(\lambda)}-\frac{C(e^{i\lambda})}{f(\lambda)+g(\lambda)}\right)e^{-ij\lambda}d\lambda=0, \,\, j =0,1,\dots
\end{split} \end{equation*}

These equations can be represented in the form
\begin{equation}\label{extrAn3} \begin{split}
\sum_{k=0}^{\infty}a(k)\int\limits_{-\pi}^{\pi}\frac{e^{i(k-j)\lambda}f(\lambda)}{f(\lambda)+g(\lambda)}d\lambda
-\sum\limits_{k=0}^{\infty}c(k)\int\limits_{-\pi}^{\pi}\frac{e^{i(k-j)\lambda}}{f(\lambda)+g(\lambda)}d\lambda=0, \,\, j =0,1,\dots
\end{split} \end{equation}

Let us introduce the following notations
$$R_{j,k}=\frac{1}{2\pi} \int\limits_{-\pi}^{\pi}e^{-i(j-k)\lambda}\frac{f(\lambda)}{f(\lambda)+g(\lambda)}d\lambda;$$
$$B_{j,k}=\frac{1}{2\pi} \int\limits_{-\pi}^{\pi}e^{-i(j-k)\lambda}\frac{1}{f(\lambda)+g(\lambda)}d\lambda;$$
$$Q_{j,k}=\frac{1}{2\pi} \int\limits_{-\pi}^{\pi}e^{-i(j-k)\lambda}\frac{f(\lambda)g(\lambda)}{f(\lambda)+g(\lambda)}d\lambda.$$

Making use the introduced notations we can write equations (\ref{extrAn3})  in the form
\begin{equation*}\begin{split}
\sum\limits_{k=0}^{\infty}R_{j,k}a(k)=
\sum\limits_{k=0}^{\infty}B_{j,k}c(k), \,\, j =0,1,\dots
\end{split} \end{equation*}

The derived equations can be written in the form
\begin{equation*}\begin{split}
\bold{R}\bold{a}=\bold{B} \bold{c},
\end{split} \end{equation*}
where $\bold{a}=(a(0),a(1),\dots)$ is a vector constructed from the coefficients that determine the functional $A\xi$,
$\bold{c}=(c(0),c(1),\dots)$ is a vector constructed from the unknown coefficients $c(k),k=0,1,\dots$, $\bold{B}$ and $ \bold{R}$ are linear operators in $\ell_2$, which are determined by matrices with elements $(\bold{B})_{j,k}=B_{j,k}$, $(\bold{R})_{j,k}=R_{j,k}$, $j,k=0,1,\dots$

We get the formula
\begin{equation} \label{rivnAn2}
\bold{c}=\bold{B}^{-1}\bold{R}\bold{a},
\end{equation}

 Hence, the unknown coefficients $c(j), j=0,1,\dots$, are calculated  by the formula
\[c(j)=\left(\bold{B}^{-1}\bold{R}\bold{a}\right)_j,
\]
where  $\left(\bold{B}^{-1}\bold{R}\bold{a}\right)_j $ is the  $j$-th  component of the vector  $\bold{B}^{-1}\bold{R}\bold{a},$
and the formula for calculating  the spectral characteristic of the estimate $\hat{A}\xi$ is of the form
\begin{equation}\label{extrspectharAAn} \begin{split}
h(e^{i\lambda})=A(e^{i\lambda})\frac{f(\lambda)}{f(\lambda)+g(\lambda)}-
\frac{\sum\limits_{k=0}^{\infty}(\bold{B}^{-1}\bold{R}\bold{a})_ke^{ik\lambda} }{f(\lambda)+g(\lambda)}.
\end{split} \end{equation}

The mean square error of the estimate of the function can be calculated by the formula
\begin{equation} \label{extrdeltaAn} \begin{split}
\Delta(h;f,g)=E\left|A\xi-\hat{A}\xi\right|^2&=\frac{1}{2\pi}\int\limits_{-\pi}^{\pi}\frac{\left|A(e^{i\lambda})g(\lambda)+
\sum\limits_{k=0}^{\infty}(\bold{B}^{-1}\bold{R}\bold{a})_k e^{ik\lambda}\right|^2}{(f(\lambda)+g(\lambda))^2}f(\lambda)d\lambda\\
&+\frac{1}{2\pi}\int\limits_{-\pi}^{\pi}\frac{\left|A(e^{i\lambda})f(\lambda)-
\sum\limits_{k=0}^{\infty}(\bold{B}^{-1}\bold{R}\bold{a})_k e^{ik\lambda}\right|^2}{(f(\lambda)+g(\lambda))^2}g(\lambda)d\lambda\\
&=\langle\bold{R}\bold{a},\bold{B}^{-1}\bold{R}\bold{a}\rangle+\langle\bold{Q}\bold{a},\bold{a}\rangle,
\end{split} \end{equation}
where $\langle\bold{a},\bold{c}\rangle=\sum_{k=0}^{\infty}a(k)\overline{c(k)}$ is the inner product in the space $\ell_2$ and $\bold{Q}$ is a linear operator in $\ell_2$, which is determined by the matrix with elements $(\bold{Q})_{j,k}=Q_{j,k}$,  $j,k=0,1,\dots$

\noindent Let us summarize our results and present them in the form
of a theorem.

\begin{thm}
Let $\xi(j)$ and $\eta(j)$ be stationary stochastic sequences with the spectral densities $f(\lambda)$ and $g(\lambda)$ that satisfy the minimality condition (\ref{extrminimalAn}).
Let conditions \eqref{condAn} be satisfied.
The spectral characteristic $h(e^{i\lambda})$ and the mean square error  $\Delta(h;f,g)$ of the optimal linear estimate $\hat{A}\xi$ of the functional $A\xi$ from observations of the sequence $\xi(j)+\eta(j)$ at points of time $j=-1,-2,\dots $ can be calculated by  formulas (\ref{extrspectharAAn}), (\ref{extrdeltaAn}).
\end{thm}

\subsection{Minimax-robust method of extrapolation}

The traditional methods of estimation of the functional $A\xi$ which depends on the unknown values of a stationary stochastic sequence $\xi(j)$ can be applied in the case  where the spectral densities $f(\lambda)$ and $g(\lambda)$ of the considered stochastic sequences $\xi(j)$ and $\eta(j)$ are exactly known.
In practise, however, we do not have complete information on spectral densities of the sequences. For this reason  we apply the minimax(robust) method of estimation of the functional $A\xi$, that is we find an estimate that minimizes the maximum of the mean square errors for all spectral densities from the given class of admissible spectral densities $D$.

\begin{ozn}
For a given class of spectral densities $D=D_f\times D_g$ the spectral densities $f_0(\lambda)\in D_f$, $g_0(\lambda)\in D_g$  are called the least favourable in $D$ for the optimal linear estimation of the functional $A\xi$ if the following relation holds true
$$\Delta\left(f_0,g_0\right)=\Delta\left(h\left(f_0,g_0\right);f_0,g_0\right)=\max\limits_{(f,g)\in D_f\times D_g}\Delta\left(h\left(f,g\right);f,g\right).$$
\end{ozn}

\begin{ozn}
For a given class of spectral densities $D=D_f\times D_g$ the spectral characteristic $h^0(e^{i\lambda})$ of the optimal linear estimate of the functional $A\xi$ is called minimax-robust if
$$h^0(e^{i\lambda})\in H_D= \bigcap\limits_{(f,g)\in D_f\times D_g} L_2^0(f+g),$$
$$\min\limits_{h\in H_D}\max\limits_{(f,g)\in D}\Delta\left(h;f,g\right)=\sup\limits_{(f,g)\in D}\Delta\left(h^0;f,g\right).$$
\end{ozn}

It follows from the introduced definitions and the obtained formulas that the following statement holds true.

\begin{lem} The spectral densities $f_0(\lambda)\in D_f$, $g_0(\lambda)\in D_g$ are the least favourable in the class of admissible spectral densities $D=D_f\times D_g$ for the optimal linear estimate of the functional $A\xi$ if the Fourier coefficients of the functions
$$(f_0(\lambda)+g_0(\lambda))^{-1}, \quad f_0(\lambda)(f_0(\lambda)+g_0(\lambda))^{-1}, \quad f_0(\lambda)g_0(\lambda)(f_0(\lambda)+g_0(\lambda))^{-1}$$
define operators $\bold B^0, \bold R^0, \bold Q^0$ that determine a solution to the optimization problem
\begin{equation} \label{extrAnDR}\begin{split}
\max\limits_{(f,g)\in D_f\times D_g}\langle\bold{R}\bold{a},\bold{B}^{-1}\bold{R}\bold{a}\rangle+\langle\bold{Q}\bold{a},\bold{a}\rangle\\
=\langle\bold{R}^0\bold{a},(\bold{B}^0)^{-1}\bold{R}^0\bold{a}\rangle+\langle\bold{Q}^0\bold{a},\bold{a}\rangle.
\end{split}\end{equation}
The minimax spectral characteristic $h^0=h(f_0,g_0)$ can be calculated by the formula (\ref{extrspectharAAn}) if $h(f_0,g_0) \in H_D.$
\end{lem}

The least favourable spectral densities $f_0(\lambda)$, $g_0(\lambda)$  and the minimax spectral characteristic $h^0=h(f_0,g_0)$
form a saddle point of the function  $\Delta \left(h;f,g\right)$ on the set $H_D\times D$. The saddle point inequalities
$$\Delta\left(h;f_0,g_0\right)\geq\Delta\left(h^0;f_0,g_0\right)\geq \Delta\left(h^0;f,g\right) $$  $$ \forall h \in H_D, \forall f \in D_f, \forall g \in D_g$$
hold true if $h^0=h(f_0,g_0)$ and $h(f_0,g_0)\in H_D,$  where $(f_0,g_0)$ is a solution to the constrained optimization problem
\begin{equation} \label{extrDeltaAn1}
\sup\limits_{(f,g)\in D_f\times D_g}\Delta\left(h(f_0,g_0);f,g\right)=\Delta\left(h(f_0,g_0);f_0,g_0\right),
\end{equation}
where
\begin{equation}\label{ExtrDfgAn}
\begin{split}
\Delta\left(h(f_0,g_0);f,g\right)&=\frac{1}{2\pi}\int\limits_{-\pi}^{\pi}\frac{\left|A(e^{i\lambda})g_0(\lambda)
+\sum\limits_{j=0}^{\infty}((\bold{B}^0)^{-1}\bold{R}^0\bold{a})_je^{ij\lambda}\right|^2}{(f_0(\lambda)+g_0(\lambda))^2}f(\lambda)d\lambda\\
&+\frac{1}{2\pi}\int\limits_{-\pi}^{\pi}\frac{\left|A(e^{i\lambda})f_0(\lambda)-
\sum\limits_{j=0}^{\infty}((\bold{B}^0)^{-1}\bold{R}^0\bold{a})_je^{ij\lambda}\right|^2}{(f_0(\lambda)+g_0(\lambda))^2}g(\lambda)d\lambda,
\end{split}\end{equation}

The constrained optimization problem (\ref{extrDeltaAn1}) is equivalent to the unconstrained optimization problem
\begin{equation} \label{extrDeltaAn2}
\Delta_D(f,g)=-\Delta(h(f_0,g_0);f,g)+\delta(f,g\left|D_f\times D_g\right.)\rightarrow \inf,
\end{equation}
where $\delta(f,g\left|D_f\times D_g\right.)$ is the indicator function of the set  $D=D_f\times D_g$. Solution $(f_0,g_0)$ to the problem (\ref{extrDeltaAn2}) is characterized by the condition $0 \in \partial\Delta_D(f_0,g_0),$ where $\partial\Delta_D(f_0,g_0)$ is the subdifferential of the convex functional $\Delta_D(f,g)$ at point $(f_0,g_0)$.
This condition makes it possible to find the least favourable spectral densities in some special classes of spectral densities $D$ (see books \cite{Ioffe}, \cite{Pshenychn}, \cite{Rockafellar} for additional details).

Note, that the form of the functional $\Delta(h(f_0,g_0);f,g)$ \eqref{ExtrDfgAn} is convenient for application the Lagrange method of indefinite multipliers for
finding solution to the problem (\ref{extrDeltaAn1}).
Making use the method of Lagrange multipliers and the form of
subdifferentials of the indicator functions
we describe relations that determine least favourable spectral densities in some special classes
of spectral densities
(see books \cite{Golichenko,Moklyachuk:2008,Moklyachuk:2012} for additional details).

\begin{lem} Let $(f_0,g_0)$ be a solution to the optimization problem  (\ref{extrDeltaAn2}). The spectral densities $f_0(\lambda)$, $g_0(\lambda)$ are the least favourable in the class $D=D_f\times D_g$, and
the spectral characteristic $h^0=h(f_0,g_0)$ is minimax for the optimal estimate of the functional $A\xi$, in the case where  $h(f_0,g_0) \in H_D$.
\end{lem}

\subsection{Least favourable spectral densities in the class $D_f^0 \times D_g^0$}

Consider the problem of the optimal estimation of the functional  $A\xi=\sum_{j=0}^{\infty}a(j)\xi(j)$ which depends on the unknown values of a stationary stochastic sequence $\xi(j)$
based on observations of the sequence $\xi(j)+\eta(j)$ at points of time $j=-1,-2,\dots$
in the case where the spectral densities  $f(\lambda)$, $g(\lambda)$ are from the class $D=D_f^0\times D_g^0$, where
\begin{equation*}
D_f^0 = \left\{f(\lambda)\left|\frac{1}{2\pi}\int\limits_{-\pi}^{\pi} f(\lambda)d\lambda\leq P_1\right.\right\},
\end{equation*}
\begin{equation*}
 D_g^0 = \left\{g(\lambda)\left|\frac{1}{2\pi}\int\limits_{-\pi}^{\pi} g(\lambda)d\lambda\leq P_2\right.\right\}.
\end{equation*}

Let the densities $f_0(\lambda) \in D_f^0$, $g_0(\lambda) \in D_g^0$ and the functions $h_f(f_0,g_0)$, $h_g(f_0,g_0)$, determined by the relations
\begin{equation} \label{extrAnhf}
h_f(f_0,g_0)=\frac{\left|A(e^{i\lambda})g_0(\lambda)+
\sum\limits_{j=0}^{\infty}((\bold{B}^0)^{-1}\bold{R}^0\bold{a})_je^{ij\lambda} \right|^2}{(f_0(\lambda)+g_0(\lambda))^2},
\end{equation}
\begin{equation} \label{extrAnhg}
h_g(g_0,g_0)=\frac{\left|A(e^{i\lambda})f_0(\lambda)-
\sum\limits_{j=0}^{\infty}((\bold{B}^0)^{-1}\bold{R}^0\bold{a})_je^{ij\lambda}\right|^2}{(f_0(\lambda)+g_0(\lambda))^2},
\end{equation}
be bounded. In this case the functional
$$
\Delta(h(f_0,g_0);f,g)=\frac{1}{2\pi}\int\limits_{-\pi}^{\pi}h_f(f_0,g_0)f(\lambda)d\lambda + \frac{1}{2\pi}\int\limits_{-\pi}^{\pi}h_g(f_0,g_0)g(\lambda)d\lambda
$$
is linear and continuous on the space $L_1 \times L_1$ and we can apply the method of Lagrange multipliers to find solution to the optimization problem  (\ref{extrDeltaAn2}).
We get the following relations that determine least favourable spectral densities  $f^0\in D^0_f$, $g^0\in D^0_g$
\begin{equation*} \begin{split}
&-\frac{1}{2\pi}\int\limits_{-\pi}^{\pi}h_f(f_0,g_0)\rho(f(\lambda))d\lambda - \frac{1}{2\pi}\int\limits_{-\pi}^{\pi}h_g(f_0,g_0)\rho(g(\lambda))d\lambda\\
&+\alpha_1\frac{1}{2\pi}\int\limits_{-\pi}^{\pi}\rho(f(\lambda))d\lambda+\alpha_2\frac{1}{2\pi}\int\limits_{-\pi}^{\pi}\rho(g(\lambda))d\lambda=0,
\end{split}\end{equation*}
where $\rho(f(\lambda))$ and $\rho(g(\lambda))$ are variations of the functions $f(\lambda)$ and $g(\lambda)$, the constants $\alpha_1\geq 0$ and $\alpha_2\geq 0$.
From this relation we get that the least favourable spectral densities
 $f_0(\lambda) \in D_f^0$, $g_0(\lambda) \in D_g^0$ satisfy equations
\begin{equation} \label{extrAnhf1}
\left|A(e^{i\lambda})g_0(\lambda)+
\sum\limits_{j=0}^{\infty}((\bold{B}^0)^{-1}\bold{R}^0\bold{a})_je^{ij\lambda}\right|=\alpha_1(f_0(\lambda)+g_0(\lambda)),
\end{equation}
\begin{equation} \label{extrAnhg1}
\left|A(e^{i\lambda})f_0(\lambda)-
\sum\limits_{j=0}^{\infty}((\bold{B}^0)^{-1}\bold{R}^0\bold{a})_je^{ij\lambda}\right|=\alpha_2(f_0(\lambda)+g_0(\lambda)).
\end{equation}
Note, that  $\alpha_1\neq 0$ in the case, where
\[\frac{1}{2\pi}\int\limits_{-\pi}^{\pi} f_{0}(\lambda)d\lambda=P_1,\]
and $\alpha_2\neq 0$ in the case, where
\[\frac{1}{2\pi}\int\limits_{-\pi}^{\pi} g_{0}(\lambda)d\lambda=P_2.
\]

Summing up our reasoning we come to conclusion that the following theorem holds true.
\begin{thm}
Let the spectral densities $f_0(\lambda) \in D_f^0$ and $g_0(\lambda) \in D_g^0$ satisfy the minimality condition (\ref{extrminimalAn}) and let the functions $h_f(f_0,g_0)$ and $h_g(f_0,g_0)$, determined by the formulas (\ref{extrAnhf}), (\ref{extrAnhg}), be bounded. The functions  $f_0(\lambda)$ and $g_0(\lambda)$, which give solution to the system of equations (\ref{extrAnhf1}), (\ref{extrAnhg1})
are the least favourable spectral densities in the class $D=D_f^0\times D_g^0$ for the optimal estimation of the functional  $A\xi=\sum_{j=0}^{\infty}a(j)\xi(j)$ which depends on the unknown values of a stationary stochastic sequence $\xi(j)$
based on observations of the sequence $\xi(j)+\eta(j)$ at points of time $j=-1,-2,\dots$, if they determine a solution to the optimization problem (\ref{extrAnDR}).
The function $h^0(e^{i\lambda})$, determined by the formula (\ref{extrspectharAAn}), is minimax-robust spectral characteristic of the optimal linear estimate of the functional $A\xi$.
\end{thm}

\begin{thm}
Let the spectral density $f(\lambda)$ be known and fixed and the spectral density $g_0(\lambda) \in D_g^0$. Let the functions $f(\lambda)$ and $g_0(\lambda)$ be such that the function $(f(\lambda)+g_0(\lambda))^{-1}$ is integrable and let the function $h_g(f,g_0)$, determined by the formula (\ref{extrAnhg}), be bounded.
The spectral density $g_0(\lambda)$ is the least favourable spectral densities in the class $D_g^0$
for the optimal estimation of the functional  $A\xi=\sum_{j=0}^{\infty}a(j)\xi(j)$ which depends on the unknown values of a stationary stochastic sequence $\xi(j)$
based on observations of the sequence $\xi(j)+\eta(j)$ at points of time $j=-1,-2,\dots$,
if it is of the form
\begin{equation*}
g_0(\lambda)= \max\left\{0, \alpha_2^{-1}\left|A(e^{i\lambda})f(\lambda)-
\sum\limits_{j=0}^{\infty}((\bold{B}^0)^{-1}\bold{R}^0\bold{a})_je^{ij\lambda}\right|-f(\lambda)\right\}
\end{equation*}
and the functions $f(\lambda)$ and $ g_0(\lambda)$ determine a solution to the optimization problem (\ref{extrAnDR}).
The function $h^0(e^{i\lambda})$, determined by the formula (\ref{extrspectharAAn}), is minimax-robust spectral characteristic of the optimal linear estimate of the functional $A\xi$.
\end{thm}

\subsection{Least favourable spectral densities in the class $D_v^u \times D_{\varepsilon}$}

Consider the problem for the optimal estimation of the functional  $A\xi=\sum_{j=0}^{\infty}a(j)\xi(j)$ which depends on the unknown values of a stationary stochastic sequence $\xi(j)$
based on observations of the sequence $\xi(j)+\eta(j)$ at points of time $j=-1,-2,\dots$
in the case where the spectral densities  $f(\lambda)$, $g(\lambda)$ are from the class $D=D_v^u \times D_{\varepsilon}$, where
\[
D_v^u = \left\{f(\lambda)\left| v(\lambda)\leq f(\lambda)\leq u(\lambda),\frac{1}{2\pi}\int\limits_{-\pi}^{\pi} f(\lambda)d\lambda\leq P_1\right.\right\},
\]
\[ D_\varepsilon = \left\{g(\lambda)\left| g(\lambda)=(1-\varepsilon)g_1(\lambda)+\varepsilon w(\lambda),\frac{1}{2\pi}\int\limits_{-\pi}^{\pi} g(\lambda)d\lambda\leq P_2\right.\right\}.
\]
Here the spectral densities $v(\lambda), u(\lambda), g_1(\lambda)$ are known and fixed and the densities $v(\lambda)$ and $u(\lambda)$ are bounded.

Let the densities $f^0(\lambda) \in D_v^u $, $g^0(\lambda) \in D_{\varepsilon}$ determine bounded functions  $h_f(f_0,g_0)$, $h_g(f_0,g_0)$ with the help of formulas  (\ref{extrAnhf}), (\ref{extrAnhg}).
 Then from the condition $0 \in \partial\Delta_{D_{f,g}}(f_0,g_0)$ we derive the following equations that determine the least favourable spectral densities
\begin{equation} \label{extrAnhf2}
\left|A(e^{i\lambda})g_0(\lambda)+
\sum\limits_{j=0}^{\infty}((\bold{B}^0)^{-1}\bold{R}^0\bold{a})_je^{ij\lambda}\right|=(f_0(\lambda)+g_0(\lambda))(\gamma_1(\lambda)+\gamma_2(\lambda)+\alpha_1),
\end{equation}
\begin{equation} \label{extrAnhg2}
\left|A(e^{i\lambda})f_0(\lambda)-
\sum\limits_{j=0}^{\infty}((\bold{B}^0)^{-1}\bold{R}^0\bold{a})_je^{ij\lambda}\right|=(f_0(\lambda)+g_0(\lambda))(\varphi(\lambda)+\alpha_2),
\end{equation}
where $\gamma_1\leq 0$ and $\gamma_1= 0$ in the case $f_0(\lambda)\geq v(\lambda)$; $\gamma_2\geq 0$ and $\gamma_2= 0$ in the case $f_0(\lambda)\leq u(\lambda)$;
 $\varphi (\lambda)\leq 0$ and $\varphi(\lambda)= 0$ in the case $g_0(\lambda)\geq (1-\varepsilon) g_1(\lambda)$.

The following theorems hold true.

\begin{thm}
Let the spectral densities $f_0(\lambda) \in D_v^u $ and $g_0(\lambda) \in D_{\varepsilon}$ satisfy the minimality condition (\ref{extrminimalAn}) and let the functions $h_f(f_0,g_0)$ and $h_g(f_0,g_0)$, determined by the formulas (\ref{extrAnhf}), (\ref{extrAnhg}), be bounded.
The functions  $f_0(\lambda), g_0(\lambda)$, which give solution to the system of equations (\ref{extrAnhf2}), (\ref{extrAnhg2})
are the least favourable spectral densities in the class $D_v^u \times D_{\varepsilon}$, if they determine a solution to the optimization problem (\ref{extrAnDR}).
The function $h^0(e^{i\lambda})$, determined by the formula (\ref{extrspectharAAn}), is minimax-robust spectral characteristic of the optimal linear estimate of the functional $A\xi$.
 \end{thm}

\begin{thm}
Let the spectral density $f(\lambda)$ be known and fixed and let the spectral density $g_0(\lambda) \in D_{\varepsilon}$. Let the function $f(\lambda)+g_0(\lambda)$  satisfy the minimality condition (\ref{extrminimalAn}), and let the function $h_g(f,g_0)$, determined by the formula (\ref{extrAnhg}), be bounded.
The spectral density $g_0(\lambda)$ is the least favourable spectral densities in the class $D_{\varepsilon}$ for the optimal linear estimate of the functional $A\xi$, based on observations of the sequence $\xi(j)+\eta(j)$ at points of time $j=-1,-2,\dots$,
if it is of the form
\begin{equation*}
g_0(\lambda)= \max\left\{(1-\varepsilon)g_1(\lambda), \alpha_2\left|A(e^{i\lambda})f(\lambda)-\sum\limits_{j=0}^{\infty}((\bold{B}^0)^{-1}\bold{R}^0\bold{a})_je^{ij\lambda}
\right|-f(\lambda)\right\}
\end{equation*}
and the functions $f(\lambda), g_0(\lambda)$ determine a solution to the optimization problem (\ref{extrAnDR}).
The function $h^0(e^{i\lambda})$, determined by the formula (\ref{extrspectharAAn}), is minimax-robust spectral characteristic of the optimal linear estimate of the functional $A\xi$.
\end{thm}

\subsection{Least favourable spectral densities in the class $D_{2\varepsilon_1}\times D_{1\varepsilon_2}$}

Consider the problem for the optimal estimation of the functional  $A\xi=\sum_{j=0}^{\infty}a(j)\xi(j)$ which depends on the unknown values of a stationary stochastic sequence $\xi(j)$
based on observations of the sequence $\xi(j)+\eta(j)$ at points of time $j=-1,-2,\dots$
in the case where the spectral densities  $f(\lambda)$, $g(\lambda)$ are from the class $D_{2\varepsilon_1}\times D_{1\varepsilon_2}$, which describe the model of
 "$\varepsilon$-neighbourhood" of spectral densities in the space $L_2 \times L_1$. Let
\begin{equation*}
D_{2\varepsilon_1}=\left\{f(\lambda)\left|\frac{1}{2\pi}\int \limits_{-\pi}^{\pi}\left|f(\lambda)-f_1(\lambda)\right|^2d\lambda\leq\varepsilon_1\right.\right\}
\end{equation*}
be the ``$\varepsilon$-neighbourhood'' in the space  $L_2$ of a given bounded spectral density $f_1(\lambda)$ and let
\begin{equation*}
D_{1\varepsilon_2}=\left\{g(\lambda)\left|\frac{1}{2\pi}\int \limits_{-\pi}^{\pi}\left|g(\lambda)-g_1(\lambda)\right|d\lambda\leq\varepsilon_2\right.\right\}
\end{equation*}
be the ``$\varepsilon$-neighbourhood'' in the space  $L_1$ of a given bounded spectral density  $g_1(\lambda)$.

Let the spectral densities  $f_0(\lambda) \in D_{2\varepsilon_1}$, $g_0(\lambda) \in D_{1\varepsilon_2}$ determine the bounded functions  $h_f(f_0,g_0)$, $h_g(f_0,g_0)$ with the help of formulas  (\ref{extrAnhf}), (\ref{extrAnhg}).
 Then from the condition $0 \in \partial\Delta_{D_{f,g}}(f_0,g_0)$ for $D=D_{2\varepsilon_1}\times D_{1\varepsilon_2}$ we derive the following equations that determine the least favourable spectral densities
\begin{equation} \label{extrAnhf3}
\left|A(e^{i\lambda})g_0(\lambda)+\sum\limits_{j=0}^{\infty}((\bold{B}^0)^{-1}\bold{R}^0\bold{a})_je^{ij\lambda}
\right|^2=(f_0(\lambda)+g_0(\lambda))^2(f_0(\lambda)-f_1(\lambda))\alpha_1,
\end{equation}
\begin{equation} \label{extrAnhg3}
\left|A(e^{i\lambda})f_0(\lambda)-\sum\limits_{j=0}^{\infty}((\bold{B}^0)^{-1}\bold{R}^0\bold{a})_je^{ij\lambda}
\right|^2=(f_0(\lambda)+g_0(\lambda))^2\Psi(\lambda)\alpha_2,
\end{equation}
where $\left|\Psi(\lambda)\right|\leq 1$ and $\Psi(\lambda)=sign (g_0(\lambda)-g_1(\lambda))$, in the case $g_0(\lambda)\neq g_1(\lambda)$, $\alpha_1, \alpha_2$ are constants.

Equations (\ref{extrAnhf3}), (\ref{extrAnhg3}) with the optimization problem (\ref{extrAnDR}) and the normalizing conditions
\begin{equation} \label{extrAn1}
\frac{1}{2\pi}\int \limits_{-\pi}^{\pi}\left|f(\lambda)-f_1(\lambda)\right|^2d\lambda=\varepsilon_1
\end{equation}
\begin{equation} \label{extrAn2}
\frac{1}{2\pi}\int \limits_{-\pi}^{\pi}\left|g(\lambda)-g_1(\lambda)\right|d\lambda=\varepsilon_2
\end{equation}
determine the least favourable spectral densities in the class $D=D_{2\varepsilon_1}\times D_{1\varepsilon_2}$.

The following theorems hold true.

\begin{thm}
Let the spectral densities $f_0(\lambda) \in D_{2\varepsilon_1}$, $g_0(\lambda) \in D_{1\varepsilon_2}$ satisfy the minimality condition (\ref{extrminimalAn}) and let the functions $h_f(f_0,g_0)$ and $h_g(f_0,g_0)$, determined by the formulas (\ref{extrAnhf}), (\ref{extrAnhg}), be bounded.
The spectral densities  $f_0(\lambda), g_0(\lambda)$, which give solution to the system of equations (\ref{extrAnhf3})--(\ref{extrAn2})
are the least favourable spectral densities in the class $D_{2\varepsilon_1}\times D_{1\varepsilon_2}$, if they determine a solution to the optimization problem (\ref{extrAnDR}).
The function $h^0(e^{i\lambda})$, determined by the formula (\ref{extrspectharAAn}), is minimax-robust spectral characteristic of the optimal linear estimate of the functional $A\xi$.
\end{thm}

\begin{thm}
Let the spectral density $f(\lambda)$ be known, and let the spectral density $g_0(\lambda) \in D_{1\varepsilon_2}$. Let the function $f(\lambda)+g_0(\lambda)$  satisfy the minimality condition (\ref{extrminimalAn}), and let the function $h_g(f,g_0)$, determined by the formula (\ref{extrAnhg}), be bounded.
The spectral density $g_0(\lambda)$ is the least favourable spectral densities in the class $D_{1\varepsilon_2}$ for the optimal linear estimate of the functional $A\xi$, based on observations of the sequence $\xi(j)$ at points of time $j=-1,-2,\dots$,
if it is of the form
\begin{equation*}
g_0(\lambda)= \max\left\{g_1(\lambda), \alpha_2\left|A(e^{i\lambda})f(\lambda)-\sum\limits_{j=0}^{\infty}((\bold{B}^0)^{-1}\bold{R}^0\bold{a})_je^{ij\lambda}
\right|-f(\lambda)\right\}
\end{equation*}
and the functions $f(\lambda), g_0(\lambda)$ determine a solution to the optimization problem (\ref{extrAnDR}).
The function $h^0(e^{i\lambda})$, determined by the formula (\ref{extrspectharAAn}), is minimax-robust spectral characteristic of the optimal linear estimate of the functional $A\xi$.
\end{thm}

\subsection{Conclusions}

In this section  we propose methods of solution of the problem of the mean-square optimal linear estimation of the functional $A\xi=\sum\limits_{j=0}^{\infty}a(j)\xi(j)$ which depends on the unknown values of the stationary stochastic sequence $\xi(j)$.
Estimates are based on observations of the sequence $\xi(j)+\eta(j)$ at points  of time $j=-1,-2,\dots$. Here $\eta(j)$ is an uncorrelated with $\xi(j)$ stationary sequence. We provide formulas for calculating the values of the mean square error and the spectral characteristic of the optimal linear estimate of the functional in the case of spectral certainty, where the spectral densities $f(\lambda)$ and $g(\lambda)$ of the sequences $\xi(j)$ and  $\eta(j)$ are exactly known.
In the case of spectral uncertainty, where the spectral densities $f(\lambda)$ and $g(\lambda)$ are not known, but a set of admissible spectral densities is given, the minimax approach is applied to estimation of the functionals. We obtain formulas that determine the least favourable spectral densities and the minimax spectral characteristics of the optimal linear estimates of the functional $A\xi$ for concrete classes of admissible spectral densities.

For the relative results on the mean-square optimal linear extrapolation of linear functionals for stationary stochastic sequences and processes based on observations with noise see papers by Moklyachuk~\cite{Moklyachuk:1993a} -- \cite{Moklyachuk:1993r}, \cite{Moklyachuk:1997} -- \cite{Moklyachuk:2008r}, book by Moklyachuk and Masyutka~\cite{Moklyachuk:2012}.

\section{Interpolation problem for stationary sequences}

In this section we deal with the problem of the mean-square optimal estimation of the linear
functional $A_N\xi=\sum\limits_{j=0}^{N}a(j)\xi(j)$ which depends on the unknown values of a stationary stochastic sequence $\xi(j),j\in \mathbb{Z},$  based on observations of the sequence at points of time $j\in\mathbb{Z}\backslash \{0,1,\dots,N\}$.
The problem is investigated in the case of spectral certainty, where the spectral density of the stationary stochastic sequence $\xi(j)$ is exactly known.
In this case the classical Hilbert space projection method of linear estimation of the functional is applied. Formulas are derived for calculation the value of the mean square error and the spectral characteristic of the mean-square optimal estimate of the linear functional.
In the case of spectral uncertainty, where the spectral density of the stationary stochastic sequence is not exactly known, but a class of admissible spectral densities is given the minimax-robust
procedure to linear estimation of the functional is applied.
Relations which determine the least favourable spectral densities and the minimax spectral characteristics are proposed for some special sets of admissible spectral densities.

\subsection{The classical Hilbert space projection method of linear interpolation}

Let $\xi(j),\,j\in \mathbb{Z}$,  be a (wide sense) stationary stochastic sequence.
We will consider values of $\xi(j),\,j\in \mathbb{Z}$, as elements of the Hilbert space  $H=L_2(\Omega,\mathcal{F},P)$ of complex valued random variables with zero first moment, $E\xi=0$, finite second moment,
 $E|\xi|^2<\infty$, and the inner product $\left(\xi,{\eta}\right)=E\xi\overline{\eta}$.
The correlation function $R(k)=\left(\xi(j+k),\xi(j)\right)=E\xi(j+k)\overline{\xi(j)}$ of the stationary stochastic sequence $\xi(j),\,j\in \mathbb{Z}$, admits
 the spectral representation \cite{Gihman}
\[
 R(k)=\int\limits_{-\pi}^{\pi}e^{ik\lambda}F(d\lambda),
 \]
where $F(d\lambda)$ is the spectral measure of the sequence.
We will consider stationary stochastic sequences with absolutely continuous spectral measures and the correlation functions of the form
\[
 R(k)=\,\frac{1}{2\pi}\int\limits_{-\pi}^{\pi}e^{ik\lambda}f(\lambda)d\lambda,
 \]
where $f(\lambda)$ is the spectral density function of the sequence $\xi(j)$
that satisfies the minimality condition
\begin {equation} \label{min}
 \int\limits_{-\pi}^{\pi} f^{-1}(\lambda)d\lambda < \infty.
\end{equation}

Under this condition the error-free interpolation of the unknown values of the sequence is impossible \cite{Rozanov}.

The stationary stochastic sequence $\xi(j),\,j\in \mathbb{Z}$, admits
 the spectral representation \cite{Gihman,Karhunen}
\begin {equation} \label{sp-r}
 \xi(j)=\int\limits_{-\pi}^{\pi}e^{ij\lambda}{Z}(d\lambda),
\end{equation}
where ${Z}(\Delta)$ is the orthogonal stochastic measure of the sequence such that
\[
E{Z}(\Delta_1)\overline{{Z}(\Delta_2)}=F(\Delta_1\cap\Delta_2)=\,\frac{1}{2\pi}\int_{\Delta_1\cap\Delta_2}f(\lambda)d\lambda.
\]

Consider the problem of the mean-square optimal estimation of the linear
functional
\[A_N\xi=\sum\limits_{j=0}^{N}a(j)\xi(j)\]
 which depends on the unknown values of a stationary stochastic sequence $\xi(j),j\in\mathbb{Z},$  from observations of the sequence at points of time $j\in\mathbb{Z}\backslash \{0,1,\dots,N\}$.

It follows from the spectral representation (\ref{sp-r}) of the sequence $\xi(j)$
that we can represent the functional $A_{N}\xi$ in the  form
\begin{equation}\label{asx}
A_{N}\xi=\int\limits_{-\pi}^{\pi}A_{N}(e^{i\lambda}){Z}(d\lambda),
 \end{equation}
where
\begin{equation*}
  A_N(e^{i\lambda})=\sum\limits_{j=0}^{N}a(j)e^{ij\lambda }.
 \end{equation*}

Denote by $H^N(\xi)$ the subspace  of the Hilbert space  $H=L_2(\Omega,\mathcal{F},P)$ generated by elements $\{\xi(j):j\in \mathbb{Z}\backslash \{0,1,\dots,N\}$. Let $L_2(f)$ be the Hilbert space of complex-valued functions  that are square-integrable  with respect to the measure whose density is $f(\lambda).$ Denote by $L_2^N(f)$ the subspace of $L_2(f)$ generated by functions $\{e^{ij\lambda},\,j\in \mathbb{Z}\backslash \{0,1,\dots,N\} \}.$
The mean square optimal linear estimate $\hat{A}_N\xi$ of the functional  $A_N\xi$ from observations of the sequence \(\xi(j)\) at points of time  $j\in\mathbb{Z} \{0,1,\dots,N\}$
is an element of the $H^N(\xi)$. It can be represented in the form
\begin{equation}\label{hatas}
\hat{A}_N\xi=\,\int\limits_{-\pi}^{\pi}h(e^{i\lambda}) {Z}(d\lambda),
 \end{equation}
where $h(e^{i\lambda}) \in L_2^N(f)$  is the spectral characteristic of the estimate $\hat{A}_N\xi.$

The mean square error $\Delta(h;f)$ of the estimate $\hat{A}_N\xi$ is given by the formula
\begin{equation*}
 \Delta(h;f)=E\left|A_N\xi-\hat{A}_N\xi\right|^2=\,\frac{1}{2\pi}\int\limits_{-\pi}^{\pi}\left|A_N(e^{i\lambda})-h(e^{i\lambda})\right|^2 f(\lambda)d\lambda.
\end{equation*}

The Hilbert space projection method proposed by A. N. Kolmogorov \cite{Kolmogorov} makes it possible to find the spectral characteristic  $h(e^{i\lambda})$ and the mean square error $\Delta(h;f)$ of the optimal linear estimate of the functional  $A_N\xi$ in the case where the spectral density $f(\lambda)$ of the sequence $\xi(j),j\in\mathbb{Z},$ is exactly known and the minimality condition (\ref{min}) is satisfied. The spectral characteristic can be found from the following conditions:
\begin{equation*} \begin{split}
1)& h(e^{i\lambda}) \in L_2^N(f), \\
2)& A_N(e^{i\lambda})- h(e^{i\lambda}) \bot \, L_2^N(f).
\end{split} \end{equation*}

It follows from the second condition that for any $\eta \in H^N(\xi)$ the following equation should be satisfied
\begin{equation*}
\left( A_N\xi-\hat{A}_N\xi,\eta\right)=E(A_N\xi-\hat{A}_N\xi)\overline{\eta} =0.
 \end{equation*}

 The last relation is equivalent to equations
\begin{equation*}
E(A_N\xi-\hat{A}_N\xi)\overline{\xi}_k =0, \,\,\, k \in \mathbb{Z} \backslash \{0,1,\dots,N\}.
 \end{equation*}

By using representations (\ref{asx}), (\ref{hatas}) and definition of the inner product in  the Hilbert space  $H=L_2(\Omega,\mathcal{F},P)$  we get
\begin{equation*}
\begin{split}
E\left(\int\limits_{-\pi}^{\pi}\left(A_N(e^{i\lambda})- h(e^{i\lambda})\right)Z(d\lambda)\cdot\int\limits_{-\pi}^{\pi}e^{-ik\lambda}\overline{Z(d\lambda)}\right)=\\
= \,\frac{1}{2\pi}\int\limits_{-\pi}^{\pi} \left(A_N(e^{i\lambda})- h(e^{i\lambda})\right)f(\lambda)e^{-ik\lambda }d\lambda=0,\,\,k \in \mathbb{Z} \backslash \{0,1,\dots,N\}.
 \end{split}
 \end{equation*}
It follows from these equations that the function $(A_N(e^{i\lambda})- h(e^{i\lambda}))f(\lambda)$ is of the form
\begin{equation}\label{A_N}
(A_N(e^{i\lambda})- h(e^{i\lambda}))f(\lambda)=C_N(e^{i\lambda}),
 \end{equation}
\[C_N(e^{i\lambda})=\sum\limits_{j=0}^{N}c(j)e^{ij\lambda},
\]
where $c(j), j=0,1,\dots,N,$ are unknown coefficients that we have to find.

 From the relation (\ref{A_N}) we deduce that the spectral characteristic $ h(e^{i\lambda})$
 of the optimal linear estimate of the functional  $A_N\xi$
 is of the form
\begin{equation} \label {5}
 h(e^{i\lambda})=\,A_N(e^{i\lambda})-C_N(e^{i\lambda})f^{-1}(\lambda).
\end{equation}

To find equations for determination the unknown coefficients $c(j), j=0,1,\dots,N,$ we use the decomposition of the function $f^{-1}(\lambda)$ into the Fourier series
\begin{equation} \label{decomp}
f^{-1}(\lambda)=\sum\limits_{m=-\infty}^{\infty}r_m e^{im\lambda},
\end{equation}
where $r_m$ are the Fourier coefficients of the function $f^{-1}(\lambda)$
\[r_m=\,\frac{1}{2\pi}\int\limits_{-\pi}^{\pi} f^{-1}(\lambda)e^{-im\lambda }d\lambda.
\]

\noindent Inserting (\ref{decomp}) into (\ref{5}) we obtain the following representation of the spectral characteristic
\begin{equation} \label{spchar}
\begin{split}
 h(e^{i\lambda})=\left(\sum\limits_{j=0}^{N}a(j)e^{ij\lambda }\right)-\left(\sum\limits_{j=0}^{N}c(j)e^{ij\lambda }\right)\left(\sum\limits_{m=-\infty}^{\infty}r_m e^{im\lambda}\right).
\end{split} \end{equation}

It follows from the first condition, $h(e^{i\lambda}) \in L_2^N(f)$,  which determines the spectral characteristic,
that the Fourier coefficients of the function  $h(e^{i\lambda})$ are equal to zero for $j=0,1,\dots,N$, namely
\[\frac{1}{2\pi}\int\limits_{-\pi}^{\pi} h(e^{i\lambda})e^{-ij\lambda }d\lambda=0, \,\, j =0,1,\dots,N.
\]

Using the last relations and (\ref{spchar}) we get the following system of equations that determine the unknown coefficients  $c(j), j=0,1,\dots,N,$
\begin{equation} \label{koef}
\begin{split}
a(0)-&\sum\limits_{j=0}^{N}c(j)r_{-j}=0;\\
a(1)-&\sum\limits_{j=0}^{N}c(j)r_{1-j}=0;\\
&\ldots\\
a(N)-&\sum\limits_{j=0}^{N}c(j)r_{N-j}=0.\\
\end{split} \end{equation}

Denote by
$\vec{\bold{a}}_N=(a(0),a(1),\ldots,a(N))$ and let  $\vec{\bold{c}}_N=(c(0),c(1),\ldots, c(N))$ be a vector constructed from the unknown coefficients $c(j), j=0,1,\dots,N.$
Let
$B_N$  be an  $(N+1)\times(N+1)$  matrix
\begin {equation*}
B_N= \left( \begin{array}{cccc}
B_{00}& B_{01}&\ldots&B_{0s}\\
B_{10}& B_{11}&\ldots&B_{1s}\\
\vdots&\vdots&\ddots &\vdots\\
B_{N0}& B_{N1}&\ldots&B_{NN}
\end{array}\right),
\end{equation*}
 with elements that are Fourier coefficients of the function $ f^{-1}(\lambda): $
\begin {equation*}
  B_{kj}=\,\frac{1}{2\pi}\int\limits_{-\pi}^{\pi}f^{-1}(\lambda)e^{-i(k-j)\lambda}d\lambda=\,r_{k-j},\quad
k= 0,1, \ldots,N, \,\,
j= 0,1, \ldots,N.
\end{equation*}

Making use the introduced notations we can write equations (\ref{koef}) in the form of equation
 \begin{equation} \label{rivn}
 \vec{\bold{a}}_N=B_N\vec{\bold{c}}_N,
 \end{equation}

Since the matrix $B_N$ is reversible \cite{Salehi}, we get the formula
\begin{equation} \label{rivn2}
\vec{\bold{c}}_N=B_N^{-1}\vec{\bold{a}}_N.
\end{equation}

 Hence, the unknown coefficients $c(j), j=0,1,\dots,N,$ are calculated  by the formula
\[c(j)=\left(B_N^{-1}\vec{\bold{a}}_N\right)_j,
\]
where  $\left(B_N^{-1}\vec{\bold{a}}_N\right)_j $ is the  $j$  component of the vector  $B_N^{-1}\vec{\bold{a}}_N,$
and the formula for calculating  the spectral characteristic of the estimate $\hat{A}_N\xi$ is of the form
\begin{equation} \label{ch}
\begin{split}
 h(e^{i\lambda})=\left(\sum\limits_{j=0}^{N}a(j)e^{ij\lambda }\right)-\left(\sum\limits_{j=0}^{N}\left(B_N^{-1}\vec{\bold{a}}_N\right)_je^{ij\lambda }\right)\left(\sum\limits_{m=-\infty}^{\infty}r_m e^{im\lambda}\right).
\end{split} \end{equation}

The mean square error of the estimate of the function can be calculated by the formula
\begin {equation} \label {6} \begin{split}
 \Delta(h;f)&=\frac{1}{2\pi}\int\limits_{-\pi}^{\pi} \left|C_N(e^{i\lambda})\right|^2 f^{-1}(\lambda)d\lambda=\\
&=\int\limits_{-\pi}^{\pi}\left(\sum\limits_{k=0}^{N}c(k) e^{ik\lambda}\right)\left(\sum\limits_{j=0}^{N}\overline{c(j)} e^{-ij\lambda}\right)\left(\sum\limits_{m=-\infty}^{\infty}r_m e^{im\lambda}\right)d\lambda=\\
&=\left<\vec{\bold{c}}_N,\vec{\bold{c}}_N B_N\right>=\left<B_N^{-1}\vec{\bold{a}}_N, \vec{\bold{a}}_N\right>,
\end{split} \end{equation}
where $\left<\cdot,\cdot\right> $\,is the inner product in $ C^{(N+1)}. $

\noindent Let us summarize our results and present them in the form
of a theorem.

\begin{thm}
Let $\xi(j)$ be a stationary stochastic sequence with the spectral density $f(\lambda)$ that satisfies the minimality condition (\ref{min}). The spectral characteristic $h(e^{i\lambda})$
and the mean square error  $\Delta(h,f)$
of the optimal linear estimate $\hat{A}_N\xi$ of the functional $A_N\xi$ based on observations of the sequence $\xi(j)$ at points of time $j\in\mathbb{Z}\backslash \{0,1,\dots,N\} $ can be calculated by  formulas (\ref{ch}), (\ref{6}).
\end{thm}

\begin{exm}
Consider the problem of linear estimation of the functional  $A_1\xi=a\xi(0)+b\xi(1)$ which depends on the unknown values $\xi(0)$, $\xi(1)$ of the stationary stochastic sequence $\xi(j)$ from observations of the
sequence $\xi(j)$
at points $j\in \mathbb{Z}\backslash \{0,1\}.$ In this case the spectral characteristic (\ref{spchar}) of the estimate $\hat{A}_1 \xi$ can be calculated by the formula
\begin{equation}\label{pr1}
h(f)=a+be^{i\lambda}-\left(c(0)+c(1)e^{i\lambda}\right)f^{-1}(\lambda),
\end{equation}
where $f(\lambda)$ is a known spectral density, coefficients $c(0), c(1)$ are determined by the system of equations
\begin{eqnarray*}
a=c(0)\alpha+c(1)\beta, \\
b=c(0)\bar\beta+c(1)\alpha, \\
\end{eqnarray*}
where
\[\alpha={\frac{1}{2\pi}}\int_{\pi}^{\pi} f^{-1}(\lambda)\,d\lambda,
\]
\[\beta={\frac{1}{2\pi}}\int_{\pi}^{\pi}e^{i\lambda}f^{-1}(\lambda)\,d\lambda,
\]

The matrix $B_1$ is of the form
\begin {equation*}
B_1= \left( \begin{array}{cc}
\alpha& \beta\\
\bar\beta& \alpha\
\end{array}\right),
\end{equation*}

The determinant $ D=det(B_1)$ of the matrix $B_1$ is as follows $ D=det(B_1)=\alpha^2-|\beta|^2$

We get the following formulas for calculating the coefficients $c(0), c(1)$
\[c(0)=(a\alpha-b\beta)/(\alpha^2-|\beta|^2),\]
\[c(1)=(b\alpha-a\bar\beta) /(\alpha^2-|\beta|^2).\]

Thus, the unknown coefficients $c(0), c(1)$ in (\ref{pr1}) are determined.

The mean square error of the estimate is calculated by the formula
\begin{equation}
\Delta(f)=\left[\left(|a|^2+|b|^2\right)\alpha-
                    (\bar ab+a\bar b)\beta\right]/(\alpha^2-|\beta|^2).
\end{equation}

\end{exm}

\subsection{Minimax-robust method of interpolation}

The traditional methods of estimation of the functional $A_N\xi$ which depends on the unknown values of a stationary stochastic sequence $\xi(j)$ can be applied in the case  where the spectral density $f(\lambda)$ of the considered stochastic sequence $\xi(j)$ is exactly known.
In practice, however, we do not have complete information on spectral density of the sequence. For this reason  we apply the minimax(robust) method of estimation of the functional $A_N\xi$, that is we find an estimate that minimizes the maximum of the mean square errors for all spectral densities from a given class of admissible spectral densities $D$.

\begin{ozn}
For a given class of spectral densities $D$ a spectral density $f_0(\lambda)\in D$  is called the least favourable in $D$ for the optimal linear estimation of the functional $A_N\xi$ if the following relation holds true
$$\Delta\left(f_0\right)=\Delta\left(h\left(f_0\right);f_0\right)=\max\limits_{f\in D}\Delta\left(h\left(f\right);f\right).$$
\end{ozn}

\begin{ozn}
For a given class of spectral densities $D$ the spectral characteristic $h^0(e^{i\lambda})$ of the optimal linear estimate of the functional $A_N\xi$ is called minimax-robust if
$$h^0(e^{i\lambda})\in H_D= \bigcap\limits_{f\in D} L_2^N(f),$$
$$\min\limits_{h\in H_D}\max\limits_{f\in D}\Delta\left(h;f\right)=\sup\limits_{f\in D}\Delta\left(h^0;f\right).$$
\end{ozn}

It follows from the introduced definitions and the obtained formulas that the following statement holds true.

\begin{lem} The spectral density $f_0(\lambda)\in D$ is the least favourable in a class of admissible spectral densities  $D$ for the optimal linear estimation of the functional $A_N\xi$ if the Fourier coefficients of the function $f_0^{-1}(\lambda)$ define a matrix $B_N^0$ that determines a solution to the optimization problem
\begin{equation} \label{extrem}
\max\limits_{f\in D}\left<B_N^{-1}\vec{\bold{a}}_N,\,\vec{\bold{a}}_N\right>=\left<(B_N^0)^{-1}\vec{\bold{a}}_N,\,\vec{\bold{a}}_N\right>.
\end{equation}
The minimax spectral characteristic $h^0=h(f_0)$ can be calculated by the formula (\ref{ch}) if $h(f_0) \in H_D.$
\end{lem}

The least favourable spectral density $f_0$ and the minimax spectral characteristic $h^0$
form a saddle point of the function $\Delta \left(h;f\right)$ on the set $H_D\times D$. The saddle point inequalities
$$\Delta\left(h;f_0\right)\geq\Delta\left(h^0;f_0\right)\geq \Delta\left(h^0;f\right)   \hspace{1cm}   \forall f \in D, \forall h \in H_D$$
hold true if $h^0=h(f_0)$ and $h(f_0)\in H_D,$  where $f_0$ is a solution to the constrained optimization problem
\begin{equation} \label{7}
\tilde{\Delta}(f)=-\Delta\left(h^0;f\right)=-\frac{1}{2\pi}\int\limits_{-\pi}^{\pi}\frac{\left|C_N^0(e^{i\lambda})\right|^2}{f_0^2(\lambda)} f(\lambda)d\lambda\rightarrow \inf, \hspace{1cm} f(\lambda) \in D,
\end{equation}
where
\[C_N^0(e^{i\lambda})=\,\sum\limits_{j=0}^{N}\left((B_N^0)^{-1}\vec{\bold{a}}_N\right)_j e^{ij\lambda}.
\]

The constrained optimization problem (\ref{7}) is equivalent to the unconstrained optimization problem
\begin{equation*}
\Delta_D(f)=\tilde{\Delta}(f)+\delta(f\left|D\right.)\rightarrow \inf,
\end{equation*}
where $\delta(f\left|D\right.)$ is the indicator function of the set  $D.$ Solution $f_0$ to this problem is characterized by the condition $0 \in \partial\Delta_D(f_0),$ where $\partial\Delta_D(f_0)$ is the subdifferential of the convex functional $\Delta_D(f)$ at point $f_{0}$.
This condition makes it possible to find the least favourable spectral densities in some special classes of spectral densities $D$ \cite{Ioffe}, \cite{Pshenychn}, \cite{Rockafellar}.

Note, that the form of the functional $\Delta\left(h^0;f\right)$ is convenient for application the Lagrange method of indefinite multipliers for
finding solution to the problem (\ref{7}).
Making use the method of Lagrange multipliers and the form of
subdifferentials of the indicator functions
we describe relations that determine least favourable spectral densities in some special classes
of spectral densities
(see books \cite{Golichenko,Moklyachuk:2008,Moklyachuk:2012} for additional details).

\subsection{Least favourable spectral densities in the class $D_0^{-}$}

Consider the problem of the optimal estimation of the functional $A_N\xi=\sum_{j=1}^{N}a(j)\xi(j)$ which depends on the unknown values of a stationary stochastic sequence $\xi(j)$
from observations of the sequence $\xi(j)$ at points of time $j\in\mathbb{Z}\backslash \{0,1,\dots,N\} $
in the case where the spectral density is from the class
\begin{equation*}
D_0^{-} = \left\{f(\lambda)\left|\frac{1}{2\pi}\int\limits_{-\pi}^{\pi} f^{-1}(\lambda)d\lambda\geq P\right.\right\}.
\end{equation*}
Let the sequence $a(k), k=0,1,\dots,N,$ that determines the functional  $A_N\xi$, be strictly positive. To find solutions to the constrained optimization problem (\ref{7}) we
use the  Lagrange multipliers method. With the help of this method we get the equation
\begin{equation*}
\frac{1}{2\pi}\int\limits_{-\pi}^{\pi}\left[\frac{\left|C_N^0(e^{i\lambda})\right|^2}{f_0^2(\lambda)} - p_0^2 \frac{1}{f_0^{2}(\lambda)}\right]\rho (f(\lambda))d\lambda=0,
\end{equation*}
where $p_0^2$ is a  constant (the Lagrange multiplier), $\rho(f(\lambda))$ is a variation of the function  $f(\lambda)$.
From a generalization of the Lagrange lemma we get that the Fourier coefficients of the function  $f_0^{-1}$ satisfy the equation
\begin{equation} \label{99}
\left|\sum\limits_{k=0}^{N}c(k) e^{ik\lambda}\right|^2=p_0^2,
\end{equation}
where $ c(k), k=0,1,\dots,N,$ are components of the vector $\vec{\bold{c}}_N$ that satisfies the equation
 \begin{equation} \label{999}
 B_N^0\vec{\bold{c}}_N=\vec{\bold{a}}_N,
 \end{equation}
 the matrix  $B_N^0$  is  determined by the Fourier coefficients of the function $f_0^{-1}(\lambda)$
\begin{equation*}
B_{N}^0(k,j)=\,\frac{1}{2\pi}\int\limits_{-\pi}^{\pi}f_0^{-1}(\lambda)e^{-i(k-j)\lambda}d\lambda=r_{k-j}^0,
\end{equation*}
\begin {equation*} \begin{split}
&k= 0,1, \ldots,N, \\
&j= 0,1, \ldots,N. \\
\end{split}\end{equation*}

The Fourier coefficients $r_k=r_{-k}, k=0,1,\dots,N,$ satisfy both equation  (\ref{99}) and equation   (\ref{999}). These coefficients can be found from the equation
\[B_N^0\vec{\bold{p}}_N^0=\vec{\bold{a}}_N, \quad \vec{\bold{p}}_N^0=(p_0,0,\ldots,0).
\]
The last relation can be presented in the form of the system of equations
\begin{equation*}
r_kp_0=a(k), \, k=0,1,\dots,N.
\end{equation*}

From the first equation of the system (for $k=0$) we find the unknown value
\[p_0=a(0) (r_0)^{-1}.\]
It follows from the extremum condition  (\ref{extrem}) and the restriction on the spectral densities from the class $D_0^{-}$  that the Fourier coefficient \[r_0=\frac{1}{2\pi}\int\limits_{-\pi}^{\pi}f_0^{-1}(\lambda)d\lambda=P.\]
 Thus,
 \[r_k=Pa(k)a^{-1}(0), \,\, k=0,1,\dots,N.\]
 We can represent the function $f_0^{-1}(\lambda)$ in the form
 \[f_0^{-1}(\lambda)=\sum_{k=-N}^{N}r_{|k|} e^{ik\lambda}.
 \]
  Since the sequence $a(k), k=0,1,\dots,N,$ is strictly positive, the sequence  $r_k,  k= 0,1,\ldots, N$, is also strictly positive and the function  $f_0^{-1}(\lambda)$ is positive, so it can be represented in the form \cite{Krein}
 \[f_0^{-1}(\lambda)=\left|\sum\limits_{k=0}^{N}\gamma_k e^{-ik\lambda}\right|^2,\,\,  \lambda \in \left[-\pi, \pi\right],
 \]
Hence, $f_0(\lambda)$ is the spectral density of the autoregressive stochastic sequence of order  $N$ generated by the equation
\begin{equation} \label{regr}
\sum\limits_{k=0}^{N}\gamma_k \xi(n-k)=\varepsilon_n,
\end{equation}
where $\varepsilon_n$ is a ``white noise'' sequence.

The minimax spectral characteristic $h(f_0)$ of the optimal linear estimate of the functional $A_N\xi$ can be calculated by the formula (\ref{5}), where
 \[C_N(e^{i\lambda})=\sum\limits_{k=0}^{N}c(k) e^{ik\lambda}=p_0=P^{-1}a(0),\]
  namely
\begin{equation} \label{spectr} \begin{gathered}
h(f_0)=\sum\limits_{k=0}^{N}a(k) e^{ik\lambda} - P^{-1}a(0)\sum\limits_{k=-N}^{N}r_k e^{ik\lambda}
 =\sum\limits_{k=1}^{N}a(k) e^{-ik\lambda}.
\end{gathered}\end{equation}

Let the sequence $a(N), a(N-1),\dots,a(0)$ that determines the functional  $A_N\xi$ be strictly positive.
In this case the Fourier coefficients $r_k=r_{-k}, k=0,1,\dots,N,$  can be found from the equation
\[B_N^0\vec{\bold{p}}_N^0=\vec{\bold{a}}_N, \quad \vec{\bold{p}}_N^0=(0,\ldots,0,p_0).
\]
From this equation we find the unknown values of
 \[r_k=Pa(N-k)a^{-1}(N), \,\, k=0,1,\dots,N.\]

The minimax spectral characteristic $h(f_0)$ of the optimal linear estimate of the functional $A_N\xi$ can be calculated by the formula
\begin{equation} \label{spectr2}
h(f_0)=\sum\limits_{k=0}^{N}a(k) e^{ik\lambda} - P^{-1}a(N)e^{iN\lambda}\sum\limits_{k=-N}^{N}r_{|k|} e^{ik\lambda}
 =\sum\limits_{k=1}^{N}a(N-k) e^{-i(N+k)\lambda}.
\end{equation}

Summing up our reasoning we come to conclusion that the following theorem holds true.
\begin{thm}
The least favourable in the class $D_0^{-}$ spectral density for the optimal linear estimation of the functional $A_N\xi$ determined by strictly positive sequence $a(0), a(1),\dots,a(N)$
is the spectral density of the autoregressive sequence (\ref{regr}) whose Fourier coefficients are $r_k=r_{-k}=Pa(k)a^{-1}(0), k=0,1,\dots,N$. The minimax spectral characteristics $h(f_0)$ is given by formula (\ref{spectr}).
The least favourable in the class $D_0^{-}$ spectral density for the optimal linear estimation of the functional $A_N\xi$ determined by strictly positive sequence $a(N), a(N-1),\dots,a(0)$
is the spectral density of the autoregressive sequence whose Fourier coefficients are $r_k=r_{-k}=Pa(N-k)a^{-1}(N), k=0,1,\dots,N$. The minimax spectral characteristics $h(f_0)$ is given by formula (\ref{spectr2}).
\end{thm}

\begin{exm} Consider the problem of the optimal linear estimation of the functional
$A_1\xi = a\xi(0)+b\xi(1)$ which depends on the unknown values $\xi(0)$, $\xi(1)$ of the stationary stochastic sequence $\left\{\xi(j):j \in \mathbb{Z}\right\}$ from observations of the sequence at points of time $\mathbb{Z}\backslash \left\{0,1\right\}. $

The least favourable spectral density in the class $D_0^{-}$ is of the form
\[  f_0(\lambda)=1\big/{|x+ye^{i\lambda}|^2},
\]
where
\[
x=
 \pm\left(P\left(1\pm\left(1-4(b/a)^2\right)^{1/2}\right)\Big/2\right)^{1/2},
\]
\[
y=
  \pm\left(P\left(1\mp\left(1-4(b/a)^2\right)^{1/2}\right)\Big/2\right)^{1/2}
\]
under the condition $|b/a|<1/2$. For example, in the case of $a=4$, $b=\sqrt3$
the least favourable spectral density
 $f_0(\lambda)$ and the minimax spectral characteristic are calculated by the formulas
\[
f_0(\lambda)=4\Big/P|\sqrt3+e^{i\lambda}|^2,
\]
\[ h(f_0)=-\sqrt3 e^{-i\lambda}.
\]

Under the condition $|b/a|>2$ the least favourable spectral density
 $f_0(\lambda)$ is as follows
\[
f_0(\lambda)=1\Big/|x+ye^{i\lambda}|^2,
\]
where
\[
x=
 \pm\left(P\left(1\pm\left(1-4(a/b)^2\right)^{1/2}\right)\Big/2\right)^{1/2},
\]
\[
y=
  \pm\left(P\left(1\mp\left(1-4(a/b)^2\right)^{1/2}\right)\Big/2\right)^{1/2}.
\]
For example, in the case of $a=\sqrt3$, $b=4$
the least favourable spectral density
 $f_0(\lambda)$ and the minimax spectral characteristic are calculated by the formulas
\[   f_0(\lambda)=
     4\Big/P|1+\sqrt3 e^{i\lambda}|^2,\]
\[
     h(f_0)=-\sqrt3 e^{2i\lambda}.
\]

\end{exm}

\subsection{Least favourable spectral densities in the class $D_M$}

Consider the problem of the optimal estimation of the functional $A_N\xi=\sum_{j=1}^{N}a(j)\xi(j)$ which depends on the unknown values of a stationary stochastic sequence $\xi(j)$
from observations of the sequence $\xi(j)$ at points of time $j\in\mathbb{Z}\backslash \{0,1,\dots,N\} $
in the case where the spectral density is from
the set of spectral densities with  restrictions on the moments of the function $f^{-1}(\lambda).$ Let
\begin{equation*}
D_M = \left\{f(\lambda)\left|\frac{1}{2\pi}\int\limits_{-\pi}^{\pi} f^{-1}(\lambda)\cos(m\lambda)d\lambda= r_m \right. , m=0,1,\ldots, M\right\},
\end{equation*}
 where $r_m, m=0,1,\ldots,M $ is a strictly positive sequence. There is an infinite number of functions in the class $D_M$ \cite{Krein} and the function
\[f^{-1}(\lambda)=\sum\limits_{m=-M}^{M}r_{\left|m\right|} e^{im\lambda}>0,  \,\,\,\lambda \in \left[-\pi,\pi\right].
\]
To find solutions to the constrained optimization problem (\ref{7}) for the set $D_M$ of admissible spectral densities we
use the  Lagrange multipliers method and get the equation
\begin{equation} \label{momenty}
\left|\sum\limits_{k=0}^{N}c(k) e^{ik\lambda}\right|^2 = \sum\limits_{m=0}^M \alpha_m \cos (m\lambda)=\left|\sum\limits_{m=0}^M p(m) e^{im\lambda}\right|^2,
\end{equation}
where $\alpha_m, m=0,1,\ldots,M $ are the Lagrange multipliers and $c(k), k=0,\ldots,N $ are solutions to the equation $B_N^0\vec{{c}}_N=\vec{{a}}_N$.

Consider two cases: $M\geq N$ and $M < N.$ Let $M\geq N.$ In this case the given Fourier coefficients  $r_m$ define the matrix $B_N^0$ and the optimization problem (\ref{extrem}) is degenerate.
If we take
$p(N+1)=\cdots=p(M)=0$ and components $(p(0),\dots,p(N))$ of the vector $\vec{{p}}_N$  find from the equation $B_N^0\vec{{p}}_N=\vec{{a}}_N$ then
the relation (\ref{momenty}) holds true. Thus the least favorable is every density $f(\lambda)\in D_M$ and  the density of the autoregression stochastic sequence
\begin{equation}\label{ff1}
f_0(\lambda)=1/{\sum\limits_{m=-M}^M r_{\left|m\right|}e^{im\lambda}}=1/{\left|\sum\limits_{k=0}^M \gamma_k e^{ik\lambda}\right|}
\end{equation}
is least favourable, too.

Let $M < N.$ Then the matrix $B_N$ is determined by the known  $r_m$, $m=0,1,\dots,M$ and the unknown  $r_m$, $m=M+1,\dots,N$,  Fourier coefficients of the function  $f^{-1}(\lambda).$

The unknown coefficients $p(k)$, $k=0,1,\dots,M$, and  $r_m$, $m=M+1,\dots,N$, can be found  from the equation  $B_N\vec{{p}}_N^0=\vec{{a}}_N$ with
 $p_N^0=(p(0),\dots,p(M),0,\dots,0)$, or from the equation  $B_N\vec{{p}}^N_0=\vec{{a}}_N$ with
$p_0^N=(0,\dots,0,p(M),p(M-1),\dots,p(0))$.

If the sequence $r_m$, $m=0,1,\dots,N$, that is constructed from the strictly positive sequence  $r_m$, $m=0,1,\dots,M$, and the calculated coefficients $r_m$, $m=M+1,\dots,N$, is also strictly  positive, then the least favourable spectral density  $f_0(\lambda)$ is determined by the Fourier coefficients  $r_m$, $m=0,1,\dots,N$, of the function $f_0^{-1}(\lambda)$
\begin{equation} \label{ff2}
f_0(\lambda)=1/{\sum\limits_{m=-N}^N r_{\left|m\right|}e^{im\lambda}}=1/{\left|\sum\limits_{k=0}^N \gamma_k e^{ik\lambda}\right|}
\end{equation}

\noindent Let us summarize our results and present them in the form
of a theorem.

\begin{thm}
The least favourable spectral density in the class  $D_M$  for the optimal linear estimation of the functional $A_N\xi$ in the case where $M\geq N$ is the spectral density (\ref{ff1}) of the autoregression stochastic sequence of order $M$ determined by  coefficients $r_m, m=0,1,\ldots,M.$
If $M < N$ and solutions $r_m$, $m=M+1,\dots,N$, to  the equation $B_N\vec{{p}}_N^0=\vec{{a}}_N$ with
 $p_N^0=(p(0),\dots,p(M),0,\dots,0)$, or to the equation  $B_N\vec{{p}}^N_0=\vec{{a}}_N$ with
$p_0^N=(0,\dots,0,p(M),p(M-1),\dots,p(0))$
together with coefficients $r_m$, $m=0,1,\dots,M$, form a strictly positive sequence, the least favourable spectral density in  $D_M$ is the density (\ref{ff2}) of the autoregression stochastic sequence of the order $N.$ The minimax characteristic of the estimate is calculated by formula (\ref{ch}).
\end{thm}

Similar statement holds true for the set of spectral densities
\begin{equation*}
D_{M,\cal R}=  \left\{f(\lambda)\left|\frac{1}{2\pi}\int\limits_{-\pi}^{\pi} f^{-1}(\lambda)\cos(w\lambda)d\lambda= r_m \right., \, m=0,1,\dots,M; \vec
r_{{M}}\in\cal R\right\},
\end{equation*}
where $\cal R$ is a convex compact which have a strictly positive sequence as an interior point.

\subsection{Least favourable spectral densities in the class $D_v^{u-}$}

Consider the problem of the optimal estimation of the functional $A_N\xi=\sum_{j=1}^{N}a(j)\xi(j)$ which depends on the unknown values of a stationary stochastic sequence $\xi(j)$
from observations of the sequence $\xi(j)$ at points of time $j\in\mathbb{Z}\backslash \{0,1,\dots,N\} $
in the case where the spectral density is from the set of spectral densities
\begin{equation*}
D_v^{u-} = \left\{f(\lambda)\left|0\leq v(\lambda)\leq f(\lambda)\leq u(\lambda), \frac{1}{2\pi}\int\limits_{-\pi}^{\pi} f^{-1}(\lambda)d\lambda= P \right. \right\},
\end{equation*}
where $v(\lambda), u(\lambda)$ are given bounded spectral densities. Let the sequence
$\{a(0),a(1),\dots,a(N)\}$ (or the sequence $\{a(N),a(N-1),\dots,a(0)\}$)
 be strictly positive.
To find solutions to the constrained optimization problem (\ref{7}) for the set $D_v^{u-}$ of admissible spectral densities we
use the condition $0 \in \partial\Delta_D(f_0)$.
It follows from the condition $0 \in \partial\Delta_D(f_0)$ for $D=D_v^{u-}$ that the Fourier coefficients of the function $f_0^{-1}$ satisfy both equation
\begin{equation*}
B_N^0\vec{{c}}_N=\vec{{a}}_N
\end{equation*}
 and the equation
\begin{equation*}
\left|\sum\limits_{k=0}^{N}\left(\left(B_N^0\right)^{-1}\vec{{a}}_N\right)_k e^{ik\lambda}\right|^2=\psi_1(\lambda)+\psi_2(\lambda)+p_0^{-2},
\end{equation*}
where $\psi_1(\lambda)\geq 0$ and $\psi_1(\lambda)=0$ if $f_0(\lambda)\geq v(\lambda);$ $\psi_2(\lambda)\leq 0$ and $\psi_2(\lambda)=0$ if $f_0(\lambda)\leq u(\lambda).$

Therefore, in the case where  $ v(\lambda)\leq f_0(\lambda)\leq u(\lambda)$, the function $ f_0^{-1}(\lambda)$ is of the form
\begin{equation*}
f_0^{-1}(\lambda)=\sum\limits_{k=-N}^{N}r_{{|k|}}e^{ik\lambda}=\left|\sum\limits_{k=0}^{N} \gamma_k e^{ik\lambda}\right|^2,
\end{equation*}
with $r_{{k}}=Pa(k)a^{-1}(0)$, in the case where the sequence $\{a(0),a(1),\dots,a(N)\}$ is strictly positive, and with
$r_{{k}}=r_{{-k}}=Pa(N-k)a^{-1}(N)$, in the case where the sequence $\{a(N),a(N-1),\dots,a(0)\}$ is strictly positive.

The least favourable in the class $D=D_v^{u-}$ is the density of the autoregression stochastic sequence of the order  $N$  if the following inequality holds true
\begin{equation}\label{18}
v^{-1}(\lambda) \geq \sum\limits_{k=0}^{N}(r_k e^{ik\lambda}+r_{-k}e^{-ik\lambda})=\left|\sum\limits_{k=0}^{N} \gamma_k e^{ik\lambda}\right|^2\geq u^{-1}(\lambda), \,\, \lambda \in \left[-\pi,\pi\right].
\end{equation}

In general case the least favourable density is of the form
\begin{equation}\label{19}
f_0(\lambda)=\max \left\{v(\lambda), \min \left\{u(\lambda),\left|p_0\sum\limits_{k=0}^{N}\left(\left(B_N^0\right)^{-1}\vec{{a}}_N\right)_k e^{ik\lambda}\right|^{-2}\right\}\right\}.
\end{equation}

The following theorem holds true.
\begin{thm}
If the sequence $\{a(0),a(1),\dots,a(N)\}$ is strictly positive and coefficients $r_{{k}}=r_{{-k}}=Pa(k)a^{-1}(0)\,$, $k=0,\dots,\,N $, satisfy the inequality  (\ref{18})
(or the sequence $\{a(N),a(N-1),\dots,a(0)\}$ is strictly positive and coefficients
$r_{{k}}=r_{{-k}}=Pa(N-k)\*a^{-1}(N)$
satisfy the inequality  (\ref{18})), then the least favourable in the class $D_v^{u-}$ spectral density for the optimal linear estimate of the functional $A_N\xi$ is density  of the autoregression stochastic sequence (\ref{regr}) of order $N$.
The minimax characteristic  $h(f_0)$ of the estimate can be calculated by the formula  (\ref{spectr}).
If the inequality  (\ref{18}) is not satisfied, then the least favourable spectral density in  $D_v^{u-}$ is determined by  relation (\ref{19}) and the extremum condition (\ref{extrem}). The minimax characteristic of the estimate is calculated by formula (\ref{ch}).
\end{thm}

\subsection{Conclusions}

In this section we propose methods of solution of the problem of the mean-square optimal linear estimation of the functional $A_N\xi=\sum\limits_{j=0}^{N}a(j)\xi(j)$ which depends on the unknown values of the stationary stochastic sequence $\xi(j)$.
Estimates are based on observations of the sequence $\xi(j)$ at points $j\in\mathbb{Z}\backslash \{0,1,\dots,N\} $. We provide formulas for calculating the values of the mean square error and the spectral characteristic of the optimal linear estimate of the functional in the case of spectral certainty, where the spectral density of the sequence $\xi(j)$ is exactly known.
In the case of spectral uncertainty, where the spectral density is unknown, but a set of admissible spectral densities is given, the minimax approach is applied to estimation of the functional. We obtain formulas that determine the least favourable spectral densities and the minimax spectral characteristics of the optimal linear estimates of the functional $A_N\xi$ for concrete classes of admissible spectral densities. It is shown that spectral densities the autoregressive stochastic sequences are the least favourable in some classes of spectral densities.

The minimax-robust approach to the problem of  estimation of one missed value of the stationary stochastic sequences based on convex optimization methods was initiated in papers by Franke \cite{Franke1984,Franke1985}. See also papers by Hosoya \cite{Hosoya}, Taniguchi \cite{Taniguchi1981}, and survey by Kassam and Poor \cite{Kassam}.

For the relative results on the mean-square optimal linear interpolation of linear functionals for stationary stochastic sequences and processes see papers by Moklyachuk~\cite{Moklyachuk:1994}  -- \cite{Moklyachuk:2008r}, book by Moklyachuk and Masyutka~\cite{Moklyachuk:2012}, papers by Moklyachuk and Sidei\cite{Sidei}, Moklyachuk and Ostapenko\cite{Ostapenko}.

\section{Interpolation problem for stationary sequences from observations with noise}

In this section we consider the problem of the mean-square optimal estimation of the linear
functional $A_N\xi=\sum\limits_{j=0}^{N}a(j)\xi(j)$ which depends on the unknown values of a stationary stochastic sequence $\xi(j),j\in \mathbb{Z},$
from observations of the sequence $\xi(j)+\eta(j)$ at points of time $j\in\mathbb{Z}\backslash \{0,1,\dots,N\}$.
The problem is investigated in the case of spectral certainty, where the spectral densities of the stationary stochastic sequences $\xi(j)$ and $\eta(j)$ are exactly known.
In this case the classical Hilbert space projection method of linear estimation of the functional is applied. Formulas are derived for calculation the value of the mean square error and the spectral characteristic of the mean-square optimal estimate of the linear functional.
In in the case of spectral uncertainty, where the spectral densities of the stationary stochastic sequences $\xi(j)$ and $\eta(j)$ are not exactly known, but classes of admissible spectral densities are given, the minimax-robust
procedure to linear estimation of the functional is applied.
Relations which determine the least favourable spectral densities and the minimax spectral characteristics are proposed for some special sets of admissible spectral densities.

\subsection{The classical Hilbert space projection method of linear interpolation}

Let  $\xi(j), j\in \mathbb{Z}$, and $\eta(j), j\in  \mathbb{Z}$,  be (wide sense) stationary stochastic sequences with zero mathematical expectations $E\xi(j)=0$, $E\eta(j)=0$.
The correlation functions $R_{\xi}(k)=E\xi(j+k)\overline{\xi(j)}$ and $R_{\eta}(k)=E\eta(j+k)\overline{\eta(j)}$
 of the stationary stochastic sequences $\xi(j), j\in \mathbb{Z}$, and $\eta(j), j\in  \mathbb{Z}$,
admit  the spectral representations \cite{Gihman}
\[
 R_{\xi}(k)=\int\limits_{-\pi}^{\pi}e^{ik\lambda}F(d\lambda), \quad R_{\eta}(k)=\int\limits_{-\pi}^{\pi}e^{ik\lambda}G(d\lambda),
 \]
where $F(d\lambda)$ and $G(d\lambda)$ are the spectral measures of the sequences.
We will consider stationary stochastic sequences with absolutely continuous spectral measures $F(d\lambda)$ and $G(d\lambda)$ and the correlation functions of the form
\[
 R_{\xi}(k)=\,\frac{1}{2\pi}\int\limits_{-\pi}^{\pi}e^{ik\lambda}f(\lambda)d\lambda, \quad R_{\eta}(k)=\,\frac{1}{2\pi}\int\limits_{-\pi}^{\pi}e^{ik\lambda}g(\lambda)d\lambda,
 \]
where $f(\lambda)$ and $g(\lambda)$ are the spectral density functions of the sequences $\xi(j), j\in \mathbb{Z}$, and $\eta(j), j\in  \mathbb{Z}$, correspondingly.

We will suppose that the spectral density functions  $f(\lambda)$ and $g(\lambda)$
satisfy the minimality condition
\begin{equation}\label{nminimal}
\int\limits_{-\pi}^{\pi}\frac{1}{f(\lambda)+g(\lambda)}d\lambda<\infty.
\end{equation}
Under this condition  the error-free interpolation of the unknown values of the sequence $\xi(j)+\eta(j)$ is impossible \cite{Rozanov}.

The stationary stochastic sequences $\xi(j)$ and $\eta(j)$ admit
 the spectral representations \cite{Gihman,Karhunen}
\begin{equation*}
\xi(j)=\int\limits_{-\pi}^{\pi}e^{ij\lambda}dZ_{\xi}(\lambda), \hspace{1cm}
\eta(j)=\int\limits_{-\pi}^{\pi}e^{ij\lambda}dZ_{\eta}(\lambda),
\end{equation*}
where $Z_{\xi}(d\lambda)$ and $Z_{\eta}(d\lambda)$ are orthogonal stochastic measure of the sequences $\xi(j)$ and $\eta(j)$ such that
\[
E{Z_{\xi}}(\Delta_1)\overline{{Z_{\xi}}(\Delta_2)}=F(\Delta_1\cap\Delta_2)=\,\frac{1}{2\pi}\int_{\Delta_1\cap\Delta_2}f(\lambda)d\lambda,
\]
\[
E{Z_{\eta}}(\Delta_1)\overline{{Z_{\eta}}(\Delta_2)}=G(\Delta_1\cap\Delta_2)=\,\frac{1}{2\pi}\int_{\Delta_1\cap\Delta_2}g(\lambda)d\lambda.
\]

Consider the problem of the mean-square optimal estimation of the linear
functional
\[A_N\xi=\sum\limits_{j=0}^{N}a(j)\xi(j)\]
 which depends on the unknown values of a stationary stochastic sequence $\xi(j),j\in\mathbb{Z},$  from observations of the sequence
$\xi(j)+\eta(j)$ at points of time $j\in\mathbb{Z}\backslash \{0,1,\dots,N\}$, where $\eta(j),j\in\mathbb{Z},$ is uncorrelated with $\xi(j),j\in\mathbb{Z},$  stationary stochastic sequence.

It follows from the spectral decomposition of the sequence $\xi(j)$
that we can represent the functional $A_{N}\xi$ in the  form
\begin{equation}\label{nasx}
A_{N}\xi=\int\limits_{-\pi}^{\pi}A_{N}(e^{i\lambda}){Z}_{\xi}(d\lambda),
 \end{equation}
where
\begin{equation*}
  A_N(e^{i\lambda})=\sum\limits_{j=0}^{N}a(j)e^{ij\lambda }.
 \end{equation*}

Denote by $\hat{A}_N\xi$  the mean square optimal linear estimate of the functional  $A_N\xi$ from observations of the sequence
 $\xi(j)+\eta(j)$ at points of time  $j\in\mathbb{Z}\backslash \{0,1,\dots,N\}$.
Denote by $\Delta(f,g)=E\left|A_N\xi-\hat{A}_N\xi\right|^2$ the mean square error of the estimate $\hat{A}_N\xi$. To find the estimate $\hat{A}_N\xi$
we will use the Hilbert space projection method proposed by A.~N.~Kolmogorov \cite{Kolmogorov}.
We will consider random values $\xi(j), j\in \mathbb{Z}$, and $\eta(j), j\in  \mathbb{Z}$, as elements of the Hilbert space  $H=L_2(\Omega,\mathcal{F},P)$
 of complex valued random variables with zero first moment, $E\xi=0$, finite second moment,
 $E|\xi|^2<\infty$, and the inner product $\left(\xi,{\eta}\right)=E\xi\overline{\eta}$.

Denote by $H^N(\xi+\eta)$ the subspace  of the Hilbert space  $H=L_2(\Omega,\mathcal{F},P)$ generated by elements $\{\xi(j)+\eta(j):j\in \mathbb{Z}\backslash \{0,1,\dots,N\}$.
Denote by $L_2(f+g)$ be the Hilbert space of complex-valued functions  that are square-integrable  with respect to the measure whose density is $f(\lambda)+g(\lambda).$
Denote by $L_2^N(f+g)$ the subspace of $L_2(f+g)$ generated by functions $\{e^{ij\lambda},\,j\in \mathbb{Z}\backslash \{0,1,\dots,N\} \}.$

The mean square optimal linear estimate $\hat{A}_N\xi$ of the functional  $A_N\xi$ from observations of the sequence $\xi(j)+\eta(j)$ at points of time  $j\in\mathbb{Z}\backslash \{0,1,\dots,N\}$
is an element of the $H^N(\xi+\eta)$. It can be represented in the form
\begin{equation}\label{nhatas}
\hat{A}_N\xi=\,\int\limits_{-\pi}^{\pi}h(e^{i\lambda}) (Z_{\xi}(d\lambda)+Z_{\eta}(d\lambda)),
 \end{equation}
where $h(e^{i\lambda}) \in L_2^N(f+g)$  is the spectral characteristic of the estimate $\hat{A}_N\xi.$

The mean square error $\Delta(h;f,g)$ of the estimate $\hat{A}_N\xi$ is given by the formula
\begin{equation*}
 \Delta(h;f,g)=E\left|A_N\xi-\hat{A}_N\xi\right|^2=\frac{1}{2\pi}\int\limits_{-\pi}^{\pi}\left|A_N(e^{i\lambda})-h(e^{i\lambda})\right|^2 f(\lambda)d\lambda
 +\frac{1}{2\pi}\int\limits_{-\pi}^{\pi}\left|h(e^{i\lambda})\right|^2 g(\lambda)d\lambda.
\end{equation*}

The Hilbert space projection method proposed by A. N. Kolmogorov \cite{Kolmogorov} makes it possible to find the spectral characteristic  $h(e^{i\lambda})$ and the mean square error $\Delta(h;f,g)$ of the optimal linear estimate of the functional  $A_N\xi$ in the case where the spectral densities $f(\lambda)$ and $g(\lambda)$ of the sequences $\xi(j),j\in\mathbb{Z},$ and $\eta(j),j\in\mathbb{Z},$ are exactly known and the minimality condition (\ref{nminimal}) is satisfied. The spectral characteristic can be found from the following conditions:
\begin{equation*} \begin{split}
1)& \hat{A}_N\xi \in H^N(\xi+\eta), \\
2)& A_N\xi-\hat{A}_N\xi \bot  H^N(\xi+\eta).
\end{split} \end{equation*}

It follows from the second condition that for any $j\in\mathbb{Z}\backslash \{0,1,\dots,N\}$ the following equations should be satisfied
\begin{equation*} \begin{split}
&E\left[\left(A_N\xi-\hat{A}_N\xi\right)\left(\overline{\xi(j)}+\overline{\eta(j)}\right)\right]=\\
&=\frac{1}{2\pi}\int\limits_{-\pi}^{\pi} \left(A_N(e^{i\lambda})- h(e^{i\lambda})\right)e^{-ij\lambda}f(\lambda)d\lambda-\frac{1}{2\pi}\int\limits_{-\pi}^{\pi}  h(e^{i\lambda})e^{-ij\lambda}g(\lambda)d\lambda=0.
\end{split} \end{equation*}

 The last equations are equivalent to equations
 \begin{equation*}
\frac{1}{2\pi}\int\limits_{-\pi}^{\pi} \left[A_N(e^{i\lambda})f(\lambda)- h(e^{i\lambda})(f(\lambda)  +g(\lambda))\right]e^{-ij\lambda}d\lambda=0, \quad  j\in \mathbb{Z} \backslash \{0,1,\dots,N\}.
 \end{equation*}

It follows from these equations that the function $\left[A_N(e^{i\lambda})f(\lambda)- h(e^{i\lambda})(f(\lambda)  +g(\lambda))\right]$ is of the form
\begin{equation}\label{nA_N}
A_N(e^{i\lambda})f(\lambda)- h(e^{i\lambda})(f(\lambda)  +g(\lambda))=C_N(e^{i\lambda}),
 \end{equation}
\[C_N(e^{i\lambda})=\sum\limits_{j=0}^{N}c(j)e^{ij\lambda},
\]
where $c(j), j=0,1,\dots,N,$ are unknown coefficients that we have to find.

 From the relation (\ref{nA_N}) we deduce that the spectral characteristic $ h(e^{i\lambda})$
 of the optimal linear estimate of the functional  $A_N\xi$
 is of the form
 \begin{equation} \label{nsphar}
 \begin{split}
h(e^{i\lambda})=&\frac{A_N(e^{i\lambda})f(\lambda)-C_N(e^{i\lambda})}{f(\lambda)+g(\lambda)}=\\
=A_N(e^{i\lambda})-&\frac{A_N(e^{i\lambda})g(\lambda)+C_N(e^{i\lambda})}{f(\lambda)+g(\lambda)}.
\end{split} \end{equation}

It follows from the first condition, which determines the spectral characteristic $h(e^{i\lambda}) \in L_2^N(f+g)$
 of the optimal linear estimate of the functional  $A_N\xi$,
that the Fourier coefficients of the function  $h(e^{i\lambda})$ are equal to zero for $j=0,1,\dots,N$, namely
\[
\frac{1}{2\pi}\int\limits_{-\pi}^{\pi} h(e^{i\lambda})e^{-ij\lambda }d\lambda=0, \,\, j =0,1,\dots,N.
\]
Using the last relations and (\ref{nsphar}) we get the following system of equations
\begin{equation*} \begin{split}
\int\limits_{-\pi}^{\pi}\left(A_N(e^{i\lambda})\frac{f(\lambda)}{f(\lambda)+g(\lambda)}-\frac{C_N(e^{i\lambda})}{f(\lambda)+g(\lambda)}\right)e^{-ij\lambda}d\lambda=0, \hspace{1cm} j =0,1,\dots,N.
\end{split} \end{equation*}

The last equations can be written in the form
\begin{equation}\label{n3} \begin{split}
\sum_{k=0}^{N}a(k)\int\limits_{-\pi}^{\pi}\frac{e^{i(k-j)\lambda}f(\lambda)}{f(\lambda)+g(\lambda)}d\lambda
-\sum\limits_{k=0}^{N}c(k)\int\limits_{-\pi}^{\pi}\frac{e^{i(k-j)\lambda}}{f(\lambda)+g(\lambda)}d\lambda=0, \,\, j =0,1,\dots,N.
\end{split} \end{equation}

Let us introduce the following notations
$$R_{j,k}^N=\frac{1}{2\pi} \int\limits_{-\pi}^{\pi}e^{-i(j-k)\lambda}\frac{f(\lambda)}{f(\lambda)+g(\lambda)}d\lambda,$$
$$B_{j,k}^N=\frac{1}{2\pi} \int\limits_{-\pi}^{\pi}e^{-i(j-k)\lambda}\frac{1}{f(\lambda)+g(\lambda)}d\lambda,$$
$$Q_{j,k}^N=\frac{1}{2\pi} \int\limits_{-\pi}^{\pi}e^{-i(j-k)\lambda}\frac{f(\lambda)g(\lambda)}{f(\lambda)+g(\lambda)}d\lambda.$$

Making use the introduced notations we can write equations (\ref{n3})  in the form
\begin{equation*}\begin{split}
\sum\limits_{k=0}^{N}R^N_{j,k}a(k)=
\sum\limits_{k=0}^{N}B^N_{j,k}c(k), \,\, j =0,1,\dots,N.
\end{split} \end{equation*}

The derived equations can be written in the matrix form
\begin{equation*}\begin{split}
\bold{R}_N\bold{a}_N=\bold{B}_N \bold{c}_N,
\end{split} \end{equation*}
where $\bold{a}_N=(a(0),a(1),\dots,a(N))$ is a vector constructed from the coefficients that determine the functional $A_N\xi$,
$\bold{c}_N=(c(0),c(1),\dots,c(N))$ is a vector constructed from the unknown coefficients $c(k),k=0,1,\dots,N$, $\bold{B}_N$ and $ \bold{R}_N$ are linear operators in $\mathbb C^{N+1}$, which are determined by matrices with elements $(\bold{B}_N)_{j,k}=B_{j,k}^N$, $(\bold{R}_N)_{j,k}=R_{j,k}^N$, $j,k=0,1,\dots,N.$

Since the matrix $\bold B_N$ is reversible \cite{Salehi}, we get the formula
\begin{equation} \label{nrivn2}
\bold{c}_N=\bold{B}_N^{-1}\bold{R}_N\bold{a}_N,
\end{equation}

 Hence, the unknown coefficients $c(j), j=0,1,\dots,N,$ are calculated  by the formula
\[c(j)=\left(\bold{B}_N^{-1}\bold{R}_N\bold{a}_N\right)_j,
\]
where  $\left(\bold{B}_N^{-1}\bold{R}_N\bold{a}_N\right)_j $ is the  $j$-th  component of the vector  $\bold{B}_N^{-1}\bold{R}_N\bold{a}_N,$
and the formula for calculating  the spectral characteristic of the estimate $\hat{A}_N\xi$ is of the form
\begin{equation}\label{n4} \begin{split}
h(e^{i\lambda})=A_N(e^{i\lambda})\frac{f(\lambda)}{f(\lambda)+g(\lambda)}-
\frac{\sum\limits_{k=0}^{N}(\bold{B}_N^{-1}\bold{R}_N\bold{a}_N)_ke^{ik\lambda} }{f(\lambda)+g(\lambda)}.
\end{split} \end{equation}

The mean square error of the estimate of the function can be calculated by the formula
\begin{equation} \label{n55} \begin{split}
\Delta(h;f,g)=E\left|A_N\xi-\hat{A}_N\xi\right|^2&=\frac{1}{2\pi}\int\limits_{-\pi}^{\pi}\frac{\left|A_N(e^{i\lambda})g(\lambda)+
\sum\limits_{k=0}^{N}(\bold{B}_N^{-1}\bold{R}_N\bold{a}_N)_k e^{ik\lambda}\right|^2}{(f(\lambda)+g(\lambda))^2}f(\lambda)d\lambda\\
&+\frac{1}{2\pi}\int\limits_{-\pi}^{\pi}\frac{\left|A_N(e^{i\lambda})f(\lambda)-
\sum\limits_{k=0}^{N}(\bold{B}_N^{-1}\bold{R}_N\bold{a}_N)_k e^{ik\lambda}\right|^2}{(f(\lambda)+g(\lambda))^2}g(\lambda)d\lambda\\
&=\langle\bold{R}_N\bold{a}_N,\bold{B}_N^{-1}\bold{R}_N\bold{a}_N\rangle+\langle\bold{Q}_N\bold{a}_N,\bold{a}_N\rangle,
\end{split} \end{equation}
where $\bold{Q_N}$ is a linear operator in $\mathbb C^{N+1}$, which is determined by matrix with elements $(\bold{Q}_N)_{j,k}=Q^N_{j,k}$,  $j,k=0,1,\dots,N.$

\noindent Let us summarize our results and present them in the form
of a theorem.

\begin{thm}
Let $\xi(j)$ and $\eta(j)$ be uncorrelated stationary stochastic sequences with the spectral densities $f(\lambda)$ and $g(\lambda)$ that satisfy the minimality condition (\ref{nminimal}). The spectral characteristic $h(e^{i\lambda})$ and the mean square error  $\Delta(h;f,g)$ of the optimal linear estimate $\hat{A}_N\xi$ of the functional $A_N\xi$ based on observations of the sequence $\xi(j)+\eta(j)$ at points of time $j\in\mathbb{Z}\backslash \{0,1,\dots,N\} $ can be calculated by  formulas (\ref{n4}), (\ref{n55}).
\end{thm}

\subsection{Minimax-robust method of interpolation}

The traditional methods of estimation of the functional $A_N\xi$ which depends on the unknown values of a stationary stochastic sequence $\xi(j)$ can be applied in the case  where the spectral densities $f(\lambda)$ and $g(\lambda)$ of the considered stochastic sequences $\xi(j)$ and $\eta(j)$ are exactly known.
In practise, however, we do not have complete information on spectral densities of the sequences. For this reason  we apply the minimax(robust) method of estimation of the functional $A_N\xi$, that is we find an estimate that minimizes the maximum of the mean square errors for all spectral densities from the given class of admissible spectral densities $D$.

\begin{ozn}
For a given class of spectral densities $D=D_f\times D_g$ the spectral densities $f_0(\lambda)\in D_f$, $g_0(\lambda)\in D_g$  are called the least favourable in $D$ for the optimal linear estimation of the functional $A_N\xi$ if the following relation holds true
$$\Delta\left(f_0,g_0\right)=\Delta\left(h\left(f_0,g_0\right);f_0,g_0\right)=\max\limits_{(f,g)\in D_f\times D_g}\Delta\left(h\left(f,g\right);f,g\right).$$
\end{ozn}

\begin{ozn}
For a given class of spectral densities $D=D_f\times D_g$ the spectral characteristic $h^0(e^{i\lambda})$ of the optimal linear estimate of the functional $A_N\xi$ is called minimax-robust if
$$h^0(e^{i\lambda})\in H_D= \bigcap\limits_{(f,g)\in D_f\times D_g} L_2^N(f+g),$$
$$\min\limits_{h\in H_D}\max\limits_{(f,g)\in D}\Delta\left(h;f,g\right)=\sup\limits_{(f,g)\in D}\Delta\left(h^0;f,g\right).$$
\end{ozn}

It follows from the introduced definitions and the obtained formulas that the following statement holds true.

\begin{lem} The spectral densities $f_0(\lambda)\in D_f,$ $g_0(\lambda)\in D_g$ are the least favourable in the class of admissible spectral densities $D=D_f\times D_g$ for the optimal linear estimate of the functional $A_N\xi$ if the Fourier coefficients of the functions
$$(f_0(\lambda)+g_0(\lambda))^{-1}, \quad f_0(\lambda)(f_0(\lambda)+g_0(\lambda))^{-1}, \quad f_0(\lambda)g_0(\lambda)(f_0(\lambda)+g_0(\lambda))^{-1}$$
define operators $\bold B_N^0, \bold R_N^0, \bold Q_N^0$ that determine a solution to the optimization problem
\begin{equation} \label{nextrem}\begin{split}
\max\limits_{(f,g)\in D_f\times D_g}\langle\bold{R}_N\bold{a}_N,\bold{B}_N^{-1}\bold{R}_N\bold{a}_N\rangle+\langle\bold{Q}_N\bold{a}_N,\bold{a}_N\rangle\\
=\langle\bold{R}_N^0\bold{a}_N,(\bold{B}_N^0)^{-1}\bold{R}_N^0\bold{a}_N\rangle+\langle\bold{Q}_N^0\bold{a}_N,\bold{a}_N\rangle.
\end{split}\end{equation}
The minimax spectral characteristic $h^0=h(f_0,g_0)$ can be calculated by the formula (\ref{n4}) if $h(f_0,g_0) \in H_D.$
\end{lem}

The least favourable spectral densities $f_0(\lambda)$, $g_0(\lambda)$  and the minimax spectral characteristic $h^0=h(f_0,g_0)$
form a saddle point of the function  $\Delta \left(h;f,g\right)$ on the set $H_D\times D$. The saddle point inequalities
$$\Delta\left(h;f_0,g_0\right)\geq\Delta\left(h^0;f_0,g_0\right)\geq \Delta\left(h^0;f,g\right) $$  $$ \forall h \in H_D, \forall f \in D_f, \forall g \in D_g$$
hold true if $h^0=h(f_0,g_0)$ and $h(f_0,g_0)\in H_D,$  where $(f_0,g_0)$ is a solution to the constrained optimization problem
\begin{equation} \label{n7}
\sup\limits_{(f,g)\in D_f\times D_g}\Delta\left(h(f_0,g_0);f,g\right)=\Delta\left(h(f_0,g_0);f_0,g_0\right),
\end{equation}
where
\begin{equation*}\begin{split}
\Delta\left(h(f_0,g_0);f,g\right)&=\frac{1}{2\pi}\int\limits_{-\pi}^{\pi}\frac{\left|A_N(e^{i\lambda})g_0(\lambda)
+C_N^0(e^{i\lambda})\right|^2}{(f_0(\lambda)+g_0(\lambda))^2}f(\lambda)d\lambda\\
&+\frac{1}{2\pi}\int\limits_{-\pi}^{\pi}\frac{\left|A_N(e^{i\lambda})f_0(\lambda)-
C_N^0(e^{i\lambda})\right|^2}{(f_0(\lambda)+g_0(\lambda))^2}g(\lambda)d\lambda,
\end{split}\end{equation*}
\begin{equation*}
C_N^0(e^{i\lambda})=\sum\limits_{j=0}^{N}((\bold{B}_N^0)^{-1}\bold{R}_N^0\bold{a}_N)_je^{ij\lambda},
\end{equation*}

The constrained optimization problem (\ref{n7}) is equivalent to the unconstrained optimization problem
\begin{equation} \label{n8}
\Delta_D(f,g)=-\Delta(h(f_0,g_0);f,g)+\delta(f,g\left|D_f\times D_g\right.)\rightarrow \inf,
\end{equation}
where $\delta(f,g\left|D_f\times D_g\right.)$ is the indicator function of the set  $D=D_f\times D_g$. Solution $(f_0,g_0)$ to the problem (\ref{n8}) is characterized by the condition $0 \in \partial\Delta_D(f_0,g_0),$ where $\partial\Delta_D(f_0,g_0)$ is the subdifferential of the convex functional $\Delta_D(f,g)$ at point $(f_0,g_0)$.
This condition makes it possible to find the least favourable spectral densities in some special classes of spectral densities $D$ \cite{Ioffe}, \cite{Pshenychn}, \cite{Rockafellar}.

Note, that the form of the functional $\Delta(h(f_0,g_0);f,g)$  is convenient for application the Lagrange method of indefinite multipliers for
finding solution to the problem (\ref{n7}).
Making use the method of Lagrange multipliers and the form of
subdifferentials of the indicator functions
we describe relations that determine least favourable spectral densities in some special classes
of spectral densities
(see books \cite{Golichenko,Moklyachuk:2008,Moklyachuk:2012} for additional details).

\begin{lem} Let $(f_0,g_0)$ be a solution to the optimization problem  (\ref{n8}). The spectral densities $f_0(\lambda)$, $g_0(\lambda)$ are the least favourable in the class $D=D_f\times D_g$, and
the spectral characteristic $h^0=h(f_0,g_0)$ is minimax for the optimal estimate of the functional $A_N\xi$ if  $h(f_0,g_0) \in H_D$.
\end{lem}

\subsection{Least favourable spectral densities in the class $D_f^0 \times D_g^0$}

Consider the problem of the optimal estimation of the functional $A_N\xi=\sum_{j=1}^{N}a(j)\xi(j)$ which depends on the unknown values of a stationary stochastic sequence $\xi(j)$
from observations of the sequence $\xi(j)+\eta(j)$ at points of time $j\in\mathbb{Z}\backslash \{0,1,\dots,N\}$
in the case where the spectral densities  $f(\lambda)$, $g(\lambda)$ are from the class $D=D_f^0\times D_g^0$, where
\begin{equation*}
D_f^0 = \left\{f(\lambda)\left|\frac{1}{2\pi}\int\limits_{-\pi}^{\pi} f(\lambda)d\lambda\leq P_1\right.\right\},
\end{equation*}
\begin{equation*}
 D_g^0 = \left\{g(\lambda)\left|\frac{1}{2\pi}\int\limits_{-\pi}^{\pi} g(\lambda)d\lambda\leq P_2\right.\right\}.
\end{equation*}

Let the densities $f_0(\lambda) \in D_f^0$, $g_0(\lambda) \in D_g^0$ and the functions $h_f(f_0,g_0)$, $h_g(f_0,g_0)$, which are determined by the relations
\begin{equation} \label{nhf}
h_f(f_0,g_0)=\frac{\left|A_N(e^{i\lambda})g_0(\lambda)+
C_N^0(e^{i\lambda}) \right|^2}{(f_0(\lambda)+g_0(\lambda))^2},
\end{equation}
\begin{equation} \label{nhg}
h_g(g_0,g_0)=\frac{\left|A_N(e^{i\lambda})f_0(\lambda)-
C_N^0(e^{i\lambda})\right|^2}{(f_0(\lambda)+g_0(\lambda))^2},
\end{equation}
be bounded. In this case the functional
$$
\Delta(h(f_0,g_0);f,g)=\frac{1}{2\pi}\int\limits_{-\pi}^{\pi}h_f(f_0,g_0)f(\lambda)d\lambda + \frac{1}{2\pi}\int\limits_{-\pi}^{\pi}h_g(f_0,g_0)g(\lambda)d\lambda
$$
is linear and continuous  on the space $L_1 \times L_1$ and we can apply the method of Lagrange multipliers to find solution to the optimization problem  (\ref{n8}).
We get the following relations that determine the least favourable spectral densities  $f^0\in D^0_f$, $g^0\in D^0_g$
\begin{equation*} \begin{split}
&-\frac{1}{2\pi}\int\limits_{-\pi}^{\pi}h_f(f_0,g_0)\rho(f(\lambda))d\lambda - \frac{1}{2\pi}\int\limits_{-\pi}^{\pi}h_g(f_0,g_0)\rho(g(\lambda))d\lambda\\
&+\alpha_1\frac{1}{2\pi}\int\limits_{-\pi}^{\pi}\rho(f(\lambda))d\lambda+\alpha_2\frac{1}{2\pi}\int\limits_{-\pi}^{\pi}\rho(g(\lambda))d\lambda=0,
\end{split}\end{equation*}
where $\rho(f(\lambda))$ and $\rho(g(\lambda))$ are variations of the functions $f(\lambda)$ and $g(\lambda)$, the constants $\alpha_1\geq 0$ and $\alpha_2\geq 0$.
From this relation we get that the least favourable spectral densities
 $f_0(\lambda) \in D_f^0$, $g_0(\lambda) \in D_g^0$ satisfy equations
\begin{equation} \label{nhf1}
\left|A_N(e^{i\lambda})g_0(\lambda)+
C_N^0(e^{i\lambda})\right|=\alpha_1(f_0(\lambda)+g_0(\lambda)),
\end{equation}
\begin{equation} \label{nhg1}
\left|A_N(e^{i\lambda})f_0(\lambda)-
C_N^0(e^{i\lambda})\right|=\alpha_2(f_0(\lambda)+g_0(\lambda)).
\end{equation}
Note, that  $\alpha_1\neq 0$ in the case where
\[\frac{1}{2\pi}\int\limits_{-\pi}^{\pi} f_{0}(\lambda)d\lambda=P_1,\]
and $\alpha_2\neq 0$ in the case where
\[\frac{1}{2\pi}\int\limits_{-\pi}^{\pi} g_{0}(\lambda)d\lambda=P_2.
\]

Summing up our reasoning we come to conclusion that the following theorem holds true.
\begin{thm}
Let the spectral densities $f_0(\lambda) \in D_f^0$ and $g_0(\lambda) \in D_g^0$ satisfy the minimality condition (\ref{nminimal}) and let the functions $h_f(f_0,g_0)$ and $h_g(f_0,g_0)$, determined by the formulas (\ref{nhf}), (\ref{nhg}), be bounded. The functions  $f_0(\lambda), g_0(\lambda)$, which give solution to the system of equations (\ref{nhf1}), (\ref{nhg1})
are the least favourable spectral densities in the class $D=D_f^0\times D_g^0$, if they determine a solution to the optimization problem (\ref{nextrem}).
The function $h^0(e^{i\lambda})$, determined by the formula (\ref{nsphar}), is minimax-robust spectral characteristic of the optimal linear estimate of the functional $A_N\xi$.
\end{thm}

\begin{thm}
Let the spectral density $f(\lambda)$ be known and fixed, and let the spectral density $g_0(\lambda) \in D_g^0$. Let the functions $f(\lambda)$ and $g_0(\lambda)$ be such that the function $(f(\lambda)+g_0(\lambda))^{-1}$ is integrable and let the function $h_g(f,g_0)$, determined by the formula (\ref{nhg}), be bounded.
The spectral density $g_0(\lambda)$ is the least favourable spectral densities in the class $D_g^0$ for the optimal linear estimate of the functional $A_N\xi$,
if it is of the form
\begin{equation*}
g_0(\lambda)= \max\left\{0, \alpha_2^{-1}\left|A_N(e^{i\lambda})f(\lambda)-
C_N^0(e^{i\lambda})\right|-f(\lambda)\right\}
\end{equation*}
and the functions $f(\lambda), g_0(\lambda)$ determine a solution to the optimization problem (\ref{nextrem}).
The function $h^0(e^{i\lambda})$, determined by the formula (\ref{nsphar}), is minimax-robust spectral characteristic of the optimal linear estimate of the functional $A_N\xi$.
\end{thm}

\begin{thm}
Let the spectral density  $f_0(\lambda) \in D_f^0$, let the function $f_0^{-1}(\lambda)$ be integrable, and let the function $h(f_0)$, determined by the formula (\ref{5}), be bounded.
The spectral density  $f_0(\lambda)$ is the least favourable spectral densities in the class $D_f^0$ for the optimal linear estimate of the functional $A_N\xi$, based on observations of the sequence $\xi(j)$ at points of time $j\in\mathbb{Z}\backslash \{0,1,\dots,N\}$,
if it satisfies the relation
\begin{equation*}
f_0(\lambda)=\alpha_1\left|C_N^0(e^{i\lambda})\right|
\end{equation*}
and $f_0(\lambda)$ determine a solution to the optimization problem (\ref{extrem}).
The function $h^0(e^{i\lambda})$, determined by the formula (\ref{5}), is minimax-robust spectral characteristic of the optimal linear estimate of the functional $A_N\xi$.
\end{thm}

\subsection{Least favourable spectral densities in the class $D_v^u \times D_{\varepsilon}$}

Consider the problem of the optimal estimation of the functional $A_N\xi=\sum_{j=1}^{N}a(j)\xi(j)$ which depends on the unknown values of a stationary stochastic sequence $\xi(j)$
from observations of the sequence $\xi(j)+\eta(j)$ at points of time $j\in\mathbb{Z}\backslash \{0,1,\dots,N\}$
in the case where the spectral densities  $f(\lambda)$, $g(\lambda)$ are from the class $D=D_v^u \times D_{\varepsilon}$, where
\[
D_v^u = \left\{f(\lambda)\left| v(\lambda)\leq f(\lambda)\leq u(\lambda),\,\,\frac{1}{2\pi}\int\limits_{-\pi}^{\pi} f(\lambda)d\lambda\leq P_1\right.\right\},
\]
\[ D_\varepsilon = \left\{g(\lambda)\left| g(\lambda)=(1-\varepsilon)g_1(\lambda)+\varepsilon w(\lambda),\,\,\frac{1}{2\pi}\int\limits_{-\pi}^{\pi} g(\lambda)d\lambda\leq P_2\right.\right\}.
\]
Here the spectral densities $v(\lambda), u(\lambda), g_1(\lambda)$ are known and fixed and the densities $v(\lambda)$ and $u(\lambda)$ are bounded.

Let the densities $f^0(\lambda) \in D_v^u$, $g^0(\lambda) \in D_{\varepsilon}$ determine the bounded functions  $h_f(f_0,g_0)$, $h_g(f_0,g_0)$ with the help of formulas  (\ref{nhf}), (\ref{nhg}).
 Then from the condition $0 \in \partial\Delta_{D_{f,g}}(f_0,g_0)$ we derive the following equations that determine the least favourable spectral densities
\begin{equation} \label{nhf2}
\left|A_N(e^{i\lambda})g_0(\lambda)+
C_N^0(e^{i\lambda})\right|=(f_0(\lambda)+g_0(\lambda))(\gamma_1(\lambda)+\gamma_2(\lambda)+\alpha_1^{-1}),
\end{equation}
\begin{equation} \label{nhg2}
\left|A_N(e^{i\lambda})f_0(\lambda)-
C_N^0(e^{i\lambda})\right|=(f_0(\lambda)+g_0(\lambda))(\varphi(\lambda)+\alpha_2^{-1}).
\end{equation}
Here $\gamma_1\leq 0$ and $\gamma_1= 0$ in the case $f_0(\lambda)\geq v(\lambda)$; $\gamma_2\geq 0$ and $\gamma_2= 0$ in the case $f_0(\lambda)\leq u(\lambda)$;
 $\varphi (\lambda)\leq 0$ and $\varphi(\lambda)= 0$ in the case $g_0(\lambda)\geq (1-\varepsilon) g_1(\lambda)$.

The following theorems hold true.

\begin{thm}
Let the spectral densities $f_0(\lambda) \in D_v^u$, $g_0(\lambda) \in D_{\varepsilon}$ satisfy the minimality condition (\ref{nminimal}) and let the functions $h_f(f_0,g_0)$ and $h_g(f_0,g_0)$, determined by the formulas (\ref{nhf}), (\ref{nhg}), be bounded.
The functions  $f_0(\lambda), g_0(\lambda)$, which give solution to the system of equations (\ref{nhf2}), (\ref{nhg2})
are the least favourable spectral densities in the class $D_v^u \times D_{\varepsilon}$, if they determine a solution to the optimization problem (\ref{nextrem}).
The function $h^0(e^{i\lambda})$, determined by the formula (\ref{nsphar}), is minimax-robust spectral characteristic of the optimal linear estimate of the functional $A_N\xi$.
 \end{thm}

\begin{thm}
Let the spectral density $f(\lambda)$ be known, and let the spectral density $g_0(\lambda) \in D_{\varepsilon}$. Let the function $f(\lambda)+g_0(\lambda)$  satisfy the minimality condition (\ref{nminimal}), and let the function $h_g(f,g_0)$, determined by the formula (\ref{nhg}), be bounded.
The spectral density $g_0(\lambda)$ is the least favourable spectral densities in the class $D_{\varepsilon}$ for the optimal linear estimate of the functional $A_N\xi$,
if it is of the form
\begin{equation*}
g_0(\lambda)= \max\left\{(1-\varepsilon)g_1(\lambda), \alpha_2\left|A_N(e^{i\lambda})f(\lambda)-
C_N^0(e^{i\lambda})\right|-f(\lambda)\right\}
\end{equation*}
and the functions $f(\lambda), g_0(\lambda)$ determine a solution to the optimization problem (\ref{nextrem}).
The function $h^0(e^{i\lambda})$, determined by the formula (\ref{nsphar}), is minimax-robust spectral characteristic of the optimal linear estimate of the functional $A_N\xi$.
\end{thm}

\begin{thm}
Let the spectral density  $f_0(\lambda) \in D_v^u$, let the function $f_0^{-1}(\lambda)$ be integrable, and let the function $h(f_0)$, determined by the formula (\ref{5}), be bounded.
The spectral density  $f_0(\lambda)$ is the least favourable spectral densities in the class $D_v^u$ for the optimal linear estimate of the functional $A_N\xi$, based on observations of the sequence $\xi(j)$ at points of time $j\in\mathbb{Z}\backslash \{0,1,\dots,N\}$,
if it satisfies the relation
\begin{equation*}
 f_0(\lambda)=\max\left\{v(\lambda),\min\left\{u(\lambda),\alpha_1\left|C_N^0(e^{i\lambda})\right|\right\}\right\}
\end{equation*}
and $f_0(\lambda)$ determine a solution to the optimization problem (\ref{extrem}).
The function $h^0(e^{i\lambda})$, determined by the formula (\ref{5}), is minimax-robust spectral characteristic of the optimal linear estimate of the functional $A_N\xi$.
\end{thm}

\subsection{Least favourable spectral densities in the class $D_{2\varepsilon_1}\times D_{1\varepsilon_2}$}

Consider the problem of the optimal estimation of the functional $A_N\xi=\sum_{j=1}^{N}a(j)\xi(j)$ which depends on the unknown values of a stationary stochastic sequence $\xi(j)$
from observations of the sequence $\xi(j)+\eta(j)$ at points of time $j\in\mathbb{Z}\backslash \{0,1,\dots,N\}$
in the case where the spectral densities  $f(\lambda)$, $g(\lambda)$ are from the class $D_{2\varepsilon_1}\times D_{1\varepsilon_2}$, which describe the models of
 ``$\varepsilon$-neighbourhood'' of spectral densities in the space $L_2 \times L_1$. Let
\begin{equation*}
D_{2\varepsilon_1}=\left\{f(\lambda)\left|\frac{1}{2\pi}\int \limits_{-\pi}^{\pi}\left|f(\lambda)-f_1(\lambda)\right|^2d\lambda\leq\varepsilon_1\right.\right\}
\end{equation*}
be  ``$\varepsilon$-neighbourhood'' in the space  $L_2$ of a given bounded spectral density $f_1(\lambda)$, and let
\begin{equation*}
D_{1\varepsilon_2}=\left\{g(\lambda)\left|\frac{1}{2\pi}\int \limits_{-\pi}^{\pi}\left|g(\lambda)-g_1(\lambda)\right|d\lambda\leq\varepsilon_2\right.\right\}
\end{equation*}
be  "$\varepsilon$-neighbourhood" in the space  $L_1$ of a given bounded spectral density  $g_1(\lambda)$.

Let the spectral densities  $f_0(\lambda) \in D_{2\varepsilon_1}$, $g_0(\lambda) \in D_{1\varepsilon_2}$ determine the bounded functions  $h_f(f_0,g_0)$, $h_g(f_0,g_0)$ with the help of formulas  (\ref{nhf}), (\ref{nhg}).
 Then from the condition $0 \in \partial\Delta_{D_{f,g}}(f_0,g_0)$ for $D=D_{2\varepsilon_1}\times D_{1\varepsilon_2}$ we derive the following equations that determine the least favourable spectral densities
\begin{equation} \label{nhf3}
\left|A_N(e^{i\lambda})g_0(\lambda)+
C_N^0(e^{i\lambda})\right|^2=(f_0(\lambda)+g_0(\lambda))^2(f_0(\lambda)-f_1(\lambda))\alpha_1,
\end{equation}
\begin{equation} \label{nhg3}
\left|A_N(e^{i\lambda})f_0(\lambda)-
C_N^0(e^{i\lambda})\right|^2=(f_0(\lambda)+g_0(\lambda))^2\Psi(\lambda)\alpha_2,
\end{equation}
where $\left|\Psi(\lambda)\right|\leq 1$ and $\Psi(\lambda)=sign (g_0(\lambda)-g_1(\lambda))$, in the case $g_0(\lambda)\neq g_1(\lambda)$, $\alpha_1, \alpha_2$ are constants.

Equations (\ref{nhf3}), (\ref{nhg3}) with the optimization problem (\ref{extrem}) and the normalising conditions
\begin{equation} \label{nn1}
\frac{1}{2\pi}\int \limits_{-\pi}^{\pi}\left|f(\lambda)-f_1(\lambda)\right|^2d\lambda=\varepsilon_1
\end{equation}
\begin{equation} \label{nn2}
\frac{1}{2\pi}\int \limits_{-\pi}^{\pi}\left|g(\lambda)-g_1(\lambda)\right|d\lambda=\varepsilon_2
\end{equation}
determine the least favourable spectral densities in the class $D=D_{2\varepsilon_1}\times D_{1\varepsilon_2}$.

The following theorems hold true.

\begin{thm}
Let the spectral densities $f_0(\lambda) \in D_{2\varepsilon_1}$, $g_0(\lambda) \in D_{1\varepsilon_2}$ satisfy the minimality condition (\ref{nminimal}) and let the functions $h_f(f_0,g_0)$ and $h_g(f_0,g_0)$, determined by the formulas (\ref{nhf}), (\ref{nhg}), be bounded.
The spectral densities  $f_0(\lambda), g_0(\lambda)$, which give solution to the system of equations (\ref{nhf3})--(\ref{nn2})
are the least favourable spectral densities in the class $D_{2\varepsilon_1}\times D_{1\varepsilon_2}$, if they determine a solution to the optimization problem (\ref{nextrem}).
The function $h^0(e^{i\lambda})$, determined by the formula (\ref{nsphar}), is minimax-robust spectral characteristic of the optimal linear estimate of the functional $A_N\xi$.
\end{thm}

\begin{thm}
Let the spectral density $f(\lambda)$ be known, and let the spectral density $g_0(\lambda) \in D_{1\varepsilon_2}$. Let the function $f(\lambda)+g_0(\lambda)$  satisfy the minimality condition (\ref{nminimal}), and let the function $h_g(f,g_0)$, determined by the formula (\ref{nhg}), be bounded.
The spectral density $g_0(\lambda)$ is the least favourable spectral densities in the class $D_{1\varepsilon_2}$ for the optimal linear estimate of the functional $A_N\xi$,
if it is of the form
\begin{equation*}
g_0(\lambda)= \max\left\{g_1(\lambda), \alpha_2^{-1}\left|A_N(e^{i\lambda})f(\lambda)-
C_N^0(e^{i\lambda})\right|-f(\lambda)\right\}
\end{equation*}
and the functions $f(\lambda), g_0(\lambda)$ determine a solution to the optimization problem (\ref{nextrem}).
The function $h^0(e^{i\lambda})$, determined by the formula (\ref{nsphar}), is minimax-robust spectral characteristic of the optimal linear estimate of the functional $A_N\xi$.
\end{thm}

\begin{thm}
Let the spectral density  $f_0(\lambda) \in D_{2\varepsilon_1}$, let the function $f_0^{-1}(\lambda)$ be integrable, and let the function $h(f_0)$, determined by the formula (\ref{5}), be bounded.
The spectral density  $f_0(\lambda)$ is the least favourable spectral densities in the class $D_{2\varepsilon_1}$ for the optimal linear estimate of the functional $A_N\xi$, based on observations of the sequence $\xi(j)$ at points of time $j\in\mathbb{Z}\backslash \{0,1,\dots,N\}$,
if it satisfies the relation
\begin{equation*}
\left|C_N^0(e^{i\lambda})\right|^2=(f_0(\lambda))^2(f_0(\lambda)-f_1(\lambda))\alpha_1
\end{equation*}
and $f_0(\lambda)$ determine a solution to the optimization problem (\ref{extrem}).
The function $h^0(e^{i\lambda})$, determined by the formula (\ref{5}), is minimax-robust spectral characteristic of the optimal linear estimate of the functional $A_N\xi$.
\end{thm}

\subsection{Conclusions}

In this section we propose methods of solution of the problem of the mean-square optimal linear estimation of the functional $A_N\xi=\sum\limits_{j=0}^{N}a(j)\xi(j)$ which depends on the unknown values of a stationary stochastic sequence $\xi(j)$.
Estimates are based on observations of the sequence $\xi(j)+\eta(j)$ at points $j\in\mathbb{Z}\backslash \{0,1,\dots,N\} $, where $\eta(j)$ is an uncorrelated with $\xi(j)$ stationary sequence. We provide formulas for calculating values of the mean square error and the spectral characteristic of the optimal linear estimate of the functional in the case of spectral certainty where the spectral densities $f(\lambda)$ and $g(\lambda)$ of the sequences $\xi(j)$ and  $\eta(j)$ are exactly known.
In the case of spectral uncertainty where the spectral densities $f(\lambda)$ and $g(\lambda)$ are not known, but a set of admissible spectral densities is given, the minimax approach is applied. We obtain formulas that determine the least favourable spectral densities and the minimax spectral characteristics of the optimal linear estimates of the functional $A_N\xi$ for concrete classes of admissible spectral densities.

For the relative results on the mean-square optimal linear interpolation of linear functionals for stationary stochastic sequences and processes based on observations with noise see papers by Moklyachuk~\cite{Moklyachuk:2000} -- \cite{Moklyachuk:2008r}, book by Moklyachuk and Masyutka~\cite{Moklyachuk:2012}.

\section{Conclusion Remarks}

In the proposed paper we describe methods of solution of the problems of the mean-square optimal linear extrapolation and interpolation of linear functionals which depend on the unknown values
of a stationary stochastic sequence $\xi(k)$ based on observations of the sequence $\xi(k)$ as well as observations of the sequence $\xi(k)+\eta(k)$, where $ \eta(k)$ is an  uncorrelated with the sequence $\xi(k)$
stationary stochastic sequence.
The corresponding methods of solution of the problem of the mean-square optimal linear filtering of stationary stochastic sequences are described in the paper by
Luz and Moklyachuk \cite{Luz9}.

Following the Ulf~Grenander~\cite{Grenander} approach to investigation the problem
of optimal linear estimation of the functional which depends on the unknown values of the stationary stochastic continuous parameter process we consider the problem as a two-person zero-sum game in which the first
player chooses a stationary stochastic sequence $ \xi (j) $ from the class $ \Xi $ of
stationary stochastic sequences with $E \xi (j)=0 $ and $ E|\xi (j)|^2=1 $ which
maximizes the value of the mean square error of estimate. The second player is looking for
an estimate of the linear functional which minimizes the value of the
mean square error. It is show that this game has equilibrium point.
The maximum error gives a one-sided moving average stationary sequence which is
least favourable in the given class of stationary sequences.
 The greatest value
of the error and the least favourable sequence are determined by the
largest eigenvalue and the corresponding eigenvector of the
operator determined by coefficients which determine the functionals.
Note, that this approach can be applied to a specific class of estimation problems.

The second approach to the estimation problems we applied is based on the
Kolmogorov \cite{Kolmogorov} Hilbert space projection method which we apply in the case of spectral certainty, where the spectral
densities of the sequences $\xi(n)$ and $\eta(n)$ are exactly known, and the convex optimization method proposed by
Franke~\cite{Franke1984,Franke1985} which is applied in the case of spectral uncertainty, where the spectral densities of
the sequences are not exactly known, but, instead, a set of admissible spectral densities is
given.
Formulas for calculation the mean-square errors and the spectral characteristics of the optimal estimates of functionals
are derived in the case of spectral certainty.
In the case of spectral uncertainty the minimax-robust estimation method is applied. Formulas that determine the least favourable spectral
densities and the minimax-robust spectral characteristics of the optimal linear estimates of the functionals are derived.

In the papers by   Moklyachuk \cite{Moklyachuk:1981a} -- \cite{Moklyachuk:2008} problems of extrapolation, interpolation and filtering
for stationary processes and sequences were studied.
The corresponding problems for vector-valued stationary sequences and processes were investigated by Moklyachuk and Masyutka~\cite{Moklyachuk:2006} -- \cite{Moklyachuk:2012}.  In the articles by Dubovets'ka and Moklyachuk \cite{Dubovetska1} - \cite{Dubovetska8}  and in the book  by Golichenko and Moklyachuk \cite{Golichenko} the minimax estimation problems were investigated for another generalization of stationary processes --  periodically correlated stochastic sequences and stochastic processes.
Luz and Moklyachuk \cite{Luz1} -- \cite{Luz8}, \cite{Luz6} investigated the classical and minimax extrapolation, interpolation and filtering problems for sequences and processes with $n$th stationary increments.
Investigation of the mean-square optimal linear estimation problems for  functionals of stationary stochastic sequences and processes with missing observations is started in the papers by Moklyachuk and Sidei~ \cite{Sidei}, \cite{Sidei2}.
The minimax estimation problems for functionals of a generalization of stationary processes --  harmonizable stable stochastic sequences and processes is started in the papers by
Moklyachuk and Ostapenko~\cite{Ostapenko},~\cite{Ostapenko2}. For the results for functionals of random fields see the book by
Moklyachuk and Shchestyuk~\cite{Shchestyuk}.

\end{document}